\documentclass{amsart}
\usepackage[latin1]{inputenc}
\usepackage[T1]{fontenc}
\usepackage{lmodern}
\usepackage[english]{babel}
\usepackage{microtype}

\usepackage{amsmath,amssymb,amsfonts,amsthm}
\usepackage{mathtools,accents}
\usepackage{mathrsfs}
\usepackage{aliascnt}
\usepackage{braket}
\usepackage{bm}
\usepackage{lineno}
\usepackage[citecolor=blue,colorlinks]{hyperref}

\usepackage{enumerate}
\usepackage{xcolor}

\allowdisplaybreaks
\newtheorem*{ackno}{Acknowledgements}

\usepackage{aliascnt}

\makeatletter
\def\newaliasedtheorem#1[#2]#3{
  \newaliascnt{#1@alt}{#2}
  \newtheorem{#1}[#1@alt]{#3}
  \expandafter\newcommand\csname #1@altname\endcsname{#3}
}
\makeatother

\theoremstyle{plain}
\newtheorem{theorem}{Theorem}[section]
\newaliasedtheorem{lemma}[theorem]{Lemma}
\newaliasedtheorem{prop}[theorem]{Proposition}
\newaliasedtheorem{claim}[theorem]{Claim}
\newaliasedtheorem{corollary}[theorem]{Corollary}

\newcommand{\vertiii}[1]{{\left\vert\kern-0.25ex\left\vert\kern-0.25ex\left\vert #1 
    \right\vert\kern-0.25ex\right\vert\kern-0.25ex\right\vert}}
    
\theoremstyle{definition}
\newaliasedtheorem{definition}[theorem]{Definition}
\newaliasedtheorem{example}[theorem]{Example}

\theoremstyle{remark}
\newaliasedtheorem{remark}[theorem]{Remark}

\numberwithin{equation}{section}

\def\Im{\textrm{Im}}
\def\Re{\textrm{Re}} 
\def\11{{\rm 1~\hspace{-1.4ex}l} }
\def\R{\mathbb R}

\def\Z{\mathbb Z}
\def\N{\mathbb N}
\def\E{\mathbb E}
\def\T{\mathbb T}

\newcommand{\dd}{\mathrm{d}}
\newcommand{\supp}{\ensuremath{\mathrm{supp}\,}}

\newcommand{\D}{\nabla}
\newcommand{\DYe}{\D Y_\eps}
\newcommand{\Wick}[1]{\mathbf{:}#1\mathbf{:}}
\newcommand{\wDY}{\Wick{\D Y^2}}

\newcommand{\wDYe}{\Wick{\D Y^2_\eps}}

\newcommand{\tDY}{\widetilde{\wDY}}

\newcommand{\tDYe}{\widetilde{\wDYe}}

\newcommand{\w}{\langle x\rangle}

\newcommand{\eps}{\varepsilon}

\newcommand{\Sw}{\ensuremath{\mathcal{S}}}

\newcommand{\ii}{\ensuremath{\imath}}

\title
[GWP of the NLS with multiplicative spatial white noise on $\R^2$]
{Global well-posedness of the 2d nonlinear Schr\"odinger equation with multiplicative spatial white noise on the full space}

\author[A. Debussche, R. Liu, N. Tzvetkov \and N. Visciglia]
{A. Debussche, R. Liu, N. Tzvetkov \and N. Visciglia}

\address{A. Debussche,
 Univ Rennes, CNRS, IRMAR - UMR 6625, F-
35000 Rennes, France}
\email{arnaud.debussche@ens-rennes.fr}

\address{R. Liu, School of Mathematics, The University of Edinburgh and The Maxwell Institute for the Mathematical Sciences, James Clerk Maxwell Building, The King's Buildings, Peter Guthrie Tait Road, Edinburgh, EH9 3FD, United Kingdom}
\email{ruoyuan.liu@ed.ac.uk}

\address{N.~Tzvetkov, Ecole Normale Sup\'erieure de Lyon, UMPA, UMR CNRS-ENSL 5669, 46, all\'ee d'Italie, 69364-Lyon Cedex 07, France}
\email{nikolay.tzvetkov@ens-lyon.fr}

\address{N.~Visciglia, Dipartimento di Matematica, Universit\`a di Pisa, Largo Bruno Pontecorvo, 5, 56100 Pisa, Italy}
\email{nicola.visciglia@unipi.it}


\begin{document}

\maketitle

\begin{abstract}
We consider the nonlinear Schr\"odinger equation with multiplicative spatial white noise and an arbitrary polynomial nonlinearity on the two-dimensional full space domain. We prove global well-posedness by using a gauge-transform introduced by Hairer and Labb\'e (2015) and constructing the solution as a limit of solutions to a family of approximating equations. This paper extends a previous result by Debussche and Martin (2019) with a sub-quadratic nonlinearity.
\end{abstract}

\section{Introduction}
\subsection{NLS with multiplicative space white noise}
We study the following Cauchy problem on $\R^2$ for the stochastic defocusing nonlinear Schr\"odinger equation (NLS)
\begin{equation}
\label{eq:NLS}
\ii \partial_t u =\Delta u+ u\xi-\lambda u|u|^{p} , \quad u(0)=u_0\,,
\end{equation} 
where $p>0,\,\lambda>0$ are parameters and $\xi\in \mathcal{S}'(\R^2)$ stands for white noise in space. The unknown $u$ is a complex valued process
and $u_0$ is a randomized initial datum (more precise assumptions are given below).

One can view the NLS \eqref{eq:NLS} as a stochastic version of the deterministic nonlinear Schr\"odinger equation with a power type nonlinearity, which has been studied extensively in recent decades. See \cite{BG80, Caz79, Caz03, Kat87} and the references therein. On the other hand, if we ignore the nonlinearity, \eqref{eq:NLS} can be viewed as the dispersive Anderson model, which is the dispersive counterpart of the well-studied parabolic Anderson model (see, for example, \cite{HL15, HL18}), i.e.~with $\ii \partial_t u$ replaced by $\partial_t u$.

The NLS \eqref{eq:NLS} was first considered by the first author and Weber in \cite{DW18} on a periodic domain $\T^2 = (\R / \Z)^2$. To deal with the ill-defined nature of the term $u \xi$, they used the gauge transform $v = e^{Y} u$ where $Y$ solves $\Delta Y = \xi$. This gauge transform, which resembles the so-called Doss-Sussmann transformation in \cite{Doss, Suss}, was first introduced by Hairer and Labb\'e in the context of the parabolic Anderson model on $\R^2$ in \cite{HL15} (the definition for $Y$ is slightly different on $\R^2$). The equation for $v$ now formally reads as
\begin{equation*}
\ii \partial_t v = \Delta v - 2 \D v \cdot \D Y + v \D Y^2 - \lambda e^{-p Y} v |v|^p, \quad
v(0) = e^Y u_0 \,,
\end{equation*}

\noindent
which is easier to solve since the most singular term is canceled. Still, the term $\D Y$ is merely a distribution, so that $\D Y^2$ is replaced by a Wick product $\wDY$ in \cite{HL15, DW18}. See the next subsection for more detailed explanations.

In \cite{DW18}, the first author and Weber showed global well-posedness of the following gauge-transformed NLS on $\T^2$ with a (sub-)cubic nonlinearity (i.e.~with $p \leq 2$):
\begin{equation}
\label{eq:NLSv1}
\ii \partial_t v = \Delta v - 2 \D v \cdot \D Y + v \, \wDY - \lambda e^{-p Y} v |v|^p, \quad
v(0) = e^Y u_0\,.
\end{equation}

\noindent
The main strategy is to consider a mollified noise $\xi_\eps$ and a smoothed process $Y_\eps = \Delta^{-1} \xi_\eps$, and then construct the solution $v$ as a limit of $v_\eps$ in probability, which is the solution of the following smoothed equation:
\begin{equation*}
\ii \partial_t v_\eps = \Delta v_\eps - 2 \D v_\eps \cdot \D Y_\eps + v_\eps \, \wDYe - \lambda e^{-p Y_\eps} v_\eps |v_\eps|^p, \quad
v_\eps (0) = e^Y u_0\,.
\end{equation*}

\noindent
The key ingredient for the convergence of $v_\eps$ is a suitable $H^2$ a-priori bound for $v_\eps$ with a logarithmic loss in $\eps$.

Later on, the third author and the fourth author \cite{TzV1} improved the result in \cite{DW18} by extending global well-posedness of the gauge-transformed NLS \eqref{eq:NLSv1} on $\T^2$ to the larger range $p \leq 3$. Specifically, they introduced modified energies that allow them to control the $H^2$ a-priori bound of $v_\eps$ for a larger range of $p$. In a subsequent paper \cite{TzV2}, the third author and the fourth author improved their global well-posedness result by covering all $p > 0$. In particular, they exploited the time averaging effect for dispersive equations and established Strichartz estimates to obtain the $H^2$ a-priori bound of $v_\eps$ for the whole range $p > 0$. Moreover, in \cite{TzV1, TzV2}, the authors improved \cite{DW18} by proving almost sure convergence of $v_\eps$ to $v$ instead of convergence in probability.

We now turn our attention to the NLS \eqref{eq:NLS} on the $\R^2$ setting, which is the main concern in this paper. In this situation, the additional difficulty comes from the logarithmic growth of the white noise. In \cite{DM}, the first author and Martin showed global well-posedness of a gauge-transformed NLS similar to \eqref{eq:NLSv1} (see \eqref{eq:NLSv} below) for $0 < p < 1$. To conquer the issue of the unboundedness of the noise, they used weighted Sobolev and Besov spaces and obtained a weighted $H^2$ a-priori bound for $v_\eps$ with a logarithmic loss in $\eps$. This approach of using the weighted Besov spaces in the study of stochastic PDEs was also used in \cite{HL15, HL18, MW17, OTWZ}. In the case of the NLS \eqref{eq:NLS}, such an approach requires more assumptions on the regularity of the initial datum than those on the $\T^2$ setting.

In this paper, we extend the result in \cite{DM} by including all $p > 0$ cases using an intricate combination of the methods mentioned above: the weighted Sobolev and Besov spaces estimates, the modified energies as in \cite{TzV1}, and the dispersive effect as in \cite{TzV2}. In addition, we are able to prove a stronger convergence result of $v_\eps$ to $v$ (almost sure convergence) than that in \cite{DM} (convergence in probability). This is obtained thanks to the estimates on the white noise on $\R^2$ and Wick products which are of independent interest. See the next subsection for a more detailed set-up and the main result of this paper.

\subsection{Set-up and the main result}
Let us recall that, given a probability space $(\Omega, \mathcal F,\mathbb P)$, a real valued white noise on $\R^2$  is a random variable $\xi\;:\; \Omega \to \mathcal D'(\R^2)$ such that for each $f\in \mathcal D(\R^2)$, $(\xi, f)$ is a real valued centered Gaussian random variable such that $\E((\xi,f)^2)=\|f\|_{L^2}^2$. 
Along the rest of the paper, $(\cdot,\cdot)$ will denote the duality $\mathcal D(\R^2)$, $\mathcal D'(\R^2)$
and hence  in the special case of classical functions, this is the usual $L^2$ scalar product.
Classically, $f\to (\xi,f)$ 
can be extended uniquely from $\mathcal D(\R^2)$ to $L^2(\R^2)$, and $(\xi,f)$ is a real valued centered Gaussian random variable such that $\E((\xi,f)^2)=\|f\|_{L^2}^2$ for all $f\in L^2(\R^2)$.

We proceed as in \cite{HL15} and use a truncated Green's function $G\in C^\infty(\R^2\backslash\{0\})$ that satisfies $\supp G \subseteq B(0,1)$ (the unit ball around 0) and $G(x)=-\frac{1}{2\pi} \log|x|$ for $|x|$ small enough, so that $Y:=G\ast \xi$ solves
\begin{align}
\label{eq:PoissonY}
\Delta Y=\xi+\varphi\ast \xi
\end{align} 
for some $\varphi\in C^\infty_c(\R^2)$. 
%
%
Hence, we introduce the new variable
\begin{equation}\label{gaude}
v=e^{Y} u
\end{equation} which converts \eqref{eq:NLS} into the following gauge-transformed NLS for $v$:
\begin{align*} 
\ii \partial_t v=\Delta v+v(\,\D Y^2-\varphi\ast \xi)-2\D v \cdot \D Y-\lambda v|v|^{p} e^{-p Y} ,\quad v(0)=v_0:=e^{Y} u_0\,.
\end{align*}
The term $\D Y^2$ is ill-defined as a square of a distribution, but can be replaced by a meaningful object $\wDY$, which is essentially the Wick product of $\D Y$ with itself.
As in \cite{HL15}, we introduce the random variable
\begin{align}
\label{eq:wDYdef}
\wDY  = \int_{\R^2} \int_{\R^2} \nabla G(\cdot -z_1) \nabla G(\cdot -z_2) \boldsymbol{\xi}(\dd z_1)\boldsymbol{\xi}(\dd z_2) 
\end{align} 
with $\boldsymbol{\xi}$ denoting the Gaussian stochastic measure on $\R^2$ induced by the white noise $\xi$ (see \cite[page 95-99]{Janson}). The relation \eqref{eq:wDYdef} should be read in the \emph{distributional} sense, i.e.~for $\phi \in \Sw(\R^2)$ we have
\begin{align*}
\wDY (\phi)  = \int_{\R^2} \int_{\R^2} \Big(\int_{\R^2} \, \phi(x)\nabla G(x -z_1) \nabla G(x -z_2) \dd x\Big) \boldsymbol{\xi}(\dd z_1)\boldsymbol{\xi}(\dd z_2),
\end{align*}
so that $\wDY$ is only defined (almost surely) as a distribution. Recall that for $f_1,\,f_2\in L^2(\R^2)$ we have the following identity for $X_1:=\int_{\R^2} f_1(z_1) \boldsymbol{\xi}(\dd z_1),\,X_2:=\int_{\R^2} f_2(z_2) \boldsymbol{\xi}(\dd z_2)$ (see \cite[Theorem 7.26]{Janson})
\begin{align}
\label{eq:WickFactorization}
\Wick{X_1\cdot X_2}=\int_{\R^2} \int_{\R^2} f_1(z_1)  f_2(z_2)\boldsymbol{\xi}(\dd z_1) \boldsymbol{\xi}(\dd z_2),
\end{align}
where $\Wick{X_1\cdot X_2}$ denotes the Wick product between the Gaussian random variables $X_1,\,X_2$. Note that the above integral is a multiple Wiener-Ito integral (see \cite{Nua06}). 
From this perspective, the definition $\wDY$ can be read as a Wick product of the distribution $\D Y$ with itself.
For an introduction to Wick calculus let us refer to \cite{Hida, Janson, Nua06}. Thus, we shall focus on the following equation
\begin{equation}
\label{eq:NLSv}
\ii \partial_t v=\Delta v+v\,\tDY -2\D v \cdot \D Y-\lambda v|v|^{p} e^{-p Y},\quad v(0)=v_0\,,
\end{equation}

\noindent
where 
\begin{equation}
\label{tDY2}
\tDY = \wDY - \varphi \ast \xi.
\end{equation}

In order to construct a solution to \eqref{eq:NLSv} we consider an approximation $v_\eps$ which solves
\begin{equation}
\label{eq:NLSve}
\ii \partial_t v_\eps=\Delta v_\eps+v_\eps\tDYe-2\D v_\eps \cdot \D Y_\eps-\lambda v_\eps  |v_\eps|^p e^{-pY_\eps}, \quad
v_\eps(0)=v_0\,.
\end{equation}
Here, $Y_\eps$ is defined as $$Y_\eps = \rho_\eps \ast Y = \rho_\eps \ast G \ast \xi = G \ast \xi_\eps,$$ 
where $\rho_\eps (x) = \eps^{-2} \rho(\eps^{-1} x)$, $\rho \in C_c^{\infty} (B(0,1))$, $\rho \geq 0$, $\int_{\R^2} \rho = 1$, and $\xi_\eps=\rho_\eps \ast \xi$ is a mollification of the considered noise. Also,
\begin{equation}
\label{tDY2e}
\tDYe=\wDYe-\varphi\ast \xi_\eps,
\end{equation}
where the Wick product $\wDYe$
is defined as follows:
\begin{equation}\label{WICK:}\wDYe(x) = \DYe^2(x)-c_\eps, \quad
c_\eps= \E\left( |\nabla Y_\eps|^2\right) = \|\rho_\eps \ast \nabla G\|_{L^2}^2.
\end{equation}
We show in this paper that $Y_\eps$ converges to $Y$, $\nabla Y_\eps$ converges to $\nabla Y$, and $\wDYe$ converges to $\wDY$ almost surely in certain function spaces (see Proposition \ref{prop:conv}).

In the sequel we shall use weighted Sobolev spaces, where the details are in Section \ref{prelimina}.
For the moment, we can consider the following equivalent norm on the space $H^s_{\delta}(\R^2)$:
$$\|\w^\delta \varphi\|_{H^s(\R^2)}, \quad s\geq 0, \quad \delta\in \R,$$
which reduces the weighted Sobolev spaces to the usual unweighted Sobolev spaces once the weight is pulled inside.

\medskip
We now state the following result regarding global well-posedness of the equation \eqref{eq:NLSve} for $v_\eps$.

\begin{theorem}\label{thmreg} There exists a full measure event $\Sigma\subset \Omega$ such that for every $\omega\in \Sigma$ and every $\eps \in (0, \frac 12)$, the following property holds.
For any $p \geq 1$, $\delta_0>0$, $s\in (1,2)$, and $\delta > 0$, there exists $\bar \delta > 0$ such that the Cauchy problem \eqref{eq:NLSve} with
$v_0 \in H^2_{\delta_0}(\R^2)$ admits one unique global solution $$v_\varepsilon(t,x)\in L_{\textup{loc}}^\infty((0, \infty); H^2_{-\delta}(\R^2))\cap 
{\mathcal C}([0, \infty); H^s_{\bar \delta}(\R^2)).$$
Moreover, for every $T>0$ and $\delta>0$ there exists constants $C, C(\omega) > 0$ independent of $\eps$ such that the following bound holds:
$$\|v_{\varepsilon}(t,x)\|_{L^\infty((0,T);H^2_{-\delta}(\R^2))}\leq C(\omega) |\ln \varepsilon|^{C}.$$  
\end{theorem}

The existence and uniqueness of solution to \eqref{eq:NLSve} is not obvious since the coefficients involved in the linear propagator associated with \eqref{eq:NLSve}
are smooth but unbounded. At the best 
of our knowledge, because of the lack of decay at the spatial infinity of the derivatives of $\xi_\eps$ the classical Strichartz estimates are not available in this framework, and so we cannot apply a classical contraction argument.
In fact, along the paper we shall establish some weighted Strichartz estimates that will allow us to perform a compactness argument as in \cite{DM} and to deduce the existence
of solutions to \eqref{eq:NLSve} as the limit of solutions $v_{\varepsilon, n}$ to the following further regularized equation at $\varepsilon>0$ fixed:
\begin{equation}\label{NLSfurthersmoothing}
\begin{split}
\ii \partial_t v_{\eps, n}&=\Delta v_{\eps, n}+v_{\eps, n} \theta_n\tDYe 
-2\D v_{\eps, n} \cdot \D (\theta_n Y_\eps )
-\lambda v_{\eps, n}  |v_{\eps, n}|^p e^{-p\theta_n Y_\eps}, \\
\quad v_{\eps, n}(0)&=v_0\,,
\end{split}
\end{equation}
where $\theta_n(x)=\theta (\frac xn )$ and $\theta\in C^\infty_0(\R^2)$, $\theta\geq 0$, $\theta(x)=1$ when $|x| \leq 1$. To see that \eqref{NLSfurthersmoothing} is globally well-posed, we let $u_{\eps, n} = e^{- \theta_n Y_\eps} v_{\eps, n}$ and consider the following equation for $u_{\eps, n}$:
\begin{equation}
\label{NLSuen}
\begin{split}
\ii \partial_t u_{\eps, n}&=\Delta u_{\eps, n}+u_{\eps, n} ( \theta_n\tDYe 
- \D (\theta_n Y_\eps )^2 + \Delta (\theta_n Y_\eps) ) -\lambda u_{\eps, n}  |u_{\eps, n}|^p, \\
u_{\eps, n}(0) &= e^{-\theta_n Y_\eps} v_0\,.
\end{split}
\end{equation}
 
\noindent
Note that $\theta_n Y_\eps$ is a Schwartz function, so that $e^{-\theta_n Y_\eps} v_0 \in H^2 (\R^2)$ given $v_0 \in H^2 (\R^2) \subset H^2_{\delta_0} (\R^2)$. Since the equation \eqref{NLSuen} contains only bounded and smooth terms, by classical results as in \cite{Caz03, GV79, Kat87}, there exists a unique solution $u_{\eps, n}$ to \eqref{NLSuen} in $\mathcal{C} ([0, \infty); H^2 (\R^2))$. This shows that there exists a unique solution $v_{\eps, n}$ to \eqref{NLSfurthersmoothing} in $\mathcal{C} ([0, \infty); H^2 (\R^2))$.

Once we establish a weighted $H^2$ a-priori bound (independent of $n$) for $v_{\eps, n}$, we can use a similar argument in \cite{DM} to prove Theorem \ref{thmreg}. See Section \ref{sec:thm} for a more detailed explanation.

\medskip
Next, we describe the behavior of the solutions $v_\varepsilon(t,x)$ in the limit $\varepsilon\rightarrow 0$.
It is remarkable that this result can be also interpreted in terms of the convergence, up to a phase shift, of $u_\varepsilon(t,x)$ solutions to the smoothed version of \eqref{eq:NLS}
\begin{equation}
\label{eq:NLSsmoothed}
\ii \partial_t u_\varepsilon =\Delta u_\varepsilon+ u_\varepsilon\xi_\varepsilon -\lambda u_\varepsilon|u_\varepsilon|^{p}, \quad u(0)=u_0\,.
\end{equation} 
It is worth mentioning that it is not obvious to prove directly the existence and uniqueness of solutions to the smoothed equation above. Nevertheless, by direct computation, one can show that 
\begin{equation}
u_\varepsilon=e^{- \ii c_\eps t} e^{-Y_\varepsilon} v_\varepsilon
\end{equation}
solves \eqref{eq:NLSsmoothed} provided that $v_\varepsilon$ is given by Theorem \ref{thmreg}.


We are now ready to state the main result of this paper.

\begin{theorem}\label{mainthm}
There exists a full measure event $\Sigma\subset \Omega$ such that for every $\omega\in \Sigma$, the following property holds. Let $\delta_0 > 0$ and $v_0 \in H^2_{\delta_0}$. For any $p \geq 1$, $s \in (1, 2)$, and $\delta > 0$, there exists $v \in \mathcal{C} ([0, \infty); H^s_{\bar \delta} (\R^2))$ for some $\bar \delta > 0,$ such that the following convergence holds:
\begin{equation}
\|v_\varepsilon(t,x)-v(t,x)\|_{\mathcal{C} ([0,T); H^s_{\bar \delta}(\R^2))}\overset{\varepsilon\rightarrow 0}\longrightarrow 0,
\end{equation}
where $v_\eps$ is given by Theorem \ref{thmreg} with $v_\eps (0) = v_0$.
In particular, $u_\varepsilon= e^{-\ii c_\eps t} e^{-Y_\varepsilon} v_\varepsilon$ solves \eqref{eq:NLSsmoothed} and
$$\|e^{\ii c_\varepsilon t} e^{Y_\varepsilon} u_\varepsilon(t,x)-v(t,x)\|_{\mathcal{C} ([0,T); H^s_{\bar \delta}(\R^2))}\overset{\varepsilon\rightarrow 0}\longrightarrow 0,$$
where $c_\varepsilon\sim |\ln \varepsilon|$ is the constant in \eqref{WICK:}.
Moreover, $v$ is the unique global solution to \eqref{eq:NLSv} in $\mathcal{C}([0, \infty); H^s_{\bar \delta}(\R^2))$.
\end{theorem}

Theorem \ref{mainthm} is an extension of previous result proved in \cite{DM} in the case $0 < p < 1$. In particular, we cover the relevant case of cubic nonlinearity.
Notice also that our convergence occurs almost surely, which is stronger than the result in \cite{DM} where the convergence is in probability.

\medskip
We conclude this subsection by stating several remarks.
\begin{remark} \rm
Throughout the whole paper, we restrict our attention to the case $\lambda > 0$ of the NLS \eqref{eq:NLS}, which refers to the defocusing case. For the focusing case (i.e.~$\lambda < 0$) of \eqref{eq:NLS}, Theorem \ref{mainthm} also holds for $0 < p < 2$. For $\lambda < 0$ and $p \geq 2$, we need to impose a smallness assumption on the initial data $\| v_0 \|_{H^1_{\delta_0}}$ in order to obtain Theorem \ref{mainthm}. Indeed, the only place that requires a different proof for this case is Proposition \ref{prop:L2H1}. See Remark \ref{rmk:foc} for details.
\end{remark}

\begin{remark} \rm
For the NLS \eqref{eq:NLS} in higher dimensions, it is not clear whether the approach in this paper based on a gauge-transform works in showing global well-posedness. The main challenge lies in the fact that the spatial white noise is too rough for our approach when $d \geq 3$. Another challenge is the weaker smoothing properties of the Schr\"odinger equation in higher dimensions. One can compare the situation with the parabolic setting in \cite{HL18}, where the authors used the theory of regularity structures due to Hairer \cite{Hai14}.
\end{remark}

\begin{remark} \rm
The authors in \cite{GUZ, MZ, Ugur} introduced another approach to the study of the NLS \eqref{eq:NLS} with $p \le 2$. Their method is based on the realization of the (formal) Anderson Hamiltonian $H = \Delta + \xi$ as a self-adjoint operator on the $L^2$ space. Specifically, \cite{GUZ} considered the equation on the torus, \cite{MZ} considered a compact manifold, and \cite{Ugur} considered the full space. In their settings, the initial data $u_0$ needs to belong to the domain of $H$. One can compare the initial condition in \cite{GUZ, Ugur} and that in this paper and in \cite{DM, DW18, TzV1, TzV2}, where the initial data is chosen to have a specific structure $e^{-Y} v_0$ with $v_0$ belonging to a weighted $H^2$ space. For more discussions on the Anderson Hamiltonian, see \cite{AC, BDM22, CZ, Lab19, Mou22}. 
\end{remark}

\subsection{Notations}
We denote by ${\mathcal P}(a,b,\dots)$ a polynomial function depending on the $a,b, \dots$.
The polynomial can change from line to line along the estimates.
We denote by $N$ dyadic numbers larger than or equal to $\frac 12$ and by $\Delta_N$ the corresponding Littelwood-Paley partition.
We shall use in special situations a second Littlewood-Paley partition $\tilde \Delta_N$ such that $\tilde \Delta_N\circ \Delta_N=\Delta_N$.
All the functional spaces that we shall use are based on $\R^2$. We denote by $C>0$ a deterministic constant that can change from line to line and $C(\omega)>0$
a stochastic constant which is finite almost surely.
We denote by $(\cdot, \cdot)$ the $L^2$ scalar product as well as the duality in ${\mathcal D(\R^2)}, {\mathcal D}'(\R^2)$.

\subsection{Organization of the paper}
This paper is organized as follows. In Section \ref{prelimina}, we recall the definitions of some useful functional spaces and their properties. In Section \ref{subsec:Noise}, we discuss properties of the process $Y$ and its related stochastic objects. In Section \ref{sec:lin} and Section \ref{sec:nonlin}, we establish some useful linear and nonlinear estimates for a generalized equation. In Section \ref{sec:energy}, we recall definitions of modified energies and provide estimates for them. In Section \ref{sec:H2}, we prove a key $H^2$ a-priori bound. Lastly, in Section \ref{sec:thm}, we prove the main result of this paper: Theorem \ref{mainthm}.

\section{Functional spaces and preliminary facts}\label{prelimina}

\subsection{Functional spaces}
Given $p\in [1,\infty],\,\mu\in \R$ we introduce respectively the weighted Lebesgue and Sobolev spaces as follows:
\begin{align*}
\|f\|_{L^p_\mu}=\Big(\int_{\R^2} \w^{\mu p}|f|^p dx\Big)^\frac 1p , \quad \w=\sqrt {1+|x|^2}
\end{align*}
and
\begin{align*}
\|f\|_{W^{1,p}_\mu}=\|f\|_{L^p_\mu}+\|\nabla f\|_{L^p_\mu}
\end{align*}
with the usual interpretation if $p=\infty$. If $\mu=0$ we simply write $L^p=L^p_0$ and  $W^{1,p}=W^{1,p}_0$.

Along the paper we shall make extensively use
of the Littlewood-Paley multipliers $\Delta_N$, namely
\begin{equation}\label{idenLP}
\hbox{Id}=\sum_{{N-{\rm dyadic}}}  \Delta_N
\end{equation}
where
$$\Delta_N f=N^2\int_{\R^2} K(N(x-y))f(y) dy, \quad N\geq 1$$
with $\hat K\in \mathcal S(\R^2)$ such that $\supp \hat K(\xi)\subset \{\frac 12 \leq |\xi|\leq 2\}$
and
$$\Delta_\frac 12 f=\int_{\R^2} L(x-y)f(y) dy$$
with $\hat L\in \mathcal S(\R^2)$ such that $\supp \hat L(\xi)\subset \{|\xi|<1\}$. From now on, the dyadic numbers are assumed to be larger than or equal to $\frac 12$.
%

We can then introduce the weighted inhomogeneous Besov spaces $\mathcal{B}^\alpha_{p,q,\mu}$ as follows:
\begin{equation}\label{besovdef}
\|f\|_{\mathcal{B}^\alpha_{p,q,\mu}} = \Big (\sum_{{N-{\rm dyadic}}} N^{\alpha q} \|\Delta_N f\|_{L^p_\mu}^q\Big)^\frac 1q
\end{equation}
for every $\alpha, \mu\in \R$, $p, q\in [1, \infty]$. Notice that for $\mu=0$ the space $\mathcal{B}^\alpha_{p,q,0}$ reduces to the usual Besov space
$\mathcal{B}^\alpha_{p,q}$. A convenient property of the spaces $\mathcal{B}^\alpha_{p,q,\mu}$ is that the weight can be ``pulled in'', namely we have the equivalent norms:
\begin{align}
\label{eq:PullWeight}
c \|f \w^\mu \|_{\mathcal{B}^{\alpha}_{p,q}}<\|f\|_{\mathcal{B}^{\alpha}_{p,q,\mu}} \leq C \|f \w^\mu \|_{\mathcal{B}^{\alpha}_{p,q}}\,,
\end{align} for suitable $c, C>0$ that depend on $\alpha,\,\mu\in \R$ and $p,q\in [1,\infty]$
(see \cite[Theorem 6.5]{TriebelIII}). Relation \eqref{eq:PullWeight} can be used to translate results from the unweighted spaces to their weighted analogues.

In the case $(p,q)=(2,2)$, the weighted Besov spaces are generalizations of weighted Sobolev spaces:
\begin{equation}\label{sobolev2}
\mathcal{B}^{\alpha}_{2,2, \mu} =H^\alpha_\mu, \quad \hbox{ where }
 \|f\|_{H^\alpha_\mu} =\|\mathcal{F}^{-1}\langle\cdot\rangle^\alpha \mathcal{F}f\|_{L^2_\mu}.
\end{equation}
In the sequel we shall make extensively use of the following obvious continuous embedding
\begin{equation}\label{inclusionHsdelta}
H^{s_1}_{\mu_1}\subset H^{s_2}_{\mu_2}, \quad s_1\geq s_2, \quad \mu_1\geq \mu_2.
\end{equation}
The embedding \eqref{inclusionHsdelta} is compact when $s_1 > s_2$ and $\mu_1 > \mu_2$ (see \cite[Section 4.2.3]{ET}).
We shall also use the notation 
\begin{equation}\label{holder}\mathcal{B}^{\alpha}_{\infty,\infty, \mu} =\mathcal{C}^{\alpha}_\mu.\end{equation}
In the special case $\alpha\in \R_{+}\backslash \{0,1,2,\ldots\}$, the space $\mathcal{C}^\alpha_\mu$ is in turn equivalent to the classical 
weighted H\"older-Zygmund space with the following norm
\begin{equation}\label{eq:ZygmundNorm}
\|f\|_{\mathcal{C}^\alpha_\mu}=\sum_{|k|\leq \lfloor \alpha\rfloor}\sup_{x\in \R^2} \,\w^{\mu}\,|\partial^k f(x)| + \sum_{|k|=\lfloor \alpha \rfloor }\sup_{\substack{x,y\in \R^2\\0<|x-y|\leq 1}}\w^\mu\frac{|\partial^k f(x)-\partial^k f(y)|}{|x-y|^{\alpha -\lfloor \alpha\rfloor}}\,.
\end{equation}
Setting $\mu = 0$ and restricting \eqref{eq:ZygmundNorm} to $x, y \in D$ for some domain $D \subset \R^2$ gives rise to the local H\"older space $\mathcal{C}^{\alpha} (D)$. For the other values of $\alpha$ (in particular for $\alpha<0$) we take \eqref{holder} as a definition of $C^\alpha_\mu$.

\subsection{Some properties of the Littlewood-Paley localization}
We now gather well--known properties of the weighted Besov spaces that will be used along the paper. We first give an elementary, but useful, property of the Littlewood-Paley decomposition in weighted spaces.
\begin{lemma}
Let $\gamma, \delta\geq 0$ and $\gamma_0>0$ be given. Then, there exists $C>0$ such that
\begin{equation}\label{sumupLP}
\sum_{{N-{\rm dyadic}}} N^\gamma \|\Delta_N f\|_{L^2_\delta}\leq C \|f\|_{H^{\gamma+\gamma_0}_{\delta}}.
\end{equation}
\end{lemma}

\begin{proof}
We have by the Cauchy-Schwartz inequality
\begin{equation}
\sum_{{N-{\rm dyadic}}} N^\gamma \|\Delta_N f\|_{L^2_\delta}\leq C
\Big(\sum_{{N-{\rm dyadic}}} N^{2(\gamma+\gamma_0)} \|\Delta_N f\|_{L^2_\delta}^2\Big)^\frac 12
\end{equation}
and hence we conclude by recalling \eqref{besovdef} and \eqref{sobolev2}.
\end{proof}

We shall need the following commutator estimates.
\begin{lemma}
For every $\delta\in (0,1)$ and $p\in [1,\infty]$, there exists $C>0$ such that for every $N$ dyadic, we have
\begin{equation}\label{LPaley}
\|[\Delta_N, \w^{\delta}] f\|_{L^p}\leq CN^{-1} \|f\|_{L^p},
\end{equation}
\begin{equation}\label{LPaley2}\|[\nabla, [\Delta_N,\w^\delta]]f\|_{L^p}\leq CN^{-1}\|f\|_{L^p}.\end{equation}
\end{lemma}
\begin{proof}

It is easy to check
that 
$$[\Delta_N, \w^{\delta}] f= N^2 \int_{\R^2} (\langle y \rangle^\delta-\w^\delta) K(N(x-y))f(y) dy$$
and hence we easily conclude \eqref{LPaley} by the Schur test since
\begin{align*}
\sup_x &\Big(N^2 \int_{\R^2} |\langle y \rangle^\delta-\w^\delta| |K(N(x-y))|dy\Big)\\
&=\sup_y \Big ( N^2 \int_{\R^2} |\langle y \rangle^\delta-\w^\delta| |K(N(x-y))|dx\Big)\\
&\leq C N^{-1}\||x|K(x)\|_{L^1}
\end{align*} 
where we used that 
$|\langle y \rangle^\delta-\w^\delta|\leq C|x-y|$
for $\delta<1$.

Concerning \eqref{LPaley2} we denote by $\tilde K_N(x,y)=N^2[\langle y \rangle^\delta-\w^\delta] K(N(x-y))$
the kernel of the operator $[\Delta_N, \w^{\delta}]$.
Hence, we have
$$[\nabla, [\Delta_N,\w^\delta]]f=\int_{\R^2} (\nabla_x \tilde K_N(x,y)+ \nabla_y \tilde K_N(x,y))f(y) dy$$
and we conclude as above via the Schuur test since
$$\nabla_x \tilde K_N(x,y)+ \nabla_y \tilde K_N(x,y)=N^2[-\nabla_x \w^\delta + \nabla_y \langle y \rangle^\delta]K(N(x-y))$$
and $|-\nabla_x \w^\delta + \nabla_y \langle y\rangle^\delta| \leq C |x-y|$.
\end{proof}

Next, we show a useful property of the Littlewood-Paley multipliers $\Delta_N$ in weighted Sobolev spaces.
\begin{lemma}\label{LPeqHs}
For every $s\in [0,2]$ and $\delta \in (0, 1)$, there exists $C>0$ such that for every $N$ dyadic, we have
\begin{equation}\label{LPtins}\|\Delta_N f\|_{H^s_\delta}\leq C N^{s}\Big (\|\Delta_\frac N2 f\|_{L^{2}_\delta}+ \|\Delta_N f\|_{L^{2}_\delta}
+\|\Delta_{2N} f\|_{L^{2}_\delta}\Big ).
\end{equation}
\end{lemma}
\begin{proof}We split the proof in several cases.

\smallskip \noindent
{\bf Case 1:} $s\in [0,1)$.

\smallskip \noindent
By combining \eqref{eq:PullWeight} and \eqref{LPaley} we get:
\begin{align}
\label{stimrip}
\begin{split}
\|\Delta_N f\|_{H^s_\delta}^2 &\leq C\sum_{{M-{\rm dyadic}}} M^{2s} \|\Delta_M (\w^\delta \Delta_N f)\|_{L^2}^2 \\
&\leq C
\sum_{{M-{\rm dyadic}}} M^{2s} \|\w^\delta \Delta_M (\Delta_N f)\|_{L^2}^2 \\
&\quad + C \sum_{{M-{\rm dyadic}}} M^{2s-2} \|\Delta_N f\|_{L^2}^2 \\
&\leq C \sum_{M=\frac N2, N, 2N} M^{2s} \|\w^\delta \Delta_M f\|_{L^2}^2 + C \|\Delta_N f\|_{L^2}^2,
\end{split}
\end{align}
where we used that $\sum_{M-{\rm dyadic}} M^{2s-2} <\infty$ for $s\in [0,1)$.
Hence, we get
\begin{align}
\label{guidot}
\begin{split}
\|\Delta_N f\|_{H^s_\delta}^2
&\leq C N^{2s}\Big (\|\Delta_\frac N2 f\|_{L^2_\delta}^2 +
\|\Delta_N f\|_{L^2_\delta}^2+\|\Delta_{2N} f\|_{L^2_\delta}^2\Big ) \\
&\quad + C N^{-2s}\|\Delta_N f\|_{H^s_\delta}^2
\end{split}
\end{align}
From \eqref{guidot}, we can deduce \eqref{LPtins} provided that $N>\bar N$, with $\bar N$ choosen in such a way that 
the last term on the r.h.s. can be absorbed on the l.h.s. On the other hand, we have finitely dyadic numbers
$1\leq N\leq \bar N$ and hence the corresponding estimate \eqref{LPtins} for those values of $N$ holds, provided that we choose the multiplicative constant large enough on the r.h.s.

\smallskip \noindent
{\bf Case 2:} $s\in [1, 2)$.

\smallskip \noindent
We start by proving that for $s\in [1, 2]$, there exists $C>0$ such that for every $N$ dyadic:
\begin{equation}\label{equivs>1} 
\|\Delta_N f\|_{H^s_\delta}
 \leq C\Big ( \|\Delta_N f\|_{H^{s-1}_\delta}+ \|\Delta_N \nabla f\|_{H^{s-1}_\delta}\Big).\end{equation}
In fact, by \eqref{eq:PullWeight} we have
\begin{align}
\label{equivalences>1}
\begin{split}
\|\Delta_N f\|_{H^s_\delta} &\leq C \|\w^\delta \Delta_N f\|_{H^s} \\
&\leq C \Big(\|\w^\delta \Delta_N f\|_{L^2} + \|\nabla (\w^\delta \Delta_N f)\|_{H^{s-1}}\Big) \\ 
&\leq C\Big (\|\w^\delta \Delta_N f\|_{L^2} +\| \w^\delta (\Delta_N \nabla f)\|_{H^{s-1}}\\
&\quad +  \|[\nabla, \w^\delta] \Delta_N f\|_{H^{s-1}}\Big) \\ 
&\leq C\Big (\|\w^\delta \Delta_N f\|_{L^2} + \| \w^\delta (\Delta_N \nabla f)\|_{H^{s-1}}+ \|\Delta_N f\|_{H^{s-1}}\Big).
\end{split}
\end{align}
Here, we used the elementary commutator bound 
\begin{equation}\label{commutator2}\|[\nabla, \w^\delta] f\|_{H^\sigma}\leq C \|f\|_{H^\sigma}, \quad \sigma\in [0,1],\end{equation}
which follows from interpolating the $L^2 \to L^2$ bound and the $H^1 \to H^1$ bound of the commutator $[\nabla, \w^\delta]$.

Next, we show \eqref{LPtins} for $s\in [1,2)$. Notice that we have $s-1\in [0,1)$  and hence we can combine \eqref{equivs>1} with the estimate proved
in the first case in order to obtain
\begin{align}
\label{deltatilde2}
\begin{split}
\|\Delta_N f\|_{H^s_\delta} &\leq C N^{s-1}\Big (\|\Delta_\frac N2 f\|_{L^{2}_\delta}+ \|\Delta_N f\|_{L^{2}_\delta}
+\|\Delta_{2N} f\|_{L^{2}_\delta}\Big)
\\
&\quad + C N^{s-1}\Big (\|\Delta_\frac N2 \nabla f\|_{L^2_\delta}+ \|\Delta_N \nabla f\|_{L^2_\delta}
+\|\Delta_{2N} \nabla f\|_{L^{2}_\delta}\Big ).
\end{split}
\end{align}
Then, notice that if we denote by
$\tilde \Delta_M$ another Littlewood-Paley partition such that $\tilde \Delta_M\circ \Delta_M=\Delta_M$, then we get for a generic dyadic $M$
\begin{equation*}\|\Delta_M \nabla f\|_{L^2_\delta}
\leq C\Big( \|\nabla(\w^\delta \Delta_M f)\|_{L^2}+  \|[\w^\delta, \nabla] \Delta_M f\|_{L^2}\Big)
\end{equation*}
and by using \eqref{commutator2} we get
\begin{align*}
 \dots &\leq C\Big( \|\nabla  \tilde \Delta_M(\w^\delta \Delta_M f)\|_{L^2}
+\|\nabla( [\tilde \Delta_M,\w^\delta] \Delta_M f)\|_{L^2}+  \|\Delta_M f\|_{L^2}\Big) \\
&\leq C\Big (M\|\Delta_M f\|_{L^2_\delta}+\| [\tilde \Delta_M,\w^\delta] \Delta_M \nabla f)\|_{L^2}
+ \|[\nabla, [\tilde \Delta_M,\w^\delta]] \Delta_M f\|_{L^2}\Big).
\end{align*}
By using \eqref{LPaley} and \eqref{LPaley2}, we can summarize the estimates above as follows
\begin{equation}\label{deltatilde}
\|\Delta_M \nabla f\|_{L^2_\delta}\leq C M\|\Delta_M f\|_{L^2_\delta}.
\end{equation}
We get the desired conclusion by combining \eqref{deltatilde2} and \eqref{deltatilde}.

\smallskip \noindent
{\bf Case 3:} $s=2$.

\smallskip \noindent
We use \eqref{equivs>1} for $s=2$ and the fact that \eqref{LPtins} has been proved for $s=1$ to obtain
\begin{align*}
\|\Delta_N f\|_{H^2_\delta} &\leq C\Big ( \|\Delta_N f\|_{H^{1}_\delta}+ \|\Delta_N \nabla f\|_{H^{1}_\delta}\Big) \\
&\leq C N\Big (\|\Delta_\frac N2 f\|_{L^{2}_\delta}+ \|\Delta_N f\|_{L^{2}_\delta}
+\|\Delta_{2N} f\|_{L^{2}_\delta}\Big) \\
&\quad +C N\Big (\|\Delta_\frac N2 \nabla f\|_{L^2_\delta}+ \|\Delta_N \nabla f\|_{L^2_\delta}
+\|\Delta_{2N} \nabla f\|_{L^{2}_\delta}\Big ).
\end{align*}
We conclude by using \eqref{deltatilde}.
\end{proof}

We shall also use the following estimates.
\begin{lemma}
For every $s\in (0,1)$, $s_0 \in \R$, and $\delta>0$, there exists $C>0$ such that for every $N$ dyadic, we have
\begin{equation}\label{LPtinsnewnew}\|\Delta_N f\|_{H^{s}_\delta}\leq C N^{s_0}\Big (\|\tilde \Delta_\frac N2 f\|_{H^{s-s_0}_\delta}+ 
\|\tilde \Delta_N f\|_{H^{s-s_0}_\delta}+\|\tilde \Delta_{2N} f\|_{H^{s-s_0}_\delta}\Big ),
\end{equation}
\begin{equation}\label{LPtinsnew}\|\Delta_N f\|_{H^{1+s}_\delta}\leq C N^{s_0}\Big (\|\tilde \Delta_\frac N2 f\|_{H^{1+s-s_0}_\delta}+ 
\|\tilde \Delta_N f\|_{H^{1+s-s_0}_\delta}+\|\tilde \Delta_{2N} f\|_{H^{1+s-s_0}_\delta}\Big ),
\end{equation}
where $\tilde \Delta_N$ is another Littlewood-Paley decomposition.
\end{lemma}
\begin{proof}
We first prove \eqref{LPtinsnewnew} and notice that by \eqref{stimrip}
\begin{equation*}
\|\Delta_N f\|_{H^s_\delta}^2
\leq C \sum_{M=\frac N2, N, 2N} M^{2s} \|\w^\delta \Delta_M f\|_{L^2}^2 + C \|\Delta_N f\|_{L^2}^2
\end{equation*}
that can be continued as follows
\begin{align*}
\dots &\leq C N^{2s_0}\sum_{M=\frac N2, N, 2N} M^{2s-2s_0} \|\w^\delta \Delta_M f\|_{L^2}^2 + C \|\Delta_N f\|_{L^2}^2 \\ 
&\leq C N^{2s_0}\sum_{M=\frac N2, N, 2N} M^{2s-2s_0} \|\Delta_M(\w^\delta \tilde \Delta_M f)\|_{L^2}^2 \\ 
&\quad +C N^{2s_0}\sum_{M=\frac N2, N, 2N} M^{2s-2s_0-2} \|\tilde \Delta_M f\|_{L^2}^2 
+ C \|\Delta_N f\|_{L^2}^2,
\end{align*}
where we used the estimate \eqref{LPaley} and as usual $\tilde \Delta_M$ is another Littlewood-Paley partition such that
$\tilde \Delta_M\circ \Delta_M=\Delta_M$. The proof follows since exactly as in the proof of \eqref{LPtins},
the term $\|\Delta_N f\|_{L^2}$ on the r.h.s. can be absorbed by $\|\Delta_N f\|_{H^s_\delta}^2$.

Next, we focus on \eqref{LPtinsnew} and we recall that by  \eqref{equivs>1}
\begin{equation}\label{montec}
\|\Delta_N f\|_{H^{1+s}_\delta}
 \leq C\Big ( \|\Delta_N f\|_{H^{s}_\delta}+ \|\Delta_N \nabla f\|_{H^{s}_\delta}\Big).\end{equation}
Arguing as above, we get
\begin{equation*}
\|\Delta_N \psi\|_{H^s_\delta}^2
\leq C N^{2s_0}\sum_{M=\frac N2, N, 2N} \|\tilde \Delta_M\psi\|_{H^{s-s_0}_\delta}^2 
+ C \|\Delta_N \psi\|_{L^2}^2
\end{equation*}
and by choosing $\psi=f$ and $\psi = \nabla f$, we get from \eqref{montec}
\begin{align*}
\|\Delta_N f\|_{H^{1+s}_\delta}^2 &\leq C N^{2s_0}\sum_{M=\frac N2, N, 2N} \|\tilde \Delta_M f\|_{H^{s-s_0}_\delta}^2 
+ C \|\Delta_N f\|_{L^2}^2 \\
&\quad +C N^{2s_0}\sum_{M=\frac N2, N, 2N} \|\tilde \Delta_M\nabla f\|_{H^{s-s_0}_\delta}^2 
+ C \|\Delta_N \nabla f\|_{L^2}^2.
\end{align*}
As in the proof of \eqref{LPtins}, we can absorb $\|\Delta_N f\|_{L^2}^2$ and $\|\Delta_N \nabla f\|_{L^2}^2$ on the l.h.s.
and hence we obtain
\begin{align*}
\|\Delta_N f\|_{H^{1+s}_\delta}^2 
&\leq C N^{2s_0}\sum_{M=\frac N2, N, 2N} \|\tilde \Delta_M f\|_{H^{s-s_0}_\delta}^2 \\
&\quad +C N^{2s_0}\sum_{M=\frac N2, N, 2N} \|\tilde \Delta_M\nabla f\|_{H^{s-s_0}_\delta}^2 \\
&\leq C N^{2s_0}\sum_{M=\frac N2, N, 2N} \|\tilde \Delta_M f\|_{H^{s-s_0}_\delta}^2 
+C N^{2s} \sum_{M=\frac N2, N, 2N} \|\tilde \Delta_M \nabla f\|_{L^2_\delta}^2 \\
&\leq C N^{2s_0}\sum_{M=\frac N2, N, 2N} \|\tilde \Delta_M f\|_{H^{s-s_0}_\delta}^2 \\
&\quad +C N^{2s_0} \sum_{M=\frac N2, N, 2N} M^{2+2s-2s_0}\|\tilde \Delta_M f\|_{L^2_\delta}^2, 
\end{align*}
where we used \eqref{deltatilde}. The conclusion follows.
\end{proof}

We close this subsection with the following result.
\begin{lemma} For every $\beta,\gamma\in \R$, $\delta\geq 0$, and $\varphi\in C^\infty_c(\R^2)$, there exists $C>0$ such that
\begin{equation}
\label{eq:SmoothMollificationNoise}
\|\varphi\ast f \|_{\mathcal{C}^{\beta}_{-\delta}} \leq C\|f\|_{\mathcal{C}^{\gamma}_{-\delta}}.
\end{equation}
\end{lemma}
\begin{proof}
In \cite[Lemma 2.2]{BCD} this estimate is proved for $\delta=0$. For $\delta>0$ we proceed as follows.
We consider two Littlewood-Paley partitions $\Delta_N$ and $\tilde \Delta_N$ such that $\Delta_N=\tilde \Delta_N\circ \Delta_N$ and hence
\begin{align*}
N^\beta \|\Delta_N (\varphi \ast f) \|_{L^\infty_{-\delta}} &= N^\beta \|\tilde \Delta_N \varphi \ast \Delta_N f\|_{L^\infty_{-\delta}} \\
&\leq N^\beta \|\tilde \Delta_N \varphi\|_{L^1_\delta} \|\Delta_N f\|_{L^\infty_{-\delta}}\\ 
&\leq N^{\beta-\gamma} \|\tilde \Delta_N \varphi\|_{L^\infty_{1 + 2 \delta}} N^\gamma \|\Delta_N f\|_{L^\infty_{-\delta}},
\end{align*}
where we used $\w^{-1}\leq 2\langle y\rangle\langle x-y\rangle^{-1}$ and the H\"older inequality.
We conclude since 
$\varphi\in \mathcal{C}_{1 + 2 \delta}^{\beta-\gamma}$.
\end{proof}

\subsection{Some estimates on the approximation of the identity $\rho_\varepsilon$}

We shall need the following estimate proved in \cite{bourdaud1} and in a different case, but with a similar proof, in \cite[Lemma 8]{bourdaud2}. 
\begin{lemma}
For every $\delta>0$, there exists $C>0$ such that:
\begin{equation}\label{bentin}\|\rho_\varepsilon\|_{{\mathcal B}^0_{1,2,\delta}}
\le C \sqrt{ |\ln \eps|}, \quad \forall \varepsilon\in \Big (0, \frac 12\Big).\end{equation}
\end{lemma}
\begin{proof}
We give the proof for the sake of completeness. 
First we notice that
$$\Delta_N (\w^\delta \rho_{\varepsilon}) (\varepsilon x)=N^2 \int_{\R^2} K(\varepsilon N(x-y)) \langle \varepsilon y\rangle^\delta\rho(y) dy$$
and hence it is easy to deduce
\begin{align}
\label{impdouble}
\begin{split}
\varepsilon^{-2}\|\Delta_N (\w^\delta \rho_{\varepsilon})\|_{L^1} 
&\leq C N^2\|K(\varepsilon Nx)\|_{L^1}\|\langle \varepsilon x\rangle^\delta\rho\|_{L^1} \\
&=C\varepsilon^{-2} \|K\|_{L^1} \|\langle \varepsilon x\rangle^\delta\rho\|_{L^1} \\
&\leq C\varepsilon^{-2} \|K\|_{L^1} \|\rho\|_{L^1_\delta}.
\end{split}
\end{align}
We fix $N_0$ as the unique dyadic such that $N_0\leq \frac 1\varepsilon<2N_0$
and from \eqref{impdouble} we get
\begin{equation}\label{travis}\sum_{\frac 12\leq N\leq 8N_0} \|\Delta_N (\w^\delta \rho_{\varepsilon})\|_{L^1}^2\leq C |\ln \varepsilon|.\end{equation}
Next, we focus on the case $N>8N_0$. First notice the identity
\begin{equation}\label{doubleloc}\varepsilon^2 \w^\delta \rho_{\varepsilon}=\sum_{{M-{\rm dyadic}}} \Delta_M (\langle \varepsilon x\rangle^\delta \rho) \Big(\frac x\varepsilon \Big),\end{equation}
along with the following inclusion for the support of the Fourier transfom
$$\supp {\mathcal F} \Big[\Delta_M (\langle \varepsilon x\rangle^\delta \rho)\Big(\frac x\varepsilon \Big)\Big]\subset \Big \{\frac M{2 \varepsilon}\leq |\eta|\leq
\frac{2M}{\varepsilon}\Big\}\subset 
\Big \{\frac {MN_0}2\leq |\eta|\leq 4MN_0\Big \}.$$
As a consequence, we get $$\Delta_N \Big[\Delta_M (\langle \varepsilon x\rangle^\delta \rho) \Big(\frac x\varepsilon \Big)\Big]=0, \quad \forall N \hbox{ s.t. } \frac N2 >4MN_0 \hbox{ or }  2N<\frac {MN_0} 2.$$
Hence, by \eqref{doubleloc} we get
\begin{equation}\label{doublelocnew}\varepsilon^2 \Delta_N (\w^\delta \rho_{\varepsilon})=\sum_{\frac{N}{8N_0}\leq M\leq \frac{4N}{N_0}} \Delta_N\Big[\Delta_M (\langle \varepsilon x\rangle^\delta \rho) \Big(\frac x\varepsilon \Big)\Big],
\quad \forall N\geq 8N_0,\end{equation}
which in turn implies (since there are exactly six dyadic numbers in the interval $[\frac{N}{8N_0}, \frac{4N}{N_0}]$ )
\begin{align*}
\|\Delta_N (\w^\delta \rho_{\varepsilon})\|_{L^1}^2 &\leq \Big(\sum_{\frac{N}{8N_0}\leq M\leq \frac{4N}{N_0}} \varepsilon^{-2}\Big \|\Delta_M (\langle \varepsilon x\rangle^\delta \rho) \Big(\frac x\varepsilon \Big)\Big\|_{L^1}\Big)^2\\ 
& \leq 6 \sum_{\frac{N}{8N_0}\leq M\leq \frac{4N}{N_0}}\|\Delta_M (\langle \varepsilon x\rangle^\delta \rho)\|_{L^1}^2, \quad \forall N\geq 8N_0
\end{align*}
and then
\begin{align}
\label{traviss}
\sum_{N>8 N_0} \|\Delta_N (\w^\delta\rho_{\varepsilon})\|_{L^1}^2 &\leq
36 \sum_{{M-{\rm dyadic}}} \|\Delta_M (\langle \varepsilon x\rangle^\delta \rho)\|_{L^1}^2 \leq C \| \langle \eps x \rangle^\delta \rho \|_{\mathcal{B}_{1, 2}^0}^2.
\end{align}
By combining \eqref{travis}, \eqref{traviss}, and the obvious bound $\sup_{\eps (0, \frac 12)} \| \langle \eps x \rangle^\delta \rho \|_{\mathcal{B}_{1, 2}^0} < \infty$, we conclude our estimate.
\end{proof}

The following estimate will also be useful.
\begin{lemma} For every $\beta\in (0,1)$, $\zeta>0$ with $\beta+\zeta\in (0,1)$, and $\delta\geq 0$, there exists $C>0$ such that
\begin{equation}\label{Yzeta}
 \|\rho_\eps\ast f\|_{\mathcal{C}^{\beta+1}_{-\delta}}\le C\eps^{-\beta-\zeta}\|f\|_{\mathcal{C}^{1-\zeta}_{-\delta}}, \quad \forall \varepsilon \in (0,1).
\end{equation}
\end{lemma}
\begin{proof}
Arguing as in the proof of \eqref{eq:SmoothMollificationNoise},
we get
\begin{equation*}
N^{1+\beta} \|\Delta_N (\rho_\eps \ast f) \|_{L^\infty_{-\delta}}
\leq N^{\beta+\zeta} \|\Delta_N \rho_\eps\|_{L^1_\delta} N^{1-\zeta} \|\tilde \Delta_N f\|_{L^\infty_{-\delta}}
\end{equation*}
and hence we conclude provided that we show
\begin{equation}\label{lasttintut*}(\varepsilon N)^{\beta+\zeta} \| \Delta_N \rho_\eps\|_{L^1_\delta}\leq C, \quad \forall \varepsilon \in (0, 1), 
\quad \forall {N\hbox{ } {\rm dyadic}}.\end{equation}
By using \eqref{LPaley}, and since we are assuming $\beta+\zeta\in (0,1)$, it is sufficient to prove
\begin{equation}\label{lasttintut} (\varepsilon N)^{\beta+\zeta} \| \Delta_N (\w^{\delta} \rho_\eps)\|_{L^1}\leq C, \quad \forall \varepsilon \in (0, 1),
\quad \forall {N\hbox{ } {\rm dyadic}}.\end{equation}
In order to do that we introduce 
the unique dyadic $N_0$ such that $N_0\leq \frac 1\varepsilon<2N_0$. Notice that in case $\frac{N}{N_0}\leq 8$ the estimate above is trivial.
On the other hand from \eqref{doublelocnew}, we get
\begin{align*}
\| \Delta_N (\langle x \rangle^\delta \rho_\eps) \|_{L^1} \leq \sum_{\frac{N}{8N_0}\leq M\leq \frac{4N}{N_0}} \| \Delta_N (\langle \eps x \rangle^\delta \rho) \|_{L^1}, \quad \forall N > 8 N_0,
\end{align*}
and hence
\begin{align*}
\Big( \frac{N}{N_0} \Big)^{\beta + \zeta} \| \Delta_N (\langle x \rangle^\delta \rho_\eps) \|_{L^1} \leq C \sup_{\eps \in (0, 1)} \| \langle \eps x \rangle^\delta \rho \|_{\mathcal{B}_{1, \infty}^{\beta + \zeta}}, \quad \forall N > 8 N_0.
\end{align*}
We conclude since $\sup_{\eps \in (0, 1)} \| \langle \eps x \rangle^\delta \rho \|_{\mathcal{B}_{1, \infty}^{\beta + \zeta}} < \infty$.
\end{proof}


We close the subsection with the following result.
\begin{lemma}
\label{l2} 
For every $\alpha\in\R$, $\eta \in (0,1)$, and $\delta \geq 0$, there exists $C>0$ such that
\begin{equation}\label{eq:l2}
\|\rho_\eps*f-f \|_{{\mathcal C}^\alpha_{-\delta}}\le C\eps^\eta \|f\|_{{\mathcal C}^{\alpha+\eta}_{-\delta}}, \quad \forall \varepsilon \in (0, 1). 
\end{equation}
\end{lemma}
\begin{proof}
For $\eps N \geq 1$, we have the bound 
\begin{align}
\label{fristweightdiff}
\begin{split}
N^\alpha \|\Delta_N (\rho_\eps* f)-\Delta_N f\|_{L^\infty_{-\delta}}
&\leq N^\alpha \|\rho_\eps* \Delta_N f\|_{L^\infty_{-\delta}}+N^\alpha \|\Delta_N f\|_{L^\infty_{-\delta}} \\
&\leq N^\alpha \|\rho_\eps\|_{L^1_{\delta}} \|\Delta_N f\|_{L^\infty_{-\delta}}+N^\alpha \|\Delta_N f\|_{L^\infty_{-\delta}} \\
&\leq C N^{\alpha+\eta} \varepsilon^{\eta}\|\Delta_N f\|_{L^\infty_{-\delta}},
\end{split}
\end{align}
where we used $\w^{-1}\leq 2\langle y\rangle\langle x-y\rangle^{-1}$.
To deal with the case $\varepsilon N<1$, we notice that
$$\w^{-\delta} \Big(\Delta_N (\rho_\eps* f)-\Delta_N f\Big)=\int_{\R^2} \tilde K_{\varepsilon,N} (x,y) \langle y \rangle^{-\delta} f(y) dy$$
where 
\begin{align*}
\tilde K_{\varepsilon,N} (x,y) &= N^2\frac{\langle y \rangle^\delta}{\w^\delta}\Big(\int_{\R^2} K(Nx-Nz)\rho_\varepsilon(z-y)dz-K(Nx-Ny)\Big) \\
&= N^2  \frac{\langle y \rangle^\delta}{\w^\delta}\int_{\R^2}\Big (K(Nx-Nz)-K(Nx-Ny)\Big )\rho_\varepsilon(z-y)dz.\end{align*}
Next, fox fixed $x$, we compute
\begin{align}
\label{strole}
\begin{split}
\int_{\R^2} &|\tilde K_{\varepsilon,N} (x,y)|dy\\& \leq \frac{N^2}{\w^\delta}  \int_{\R^2}\int_{\R^2} |K(Nx-Nz)-K(Nx-Ny)|\rho_\varepsilon(z-y)
\langle y \rangle^\delta dydz \\
&=\frac{(N \varepsilon)^2}{\w^\delta} \int_{\R^2}\int_{\R^2} |K(Nx-N\varepsilon z)-K(Nx-N\varepsilon y)|\rho(z-y)\langle \varepsilon y \rangle^\delta dydz \\
&\leq \frac{(N\varepsilon)^3}{\w^\delta} \int_{\R^2}\int_{\R^2} \Big( \sup_{[Nx-N\varepsilon z,Nx-N\varepsilon y ]} |\nabla K|\Big) |z-y| \rho(z-y) \langle \varepsilon y \rangle^\delta dydz
\end{split}
\end{align}
and notice that for every $\bar x\in \R^2$ and $\lambda\in (0,1)$, we get
\begin{align*}
\int_{\R^2} &\int_{\R^2} \Big( \sup_{[\bar x- \lambda z, \bar x-\lambda y ]} |\nabla K|\Big) |z-y| \rho(z-y) \langle \varepsilon y \rangle^\delta dydz \\
&\leq C\int_{\R^2} \int_{|z-y|<1} \Big( \sup_{B((\bar x- \lambda z), \lambda)} |\nabla K|\Big) \langle \varepsilon z \rangle^\delta dzdy \\
&\leq C  \int_{\R^2} \Big( \sup_{B((\bar x- \lambda z), 1)} |\nabla K|\Big) \langle \varepsilon z \rangle^\delta dz.
\end{align*}
Here, we denoted by $[a, b]$ the segment between $a$ and $b$, $B(x,r)$ the ball in $\R^2$ centered in $x$ of radius $r$ and we used the inclusion
$[\bar x- \lambda z, \bar x-\lambda y ]\subset B((\bar x- \lambda z), 1)$ since $|z-y|<1$ and $\lambda\in (0,1)$.
In particular, by introducing the function 
\begin{equation}\label{Gw}G(w)=\sup_{B(w, 1)} |\nabla K|,\end{equation} we get 
from the estimate above
\begin{align}
\label{putle}
\begin{split}
\int_{\R^2} &\int_{\R^2} \Big( \sup_{[\bar x- \lambda z, \bar x-\lambda y ]} |\nabla K|\Big) |z-y| \rho(z-y)\langle \varepsilon z \rangle^\delta
dydz \\
&\leq \int_{\R^2} G(\bar x - \lambda z) \langle \varepsilon z \rangle^\delta dz \\
&=\lambda^{-2}
\int_{\R^2} G(u) \langle \varepsilon \lambda^{-1}(\bar x - u) \rangle^\delta du.
\end{split}
\end{align}
By combining \eqref{strole} with \eqref{putle}, where we choose $\lambda=\varepsilon N$ and $\bar x=Nx$, we conclude
\begin{equation}\label{schux}\sup_{x,N} \int_{\R^2} |\tilde K_{\varepsilon,N} (x,y)|dy=O(\varepsilon N),\end{equation}
where we used
$$\sup_{\bar x, N} \int_{\R^2} G(u) \langle N^{-1}(\bar x - u) \rangle^\delta du<\infty.$$
Similarly, one can show 
\begin{equation}\label{schuy}\sup_{y,N} \int_{\R^2} |\tilde K_{\varepsilon,N} (x,y)|dx=O(\varepsilon N).\end{equation}
In fact for $y$ fixed we get:
\begin{align*}
\int_{\R^2} &|\tilde K_{\varepsilon,N} (x,y)|dx \\
&=N^2 \langle y \rangle^\delta \int_{\R^2}\int_{\R^2} |K(Nx-N\varepsilon z-Ny)-K(Nx-N y)|\rho(z)\langle x \rangle^{-\delta} dxdz
\\\nonumber&
\leq C \varepsilon N^3 \langle y \rangle^\delta \int_{\R^2}\int_{\R^2} \Big( \sup_{[Nx-N\varepsilon z-Ny,Nx-N y ]} |\nabla K|\Big) |z| \rho(z) \langle x \rangle^{-\delta} dxdz
\\\nonumber &
\leq C\varepsilon N^3 \langle y \rangle^\delta \int_{\R^2}\int_{|z|<1} G(N(x-y))  \langle x \rangle^{-\delta} dxdz,
\end{align*}
where the function $G(w)$ is defined in \eqref{Gw}
and we used the inclusion $$[Nx-N\varepsilon z-Ny,Nx-N y ]\subset B((Nx-Ny), 1)$$
for $|z|<1$ and $\varepsilon N<1$.
By a change of variable, we conclude
$$\int_{\R^2} |\tilde K_{\varepsilon,N} (x,y)|dx\leq C\varepsilon N \langle y \rangle^\delta \int_{\R^2}G(x)  \langle \frac xN+y \rangle^{-\delta} dx
\leq C \varepsilon N \int_{\R^2} G(x) \w^\delta dx.$$
Summarizing by \eqref{schux} and \eqref{schuy} we can apply the Schuur Lemma and we get
$$\|\Delta_N (\rho_\eps* f)-\Delta_N f\|_{L^\infty_{-\delta}}\leq C\varepsilon N \|\Delta_N f\|_{L^\infty_{-\delta}},\quad \varepsilon N<1,$$
which in turn implies
$$\|\Delta_N (\rho_\eps* f)-\Delta_N f\|_{L^\infty_{-\delta}}\leq C\varepsilon^\eta N^\eta \|\Delta_N f\|_{L^\infty_{-\delta}},\quad \varepsilon N<1.$$
We conclude by combining this estimate with \eqref{fristweightdiff}.
\end{proof}

\subsection{Embeddings and product rules}
We first record the useful interpolation inequality in weighted Besov spaces (see \cite[Theorem 3.8]{SSV14})
\begin{lemma}
Let $p_0,q_0,p_1,q_1,p,q\in [1,\infty],\,\alpha_0,\alpha_1,\mu_0,\mu_1, \alpha, \mu\in \R$ be such that
\begin{align*}
\frac{1}{p}&=\frac{1-\varTheta}{p_0}+\frac{\varTheta}{p_1},\quad\frac{1}{q}=\frac{1-\varTheta}{q_0}+\frac{\varTheta}{q_1},\\
\alpha &= (1-\varTheta)\alpha_0+\varTheta \alpha_1, \quad  \mu=(1-\varTheta)\mu_0+\varTheta \mu_1
\end{align*} 
for a suitable $\varTheta\in [0,1]$. 
Then there exists $C>0$ such that
\begin{equation}\label{lem:BesovInterpolation}
\|f\|_{\mathcal{B}^{\alpha}_{p,q,\mu}}\leq C \|f\|^{1-\varTheta}_{\mathcal{B}^{\alpha_0}_{p_0,q_0, \mu_0}} \|f\|^\varTheta_{\mathcal{B}^{\alpha_1}_{p_1,q_1,\mu_1}}\,.
\end{equation} 
\end{lemma}
Next, we state a Sobolev embedding for weighted Sobolev spaces. The proof follows from the corresponding unweighted version along with \eqref{eq:PullWeight}.
\begin{lemma}
Let $p\in [2,\infty)$, $\alpha\geq 1-\frac{2}{p}$, and $\mu,\nu \in \R$ such that $\nu \leq \mu$. Then, 
we have the continuous embedding
\begin{equation}\label{Sobolevweighted}
H^{\alpha}_{\mu} \subset L^p_{\nu}.
\end{equation}
Moreover for every $\alpha>1$ we have
\begin{equation}\label{Sobolevweightedinfty}
H^{\alpha}_{\mu} \subset L^\infty_{\nu}.
\end{equation}

\end{lemma}

As a consequence, we can prove the following estimate.
\begin{lemma} For every
$q\in [2,\infty)$ and $\delta \in (0, 1]$, there exists $C>0$ such that
\begin{equation}\label{*****}\|f\|_{W^{1,q}_{\delta}}\leq C \|f\|_{H^2_{-\delta}}^{1-\frac 1q} \|f\|_{L^2_{(2q-1)\delta }}^{\frac 1q}.
\end{equation}
\end{lemma}
\begin{proof}
We have the following chain of estimates
\begin{align*}
\|f\|_{W^{1,q}_{\delta}} &\leq C \|f\|_{H^{2-\frac 2q}_{\delta}} \\
&\leq C \|f\|_{H^2_{-\delta}}^{1-\frac 2q} 
\|f\|_{H^1_{(q-1)\delta }}^\frac 2q \\
&\leq C \|f\|_{H^2_{-\delta}}^{1-\frac 2q} 
\|f\|_{H^2_{-\delta }}^\frac 1q \|f\|_{L^2_{(2q-1)\delta }}^{\frac 1q}
\end{align*}
where the first inequality is a consequence of \eqref{commutator2} and \eqref{Sobolevweighted}. Moreover we have also used twice special cases
of \eqref{lem:BesovInterpolation}.
\end{proof}

We shall also need the following product estimate (see \cite{BCD} or \cite{PT16}).
\begin{lemma} Let $\alpha_1, \alpha_2\in \R$ with $\alpha_1 + \alpha_2 > 0$, $\alpha = \min (\alpha_1, \alpha_2, \alpha_1 + \alpha_2)$, $\mu_1, \mu_2 \in \R$, $\mu = \mu_1 + \mu_2$, $p_1, p_2 \in [1, \infty]$, and $\frac{1}{p} = \frac{1}{p_1} + \frac{1}{p_2}$. Then, for any $\kappa > 0$, we have
\begin{equation}\label{prod1}
\| f_1 \cdot f_2 \|_{\mathcal{B}^{\alpha - \kappa}_{p, p, \mu}} \leq C \| f_1 \|_{\mathcal{B}^{\alpha_1}_{p_1, p_1, \mu_1}} \| f_2 \|_{\mathcal{B}^{\alpha_2}_{p_2, p_2, \mu_2}}.
\end{equation}
Also, we have
\begin{equation}\label{prodim}\|f_1\cdot f_2\|_{{\mathcal C}^{\alpha}_{\mu}}\leq C\|f_1\|_{{\mathcal C}^{\alpha_1}_{\mu_1}}
\|f_2\|_{{\mathcal C}^{\alpha_2}_{\mu_2}}.
\end{equation}
\end{lemma}

The following duality estimate will also be useful (see \cite[Theorem 2.11.2]{Tri83}).
\begin{lemma}
Let $\alpha, \mu \in \R$ and $p, q \in [1, \infty)$. Then, we have
\begin{equation}\label{dual}
\Big| \int_{\R^2} f_1 \cdot f_2 dx \Big| \leq C \| f_1 \|_{\mathcal{B}^\alpha_{p, q, \mu}} \| f_2 \|_{\mathcal{B}^{-\alpha}_{p', q', - \mu}},
\end{equation}
where $\frac{1}{p} + \frac{1}{p'} = 1$ and $\frac{1}{q} + \frac{1}{q'} = 1$.
\end{lemma}

\section{Stochastics bounds}
\label{subsec:Noise}

The following results are improvements of some results from \cite{HL15} and \cite{DM}, where similar estimates were
given in terms of moments. In this paper, these estimates hold almost surely. 
\subsection{Estimates in classical spaces}
We first show some estimates in classical spaces.

\begin{prop}\label{lem:NoiseLp}
Given $\delta\in (0,1)$, $r\in (2,\infty)$ such that $\delta\cdot  r>2$, and $\eps \in (0, \frac 12)$,  we have
\begin{equation}\label{firstnoiseLp}
\|\D Y_\eps\|_{L^r_{-\delta}}^2+\|\wDYe\|_{L^r_{-\delta}}\leq C(\omega) |\ln \eps|, \quad \hbox{ a.s. } \omega.
\end{equation}
Moreover, for $\alpha\in (0,1),\,\delta>0, \; \beta \in \R$, $\eps \in (0, 1)$, and $\varphi \in C_c^\infty (\R^2)$, there exists $\kappa\in (0,1)$ such that
\begin{equation}\label{convexpY}
\|Y_\eps-Y\|_{\mathcal{C}^{\alpha}_{-\delta}}
+\|\varphi\ast \xi_\eps-\varphi\ast \xi\|_{\mathcal{C}^{\beta}_{-\delta}}
\leq C(\omega) \eps^{\kappa}, \quad \hbox{ a.s. } \omega.
\end{equation}
\end{prop} 
To prove Proposition \ref{lem:NoiseLp}, we shall use the following result (see \cite[Proposition 3.1]{hytonen-veraar} and \cite[Proposition 2.3]{veraar}).
\begin{lemma}\label{l1}
Let $(X, \|\cdot \|_{X})$ be a separable Banach space and $(\eta_n)$ be  a sequence of $X$-valued random variables.
Assume that there exists a sequence $\sigma_n$ of real numbers such that:
$$
\E\Big ((\eta_n,f)^2\Big )\le \sigma_n^2 \|f\|_{X'}^2, \quad \forall f\in X'.
$$
Then 
$$
\E \Big(\sup_n \|\eta_n\|_X\Big)\le \sup_n \E(\|\eta_n\|_X)+ 3 \rho(\sigma_n),
$$
where
$$
\rho(\sigma_n)= \inf\Big\{\delta \;|\;\sum_n \sigma_n^2\delta^{-2}\exp (-2^{-1}(\delta \sigma_n^{-1})^2)\le 1  \Big\}.
$$
\end{lemma}
\begin{remark}\label{rhoeasy}
As recalled in \cite{veraar}, if $\sigma_n\le \alpha^n$ then $\rho(\sigma_n)\sim \sqrt{\ln(1-\alpha)^{-1}}$.
\end{remark}

We also need the following result, which is an immediate consequence of \cite[Lemma 2.5]{DM}.
\begin{lemma}
\label{lem:BoundNoiseP}
For any  $\alpha\in (0,1)$ and $\delta>0$, we have 
\begin{equation}\label{eq:BoundNoiseP}
\|Y\|_{\mathcal{C}^\alpha_{-\delta}}+\|\xi\|_{\mathcal{C}^{\alpha-2}_{-\delta}}\leq C(\omega), \quad \hbox{ a.s. } \omega.
\end{equation}
\end{lemma}

\begin{proof}[Proof of Prop. \ref{lem:NoiseLp}]
We split the proof in several steps.
First we show the following property:
\begin{equation}\label{e1}
 \E \Big (\|\nabla Y \|_{B_{r,\infty,-\delta}^{0}}\Big) <\infty, \quad \delta \cdot r>2, \quad r\in (2, \infty),
\end{equation}
whose  proof is a generalization of \cite[Theorem 3.4]{veraar}. 

For every fixed $x\in \R^2$, we have that $\Delta_N \xi(x)$ is a Gaussian random random variable
and hence
\begin{align*}
\E\left( \|\Delta_N \xi\|_{L^r_{-\delta}}^r\right)
&=\int_{\R^2} \E (\left|\Delta_N \xi(x)\right|^r) \w^{-\delta r}dx\\
& \le C \int_{\R^2} \left(\E \left|\Delta_N\xi(x)\right|^2\right)^{r/2} \w^{-\delta r}dx.
\end{align*}
Moreover, we have
$$
\E (\left|\Delta_N \xi(x)\right|^2) =N^4 \E \Big[\Big(\xi, K(N(x-\cdot)\Big)^2\Big] = N^{4} \| K(N(x-\cdot)\|_{L^2}^2= N^2\|K\|_{L^2}^2
$$
and hence 
\begin{equation}\label{Nfixed}
\E \left( \|\Delta_N \xi\|_{L^r_{-\delta}} \right) \le C N
\end{equation}
provided that $\delta \cdot r >2$. 
Also, for $f\in L^{r'}_\delta$, where $\frac 1r+\frac 1 {r'}=1$, we get
\begin{align}
\label{moments}
\begin{split}
\E\Big ((\Delta_N \xi,f)^2\Big) &= \E\Big ((\xi, \Delta_Nf)^2\Big) \\
&= N^{4} \| K(N\cdot)\ast f\|_{L^2}^2 \\
& \le N^4 \|K(N\cdot)\|_{L^\frac{2r}{r+2}}^2 \|f\|_{L^{r'}}^2 \\
&\le C N^{\frac{2(r-2)}{r}}\|f\|_{L^{r'}_\delta}^2.
\end{split}
\end{align}
By using the fact that $L^{r'}_{\delta}= (L^r_{-\delta})'$ and \eqref{Nfixed}, we can apply Lemma \ref{l1}, where we choose: $X=L^r_{-\delta}$,
$\eta_N=N^{-1}\Delta_N \xi$, $\sigma_N=N^{-\frac{1}{r}}$ (see \eqref{moments}).
Hence, we obtain
\begin{equation}\label{xiprime}
\E(\|\xi \|_{B_{r,\infty,-\delta}^{-1}}) =\E\left(\sup_N N^{-1}\|\Delta_N \xi\|_{L^r_{-\delta}}\right)\le C + C\rho(N^{-\frac 2r})\le C.
\end{equation}
where we have used Remark \ref{rhoeasy} at the last step.
The estimate \eqref{e1} clearly follows from \eqref{xiprime}.

Next, we claim that
\begin{equation}\label{B0p2-delta}
\| \rho_\eps*\nabla Y \|_{B^0_{r,2,-\delta}}   \le C \|\rho_\eps\|_{B^0_{1,2,\delta}}\|\nabla Y \|_{B_{r,\infty,-\delta}^{0}}.
\end{equation}
Once this estimate is established, by combining  \eqref{e1} with \eqref{bentin}, we get
$$\|\D Y_\eps\|_{L^r_{-\delta}}^2\leq C(\omega) |\ln \eps|$$
and in turn also
$$\|\wDYe\|_{L^r_{-\delta}}^2\leq C(\omega) |\ln \eps|$$
since $\wDYe=\nabla Y_\eps^2-c_\eps$ with $c_\eps \sim |\ln \eps|$. Hence, \eqref{firstnoiseLp} follows from \eqref{B0p2-delta}, 
whose proof
without weight is in \cite[ Theorem 2.2]{KS}. Here, we denote by
$\Delta_N$ and $\tilde \Delta_N$ two Littlewood-Paley decompositions such that $\tilde \Delta_N\circ \Delta_N=\Delta_N$, and we write 
\begin{align*}
\| \rho_\eps*\nabla Y \|_{B^0_{r,2,-\delta}}^2 &=\sum_{N-{\rm dyadic}} \| \Delta_N  ( \rho_\eps*\nabla Y)\|_{L^r_{-\delta}}^2 \\
&=\sum_{N-{\rm dyadic}} \| \tilde \Delta_N  \rho_\eps* \Delta_N (\nabla Y)\|_{L^r_{-\delta}}^2 \\
&\le C\sum_{N-{\rm dyadic}}\|\tilde\Delta_N  \rho_\eps \|_{L^1_\delta}^2 \|\Delta_N  \nabla Y\|_{L^r_{-\delta}}^2 \\
& \le C \|\rho_\eps\|_{B^0_{1,2,\delta}}^2 \|\nabla Y \|_{B_{r,\infty, -\delta}^0}^2,
\end{align*}
where we have used the inequality $\w^{-1}\le\;  2 \langle y\rangle \langle x-y\rangle^{-1}$. Thus, \eqref{B0p2-delta} follows.

Then, we consider \eqref{convexpY}. The bound 
$$\|\varphi\ast \xi_\eps-\varphi\ast \xi\|_{\mathcal{C}^{\beta}_{-\delta}}
\leq C(\omega) \eps^{\kappa}, \quad \hbox{ a.s. } \omega$$
follows by combining \eqref{eq:l2} with \eqref{eq:SmoothMollificationNoise} and \eqref{eq:BoundNoiseP}.
In order to conclude the proof of \eqref{convexpY}, 
we notice that again by \eqref{eq:l2} and \eqref{eq:BoundNoiseP} we get
\begin{equation*}
\|Y_\eps-Y\|_{\mathcal{C}^{\alpha}_{-\delta}}\leq C(\omega) \eps^{\kappa}, \quad \hbox{ a.s. } \omega,
\end{equation*}
as desired.
\end{proof}

We also establish the following uniform bound and convergence result for $e^{aY_\eps}$ for any $a \in \R$.
\begin{prop}
\label{prop:exp}
For $\alpha \in (0, 1)$, $\delta > 0$, and $a \in \R$, we have
\begin{equation}\label{exponential}\sup_{\varepsilon\in (0,1)}\|e^{a Y_\varepsilon}\|_{\mathcal{C}^\alpha_{-\delta}}\leq C(\omega), \quad \hbox{ a.s. } \omega.
\end{equation}
Moreover, there exists $\kappa \in (0, 1)$ such that
\begin{equation}\label{exp_conv}
\| e^{a Y_\eps} - e^{a Y} \|_{L^\infty_{- \delta}} \leq C(\omega) \eps^\kappa, \quad \hbox{ a.s. } \omega.
\end{equation}
\end{prop}

To prove Proposition \ref{prop:exp}, we shall use the following lemma (see \cite[Corollary 2.6]{DM}).
\begin{lemma}
For any $\alpha \in (0, 1)$, $a \in \R$, and $\delta > 0$, we have
\begin{align}\label{expbound}
\| e^{a Y} \|_{C^\alpha_{-\delta}} \leq C(\omega), \quad \hbox{ a.s. } \omega.
\end{align}
\end{lemma}

We also need the following result (see \cite[Lemma 2.3]{DM} and \cite[Lemma 5.3]{AC}). Below we shall use the functions $\chi_k \in C_c^\infty (\R^2)$ with $k \in \N$, $\supp \chi_k \subseteq [-k - 1, k + 1]^2$, and $\chi_k = 1$ on $[-k, k]^2$.
\begin{lemma}
For $\alpha < 1$, there exist $\lambda, \lambda' > 0$ such that
\begin{align}\label{expnoise}
\sup_{k \in \N} \frac{\E \Big( \exp (\lambda \| \chi_k \xi \|_{\mathcal{C}^{\alpha - 2}}^2) \Big)}{k^{\lambda'}} \leq C.
\end{align}
\end{lemma}

\begin{proof}[Proof of Prop. \ref{prop:exp}]
We first establish the uniform bound \eqref{exponential}, whose proof is similar to that of Corollary 2.6 in \cite{DM}. For every $k \in N$, we easily get
\begin{align}
\label{exp1}
\begin{split}
\| Y_\eps \|_{\mathcal{C}^\alpha ([-k, k]^2)} &= \| \rho_\eps \ast Y \|_{\mathcal{C}^\alpha ([-k, k]^2)} \\
&\leq \| Y \|_{\mathcal{C}^\alpha ([-k - 1, k + 1]^2)} \\
&\leq C \| G \ast \chi_{k+2} \xi \|_{\mathcal{C}^\alpha} \\
&\leq C \| \chi_{k + 2} \xi \|_{\mathcal{C}^{\alpha - 2}}
\end{split}
\end{align}
uniformly in $\eps \in (0, 1)$, where in the second inequality we used $Y = G \ast \xi$ and the fact that $\supp G \subseteq B(0, 1)$, and in the last inequality we used a Schauder estimate (see \cite{Sim}). We also note that
\begin{align}
\label{exp2}
\| e^{a Y_{\eps}} \|_{\mathcal{C}^\alpha_{-\delta}} \leq C \sup_{k \in \N} \frac{\| e^{a Y_\eps} \|_{\mathcal{C}^\alpha ([-k, k]^2)}}{k^\delta} \leq C \sup_{k \in \N} \frac{\exp (C|a| \| Y_\eps \|_{\mathcal{C}^\alpha ([-k, k]^2)})}{k^\delta}.
\end{align}
Thus, for $p > 1$ big enough, by \eqref{exp2}, \eqref{exp1}, and \eqref{expnoise}, we get
\begin{align*}
\E \Big( \sup_{\eps \in (0, 1)} \| e^{a Y_\eps} \|_{\mathcal{C}^\alpha_{- \delta}}^p \Big) &\leq C \E \bigg( \bigg| \sup_{\eps \in (0, 1)} \sup_{k \in \N} \frac{\exp (C|a| \| Y_\eps \|_{\mathcal{C}^\alpha ([-k, k]^2)})}{k^\delta} \bigg|^p \bigg) \\
&\leq C \E \bigg( \bigg| \sup_{k \in \N} \frac{\exp (C|a| \| \chi_{k + 2} \xi \|_{\mathcal{C}^{\alpha - 2} })}{k^\delta} \bigg|^p \bigg) \\
&\leq C \sum_{k = 1}^\infty \frac{\E \Big( \exp (p C |a| \| \chi_k \xi \|_{\mathcal{C}^{\alpha - 2}}) \Big)}{k^{\delta p}} \\
&\leq C \sum_{k = 1}^\infty \frac{\E \Big( \exp (\lambda \| \chi_k \xi \|_{\mathcal{C}^{\alpha - 2}}^2) \Big)}{k^{2 + \lambda'}} \\
&\leq C,
\end{align*}
where in the last step we used $\exp(p C |a| x) \leq C \exp (\lambda x^2)$ and we picked $p$ big enough such that $\delta p \geq 2 + \lambda'$. The bound \eqref{exponential} then follows.

We now consider \eqref{exp_conv}. By \eqref{convexpY}, \eqref{exponential}, and \eqref{expbound}, we get
\begin{align*}
\| e^{a Y_\eps} - e^{a Y} \|_{L^\infty_{- \delta}} \leq C \| Y_\eps - Y \|_{L^\infty_{- \frac{\delta}{2}}} \Big( \| e^{a Y_\eps} \|_{L^\infty_{- \frac{\delta}{2}}} + \| e^{a Y} \|_{L^\infty_{- \frac{\delta}{2}}} \Big) \leq C(\omega) \eps^\kappa,
\end{align*}
which is the desired estimate.
\end{proof}


\subsection{Estimates in spaces at negative regularity}

The main point of this subsection is the following result where the convergence in negative regularity occurs almost surely in $\omega \in \Omega$.

\begin{prop}
\label{prop:conv}
For $\alpha\in (0,1)$, $\delta>0$, and $\eps \in (0, \frac 12)$, there exists $\kappa\in (0,1)$ such that
\begin{equation}\label{lem:ConvergenceYnew}\|\nabla Y_\varepsilon- \nabla Y\|_{\mathcal{C}^{\alpha-1}_{-\delta}}+
\|\wDYe-\wDY\|_{\mathcal{C}^{\alpha-1}_{-\delta}}
\leq C(\omega) \eps^{\kappa}, \quad \hbox{ a.s. } \omega.
\end{equation}
Moreover, we have
\begin{equation}\label{tY2_conv}
\| \widetilde{\wDYe} - \widetilde{\wDY} \|_{\mathcal{C}^{\alpha - 1}_{-\delta}} \leq C(\omega) \eps^{\kappa}, \quad \hbox{ a.s. } \omega,
\end{equation}
where $\widetilde{\wDYe}$ and $\widetilde{\wDY}$ are defined in \eqref{tDY2e} and \eqref{tDY2}, respectively.
\end{prop}
To prove Proposition \ref{prop:conv}, we need the following result which follows from \cite{HL15}.
\begin{lemma}

For any  $\alpha\in (0,1)$ and $\delta>0$, we have the bound 
\begin{equation}\label{lem:BoundNoise}
	\|\nabla Y\|_{\mathcal{C}^{\alpha-1}_{-\delta}}+\|\wDY\|_{\mathcal{C}^{\alpha-1}_{-\delta}}\leq C(\omega),
	\quad \hbox{ a.s. } \omega.
\end{equation}
\end{lemma}

\begin{proof}[Proof of Proposition \ref{prop:conv}]
The estimate $\|\nabla Y_\varepsilon- \nabla Y\|_{\mathcal{C}^{\alpha-1}_{-\delta}} \leq C(\omega) \eps^{\kappa}$
follows by combining \eqref{eq:l2} with $\|\nabla Y\|_{\mathcal{C}^{\alpha-1}_{-\delta}}\leq C(\omega)$ (see \eqref{lem:BoundNoise}). Also, the estimate \eqref{tY2_conv} follows immediately from \eqref{lem:ConvergenceYnew} and \eqref{convexpY}

Hence, we focus on the proof of  $\|\wDYe-\wDY\|_{\mathcal{C}^{\alpha-1}_{-\delta}}\leq C(\omega) \eps^{\kappa}$. The argument is a little more complicated since Wick products cannot be estimated pathwise. It is shown in \cite{HL15} that there exists $\kappa_0>0$ such that for all $k\ge 1$ we have:
\begin{equation}\label{e1.14bis}
\E\left(\|\wDYe-\wDY\|_{\mathcal{C}^{\alpha-1}_{-\delta}}^k\right) \le C \eps^{k\kappa_0}.
\end{equation} 

Recall \eqref{WICK:} and notice that by elementary considerations we have:
\begin{align}
\label{tinbenfr}
\begin{split}
|c_\eps-c_\eta| &\le \|(\rho_\eps-\rho_\eta) \ast \nabla G\|_{L^2}\Big(\|\rho_\eps \ast \nabla G\|_{L^2}+\|\rho_\eta \ast \nabla G\|_{L^2}\Big) \\
& \le C|\ln \eps|^{\frac 12} \|\rho_\eps-\rho_\eta\|_{L^q}, \quad \forall q\in (2,\infty), \quad 0<\eps<\eta<\frac 12,
\end{split}
\end{align}
where we used $\nabla G\in L^p$ for any $p\in (1,2)$, along with Young convolution inequality and
 the classical bound $\|\rho_\eps \ast \nabla G\|_{L^2}^2=c_\eps \sim |\ln \eps|$ when $\eps\to 0$. 
By combining the following inequality (whose proof is elementary)
\begin{equation}\label{etaepseta}
\|\rho_\eps-\rho_\eta\|_{L^q} \le C \eps^{-3+2/q} |\eps-\eta|, \quad  0<\eps<\eta<\frac 12, \; q \ge 1,
\end{equation}
with \eqref{tinbenfr}, we get
\begin{equation}\label{e1.15}
  |c_\eps-c_\eta| \le C \eps^{-\gamma} |\eps-\eta|, \quad \quad  0<\eps<\eta<\frac 12, \quad \gamma>1.
 \end{equation}
 Moreover, by \eqref{prodim} and the estimate $\| \nabla G \ast f \|_{\mathcal{C}^{\gamma}_\chi} \leq C \| f \|_{\mathcal{C}^{\gamma - 1}_\chi}$ ($\gamma, \chi \in \R$) from \cite{HL15}, we obtain that for $1>\beta>1-\alpha$ and  $p> \frac 2{2-\alpha}$,
\begin{align}
\label{wicK}
\begin{split}
\| &\nabla Y_\eps^2-\nabla Y_\eta^2\|_{\mathcal{C}^{\alpha-1}_{-\delta}} \\
&\le C  \|\nabla Y_\eps- \nabla Y_\eta\|_{\mathcal{C}^{\alpha-1}_{-\frac \delta 2}}\left( \|\nabla Y_\eps\|_{\mathcal{C}^{\beta}_{-\frac \delta 2}}+   \|\nabla Y_\eta\|_{\mathcal{C}^{\beta}_{-\frac \delta 2}}\right)\\
&\le C \| \xi_\eps- \xi_\eta\|_{\mathcal{C}^{\alpha-2}_{-\frac \delta 2}}\left( \| Y_\eps\|_{\mathcal{C}^{\beta+1}_{-\frac \delta 2}}+\|Y_\eta\|_{\mathcal{C}^{\beta+1}_{-\frac \delta 2}}\right) \\
&\le C \| \xi_\eps- \xi_\eta\|_{L^p_{-\frac \delta 2}}\left( \| Y_\eps\|_{\mathcal{C}^{\beta+1}_{-\frac \delta 2}}+\|Y_\eta\|_{\mathcal{C}^{\beta+1}_{-\frac \delta 2}}\right),
\end{split}
\end{align}
where we used at the last step $L^p_{-\frac \delta 2}\subset \mathcal{C}^{\alpha-2}_{-\frac \delta 2}$
 for $p> \frac 2{2-\alpha}$. In fact, this embedding comes from
 the following computation (recall \eqref{eq:PullWeight}):
  \begin{equation*}\sup_N \Big(N^{\alpha-2}\|\Delta_N (\langle x \rangle^{-\frac \delta 2} f) \|_{L^\infty}\Big)
 \leq C \sup_N N^{\alpha-2}\|\Delta_N (\langle x \rangle^{-\frac \delta 2} f) \|_{W^{2-\alpha,p}}
 \le C \|f \|_{L^{p}_{-\frac \delta 2}},
\end{equation*}
where we used the classical Sobolev embedding $W^{2-\alpha,p}\subset L^\infty$ for $(2-\alpha)\cdot p>2$.
 \\
Now, using the Gaussianity of $\xi$ and Minkowski's inequality, we have the following estimates for $k\geq p$:
\begin{align}
\label{pkmin}
\begin{split}
\E\left(\| \xi_\eps- \xi_\eta\|_{L^p_{-\frac \delta2}}^k \right) &\leq \Big(\int_{\R^2} \w^{-\frac{p\delta}2}\E\left(|\xi_\eps(x)-\xi_\eta(x)|^k \right)^\frac pk dx\Big)^\frac kp \\
& \le C \Big(\int_{\R^2} \w^{-\frac{p\delta}2}\E\left(|\xi_\eps(x)-\xi_\eta(x)|^2 \right)^{\frac p2}dx\Big)^\frac kp \\
&= C\|\rho_\eps -\rho_\eta\|_{L^2}^k \\
&\le C \eps^{-2k} |\eps-\eta|^k,
\end{split}
\end{align}
where $C>0$ depends on $k$ and we have used \eqref{etaepseta} at the last step.
Then, by combining \eqref{wicK}, \eqref{pkmin}, and \eqref{Yzeta}, using Cauchy-Schwarz and \cite[Lemma 2.7]{DM}, we have that for any $k$ large enough,
 \begin{equation}\label{monteccv}
 \E\left( \|\nabla Y_\eps^2-\nabla Y_\eta^2\|_{\mathcal{C}^{\alpha-1}_{-\delta}}^k\right) \le C \eps^{-(2+\beta+\zeta)k} |\eps-\eta|^k.
 \end{equation}
By combining \eqref{monteccv} with \eqref{e1.15} and recalling \eqref{WICK:}, we get
 \begin{equation}\label{e1.16}
  \E\left( \|\Wick{\nabla Y_\eps^2}-\Wick{\nabla Y_\eta^2}\|_{\mathcal{C}^{\alpha-1}_{-\delta}}^k\right) \le C \eps^{-(2+\beta+\zeta)k} |\eps-\eta|^k.
\end{equation}
On the other hand, by \eqref{e1.14bis}, we have
\begin{equation}\label{e1.17}
  \E\left( \|\Wick{\nabla Y_\eps^2}-\Wick{\nabla Y_\eta^2}\|_{\mathcal{C}^{\alpha-1}(\w^{-\delta})}^k\right) \le C \eta^{k\kappa_0}, \quad 0<\varepsilon<\eta < \frac 12.
\end{equation}
Let us now consider several cases. 

\smallskip \noindent
{\bf Case 1:} $2\eps<\eta$.
We have necessarily $\eta< 2|\eps-\eta|$ and by  \eqref{e1.17},
$$
  \E\left( \|\Wick{\nabla Y_\eps^2}-\Wick{\nabla Y_\eta^2}\|_{\mathcal{C}^{\alpha-1}(\w^{-\delta})}^k\right) \le C |\eps-\eta|^{k\kappa_0}\le C  |\eps-\eta|^{\frac{k\kappa_0}{2+\kappa_0+\beta+\zeta}}.
$$
{\bf Case 2:} $\eps<\eta\le 2\eps$ and $\eps< |\eps-\eta|^{\frac 1{\kappa_0+2+\beta+\zeta}}$.
Then, again by  \eqref{e1.17} we get
$$
 \E\left( \|\Wick{\nabla Y_\eps^2}-\Wick{\nabla Y_\eta^2}\|_{\mathcal{C}^{\alpha-1}(\w^{-\delta})}^k\right) \le C\eps^{k\kappa_0}\le  C|\eps-\eta|^{\frac{k\kappa_0}{\kappa_0+2+\beta+\zeta}}.
$$
{\bf Case 3:} $\eps<\eta\le 2\eps$ and $\eps\ge  |\eps-\eta|^{\frac 1{\kappa_0+2+\beta+\zeta}}$.
In this case, we use \eqref{e1.16} to get
$$  
\E\left( \|\Wick{\nabla Y_\eps^2}-\Wick{\nabla Y_\eta^2}\|_{\mathcal{C}^{\alpha-1}(\w^{-\delta})}^k\right) \le C\eps^{-(2+\beta+\zeta)k} |\eps-\eta|^k \le C  |\eps-\eta|^{\frac{k\kappa_0}{2+\kappa_0+\beta+\zeta}}.
$$
Summarizing, we get
$$
  \E\left( \|\Wick{\nabla Y_\eps^2}-\Wick{\nabla Y_\eta^2}\|_{\mathcal{C}^{\alpha-1}_{-\delta}}^k\right) \le C  |\eps-\eta|^{\frac{k\kappa_0}{2+\kappa_0+\beta+\zeta}},
  \quad \forall \eps,\, \eta \in \Big( 0, \frac 12 \Big).
$$
It remains to choose $k$ large enough so that $\frac{k\kappa_0}{2+\kappa_0+\beta+\zeta}>1$ and we may invoke Kolmogorov continuity criterion (see \cite[Theorem 3.3]{DPZa14}) to deduce that 
$\eps\mapsto \Wick{\nabla Y_\eps^2}$ from $[0,1]$ to $\mathcal{C}^{\alpha-1}_{-\delta}$ is almost surely H\"older continuous of exponent arbitrarily less than $\frac{\kappa_0}{2+\kappa_0}-\frac 1k$ on $[0,1]$. The proof is complete. 
\end{proof}

\section{Linear estimates}
\label{sec:lin}

We introduce the propagator $S_{A,V}(t)$ associated with
\begin{equation}\label{LS}\ii \partial_t w=\Delta w- 2\nabla A\cdot \nabla w+V w.\end{equation}
We also denote for shortness 
\begin{equation}\label{eq:HAV}H_{A,V}=\Delta - 2\nabla A\cdot \nabla +V.\end{equation}
In the sequel we shall assume that $A$ and $V$ satisfy:
\begin{equation}\label{exponentialAV} 
\forall \delta>0, ~ \exists C>0 \hbox{ s.t. }
\|\langle x \rangle^{-\delta} e^{\pm A}\|_{L^\infty}\leq C;\end{equation}
\begin{equation}\label{functionalDMAV}
\forall \delta>0, ~ \exists C>0 \hbox{ s.t. } \Big |\int_{\R^2} V |\varphi|^2 e^{-2A} dx\Big | 
\leq C \|\varphi\|_{H^\frac 12_{\delta}}^2, \quad \forall \varphi\in H^\frac 12_{\delta}.
\end{equation}
Under the assumption \eqref{exponentialAV}, we have that by \eqref{eq:PullWeight} and classical elliptic regularity,
\begin{equation}\label{equivH2A} 
\forall \delta>0, ~ \exists c,C>0 \hbox{ s.t. } c\|u\|_{H^2_{-\delta}}\leq \|\Delta u e^{-A}\|_{L^2}+ \|u e^{-A}\|_{L^2}\leq C \|u\|_{H^2_\delta}.
\end{equation}
Here, we emphasize that the constants $c$ and $C$ in \eqref{equivH2A} depend only on the constant $C$ in \eqref{exponentialAV}.
It is easy to check that any solution to \eqref{LS} satisfies the following conservation laws:
\begin{equation}\label{LSL2}\frac{\dd}{\dd t} \int_{\R^2} |w(t)|^2 e^{-2A} dx=0;\end{equation}
\begin{equation}\label{LSH1}\frac {\dd }{\dd t} \int_{\R^2} \Big( |\nabla w(t)|^2 e^{-2A} - V |w(t)|^2 e^{-2A} \Big) dx=0.\end{equation}
Next we associate to any couple $(A,V)$ the following quantity for any given $\delta>0$, $r\in (2,\infty)$:
\begin{align}
\label{AVdeltar}
\begin{split}
|(A,V)|_{\delta, r} &=
\|\w^{-\delta} \nabla A e^{-A}\|_{L^r} 
+ \|\w^{-2\delta } \nabla A e^{-2A} V\|_{L^\frac r2} +\|\w^{-\delta}V e^{-A}\|_{L^r} \\
&\quad +\|\w^{-\delta}e^{-(p+2)A}\|_{L^\infty}+
\|\w^{-\delta} V e^{-(p+2)A}\|_{L^r} \\
&\quad +\|\w^{-\delta}\nabla A e^{-(p+2)A}\|_{L^r}
+ \|\w^{-\delta} e^{-(p+1)A}\|_{L^\infty} \\
&\quad + \|\w^{-\delta} \nabla A e^{-(p+1)A}\|_{L^r} + \|\w^{-\delta} e^{-pA}\|_{L^\infty} \\
&\quad + \|\w^{-\delta} \nabla A e^{-pA}\|_{L^r}.
\end{split}
\end{align}

\subsection{Linear Energy Estimates}

We first prove some $L^2$ weighted estimates for the propagator $S_{A,V}(t)$ associated with \eqref{LS}. 
\begin{prop}
\label{propkey} Assume $A, V$ satisfy \eqref{exponentialAV} and \eqref{functionalDMAV}.

\smallskip \noindent
\textup{(i)} For every $\delta > 0$, we have
\begin{equation}\label{linconserL2}
\|S_{A,V}(t)\varphi\|_{L^\infty((0,\infty);L^2_{-\delta})}\leq  C \| \varphi\|_{L^2_{\delta}}.
\end{equation}
\smallskip \noindent
\textup{(ii)} For every $T > 0$ and $0 < \delta < \delta^+$ satisfying $\frac \delta 2+2\delta^+<1$, we have
\begin{equation}\label{linconserL2weight}
\|S_{A,V}(t)\varphi\|_{L^\infty((0,T); L^2_{\delta})}\leq C \|\varphi\|_{H^1_{\delta^+}}.
\end{equation}
\smallskip \noindent
\textup{(iii)} For every $T > 0$, $s \in (0, 1)$, and $0 < \delta < \delta^+$ satisfying $\delta+9\delta^+<4s$, we have
\begin{equation}\label{L2Hs} 
\|S_{A,V}(t)\varphi\|_{L^\infty((0,T); L^2_{\delta})}\leq C \|\varphi\|_{H^{s^+}_{\frac{\delta^+}s}}.
\end{equation}
\smallskip \noindent
\textup{(iv)} For every $T > 0$, $r \in (2, \infty)$, and $0 < \delta < \delta^+$ satisfying $\frac \delta 2 + 2\delta^+<\frac{r-2}{3r+2}$, we have
\begin{equation}\label{H2} 
\|S_{A,V}(t)\varphi\|_{L^\infty((0, T);H^2_{-\delta})}\leq {\mathcal P} \Big (|(A,V)|_{\delta,r}\Big)\|\varphi\|_{H^2_{(\frac{3r+2}{r-2})\delta^+ }}.
\end{equation}

\end{prop}

\begin{proof}
We denote for simplicity $w(t)=S_{A,V}(t)\varphi$. The conservation of mass
\eqref{LSL2}
along with \eqref{exponentialAV} imply \eqref{linconserL2}.
In order to prove \eqref{linconserL2weight} we rely on \eqref{LSL2} and \eqref{LSH1}.
After integration in time, by recalling \eqref{functionalDMAV} with $\chi=\frac{\delta}4$,
we get the following bound:
\begin{align*} 
\| &e^{-A} w(t)\|_{L^2}^2+
\|e^{-A}\nabla w(t)\|_{L^2}^2 \\
&\leq \|e^{-A} w(0)\|_{L^2}^2 +\|e^{-A}\nabla w(0)\|_{L^2}^2 + C \|w(0)\|_{H^\frac 12_{\frac{\delta}4}}^2
+ C \|w(t)\|_{H^\frac 12_{\frac{\delta}4}}^2 \\
&\leq C \|w(0)\|_{H^1_{\frac{\delta}2}}^2
+ C \|w(t)\|_{L^2_{\delta}}\|w(t)\|_{H^1_{-\frac{\delta}2}},
\end{align*}
where we used interpolation and \eqref{exponentialAV}.
By  using again \eqref{exponentialAV} we get 
the bound
\begin{equation}\label{epsilon}\|w(t)\|_{H^1_{-\frac{\delta}2}}\leq C
\|w(0)\|_{H^1_{\delta^+}} + C\|w(t)\|_{L^2_{\delta}}.\end{equation}
Next, by fixing $\bar \delta\in (\delta, \delta^+)$ and following the proof of Lemma 3.1 in \cite{DM}, we get
\begin{align}
\label{latersee}
\begin{split}
\frac{\dd}{\dd t} &\int_{\R^2}  \w^{2\bar \delta}|w(t)|^2 e^{-2A} \\
&=2\Re {\int_{\R^2} \w^{2\bar \delta} \,  \partial_t w(t) \overline{w}(t)\, e^{-2A}} dx\\
&= 2\,\Im{\int_{\R^2} \w^{2\bar \delta}(\Delta w(t) -2 \D w (t)\cdot  \D A\, )\bar w(t)\,e^{-2A}} dx \\
&=-2\,\Im{\int_{\R^2} \D\w^{2\bar \delta}\cdot \D w(t)\,\overline{w} (t) e^{-2A}} dx \\
&\leq C \,\int_{\R^2} \w^{2\bar \delta-1} |\D w(t)| \,|w(t)| \,e^{-2A} dx.
\end{split}
\end{align}
After integration in time, by using \eqref{exponentialAV} and the Cauchy-Schwartz inequality, we get
\begin{align}
\label{***}
\begin{split}
\|w(t) e^{-A}\|_{L^2_{\bar \delta}}^2 &\leq \|w(0) e^{-A}\|_{L^2_{\bar \delta}}^2
+C \int_0^t \|\nabla w(\tau)\|_{L^2_{2\bar \delta-1}} \|w(\tau)\|_{L^2_{\delta}}d\tau \\
&\leq  \|w(0) e^{-A}\|_{L^2_{\bar \delta}}^2
+C\int_0^t \|\nabla w(\tau)\|_{L^2_{-\frac{\delta}2}}^2 d\tau+ C \int_0^t \|w(\tau)\|_{L^2_{\delta}}^2d\tau,
\end{split}
\end{align}
where we used the condition $\frac \delta 2+2\delta^+<1$ in order to guarantee $\langle x\rangle^{2\bar \delta- 1}\leq \langle x\rangle^{-\frac{\delta}2}$.
By inserting in \eqref{***} the estimate \eqref{epsilon} and by using again \eqref{exponentialAV}, we deduce
$$\|w(t)\|_{L^2_{\delta}}^2\leq C (1+T) \|w(0)\|_{H^1_{\delta^+}}^2
+ C \int_0^t \|w(\tau)\|_{L^2_{\delta}}^2 d\tau$$
and we conclude by the Gronwall inequality. 

Next, we focus on \eqref{L2Hs} and we fix
$\eta, \mu>0$ such that
\begin{equation}\label{necDTV}
s \eta-(1-s)\mu=\delta, \quad \eta=\frac{\delta}s+\frac{(\delta^+-\delta)}{2s}.
\end{equation} 
Then, for every $t\in [0,T]$ we use interpolation to obtain
\begin{align*}
\|S_{A,V}(t) \varphi\|_{L^2_{\delta}} &\leq 
\sum_{{N-{\rm dyadic}}} \|S_{A,V} (t) \Delta_N \varphi\|_{L^2_{\delta}}\\
&\leq C
\sum_{{N-{\rm dyadic}}} \|S_{A,V}(t) \Delta_N \varphi\|_{L^2_{\eta}}^{s} \|S_{A,V}(t) \Delta_N \varphi\|_{L^2_{-\mu}}^{1-s} \\
&\leq C \sum_{{N-{\rm dyadic}}} \|S_{A,V}(t) \Delta_N \varphi\|_{L^2_{\eta}}^{s} \|\Delta_N \varphi\|_{L^2_{\frac \delta s}}^{1-s},
\end{align*}
where we have used \eqref{linconserL2} at the last step. By \eqref{linconserL2weight} (that we can use thanks to the conditions 
$\delta+9\delta^+<4s$ and $\delta^+>\delta$), we continue the estimate above as follows
\begin{align*}
\dots &\leq C \sum_{{N-{\rm dyadic}}} \|\Delta_N \varphi\|_{H^1_{\frac{\delta^+}{s}}}^{s} \|\Delta_N \varphi\|_{L^2_{\frac \delta s}}^{1-s} \\
&\leq C \sum_{{N-{\rm dyadic}}} N^s \|\Delta_N \varphi\|_{L^2_{\frac{\delta^+}s}}^{s} \|\Delta_N \varphi\|_{L^2_{\frac {\delta}{s}}}^{1-s}
\leq C \|\varphi\|_{H^{s^+}_{\frac{\delta^+}s}},
\end{align*}
where we used \eqref{sumupLP} and \eqref{LPtins}.

Concerning \eqref{H2}, we need to use \eqref{MODIFIED} (recall $\lambda=0$ since we are working with the linear flow) to obtain
$$\frac{\dd}{\dd t} \int_{\R^2} \Big(|\Delta w(t)|^2 e^{-2A} + \mathcal F_{A,V}(w(t))\Big) dx=0,$$
which in turn after integration in time implies 
$$\int_{\R^2} |\Delta w(t)|^2 e^{-2A} dx=\int_{\R^2} |\Delta w(0)|^2 e^{-2A}dx
- {\mathcal F}_{A,V}(w(t))+{\mathcal F}_{A,V}(w(0)).$$
By \eqref{intermF} (whose proof only involves the H\"older inequality) in conjunction with \eqref{LSL2},
we get that for every $\mu>0$,
\begin{align*}
\int_{\R^2} &(|w(t)|^2+ |\Delta w(t)|^2) e^{-2A} dx\\
&\leq \int_{\R^2} (|w(0)|^2 + |\Delta w(0)|^2) e^{-2A}dx\\
&\quad +
\mathcal P\Big( |(A,V)|_{\delta,r}\Big) \Big (\mu \|e^{-A} \Delta w(t)\|_{L^2}^2+ (1+ \frac 1{2\mu}) \| w(t)\|_{W_\delta^{1,\frac{2r}{r-2}}}^2\Big)\\
&\quad +\mathcal P\Big(|(A,V)|_{\delta,r}\Big) \Big (\|e^{-A} \Delta w(0)\|_{L^2} \|w(0)\|_{W_\delta^{1,\frac{2r}{r-2}}}
+\|w(0)\|_{W_\delta^{1,\frac{2r}{r-2}}}^2\Big).
\end{align*}
In particular, we can choose $\mu=\frac 1{2\mathcal P(|(A,V)|_{\delta,r})}$ and
get
\begin{align}
\label{mu}
\begin{split}
\frac 12 &\int_{\R^2} (|w(t)|^2+ |\Delta w(t)|^2) e^{-2A} dx \\
&\leq \int_{\R^2} (|w(0)|^2 + |\Delta w(0)|^2) e^{-2A}dx +
\mathcal P\Big (|(A,V)|_{\delta,r}\Big ) \|w(t)\|_{W_\delta^{1,\frac{2r}{r-2}}}^2\\
&\quad +\mathcal P\Big (|(A,V)|_{\delta,r}\Big) \Big (\|e^{-A} \Delta w(0)\|_{L^2} \|w(0)\|_{W_\delta^{1,\frac{2r}{r-2}}}
+\|w(0)\|_{W_\delta^{1,\frac{2r}{r-2}}}^2\Big)
\end{split}
\end{align}
Then, notice that by \eqref{*****}, where we choose $q=\frac{2r}{r-2}$, we get
\begin{equation}\label{tinmac}
\|w(t)\|_{W^{1,\frac{2r}{r-2}}_\delta}
\leq C \|w(t)\|_{H^2_{-\delta }}^\frac{r+2}{2r} \|w(t)\|_{L^2_{(\frac{3r+2}{r-2})\delta }}^{\frac {r-2}{2r}} 
\leq C \|w(t)\|_{H^2_{-\delta }}^\frac{r+2}{2r} \|w(0)\|_{H^1_{(\frac{3r+2}{r-2})\delta^+ }}^{\frac {r-2}{2r}},
\end{equation}
where we used at the last step \eqref{linconserL2weight} (recall our assumptions $\frac \delta 2 + 2\delta^+<\frac{r-2}{3r+2}$
and $\delta^+>\delta$). We conclude by combining \eqref{mu}, \eqref{tinmac}
and \eqref{equivH2A}.
\end{proof}


\subsection{Linear Strichartz Estimates}
In this subsection, we shall need the following norm associated with $A,V$:
\begin{equation}\label{AVdeltakappa}
\|(A,V)\|_{\delta,k}=\|V\|_{L^{k}_{-\delta}} +
 \|\nabla A \|_{L^{k}_{-\delta}}, \quad  k\in [1, \infty), \quad \delta>0.
\end{equation}
We start with some useful lemmas.
\begin{lemma}
For every $s\in (0,1)$, there exists $C>0$ such that
\begin{equation}\label{bound1} 
\|(H_{A,V} -\Delta) u\|_{L^2}\leq C  \|(A,V)\|_{\delta, \frac 2 s} \| u\|_{H^{1+s}_\delta}, \quad \forall \delta>0.
\end{equation} 
\end{lemma}
\begin{proof} 
We have for every $s\in (0,1)$ and $\delta>0$,
$$\|\nabla u  \cdot \nabla A\|_{L^2}\leq \|\nabla A\|_{L^{\frac 2s}_{-\delta}} \|\nabla u\|_{L^{\frac 2{1-s}}_\delta}\leq C  \|\nabla A\|_{L^{\frac 2s}_{-\delta}}
\|u\|_{H^{1+s}_\delta},
$$
where we have used the embedding $H^s_\delta \subset L^{\frac 2{1-s}}_\delta$ (see \eqref{Sobolevweighted})
along with \eqref{commutator2} and \eqref{eq:PullWeight}.
By a similar argument,
$$
\|V u\|_{L^2}\leq C \|V\|_{L^{\frac 2 s}_{-\delta}}  \|u\|_{H^{s}_{\delta}}.
$$
The proof is complete.
\end{proof}

\begin{lemma} Let $T>0$ be fixed and $A,V$ satisfy \eqref{exponentialAV} and \eqref{functionalDMAV}. There exists $\bar \delta>0$ such that for $\delta\in (0, \bar \delta)$ and for every $r\in [4, \infty)$, we have
\begin{equation}\label{key}\|S_{A,V}(t)\varphi\|_{L^\infty((0,T); H^{1+\delta}_{\frac \delta 2(1-3\delta)})}\leq
{\mathcal P} \Big (|(A,V)|_{\delta,r}\Big ) \|\varphi\|_{H_{4\sqrt{\delta}}^{\frac{\sqrt{\delta}(1-\delta)}{\sqrt 2}+1+\delta^+}}.
\end{equation}
\end{lemma}
\begin{proof}
By \eqref{L2Hs} (where we choose $s=\sqrt \delta$), there exists $\bar \delta>0$ such that  for $0<\delta<\bar \delta$, we have
\begin{equation}\label{intil}\|S_{A,V}(t)\varphi\|_{L^\infty((0,T); L^2_{2\delta})}\leq C \|\varphi\|_{H^{\sqrt {2\delta}}_{4\sqrt {\delta}}}.
\end{equation}
From \eqref{H2}, under the extra assumption $r\geq 4$, we get
\begin{equation}\label{tinil}\|S_{A,V}(t)\varphi\|_{L^\infty((0, T);H^2_{-\delta})}\leq {\mathcal P} \Big (|(A,V)|_{\delta,r}\Big )\|\varphi\|_{H^2_{8\delta }}.\end{equation}
Next, we notice that by special case of \eqref{lem:BesovInterpolation} and by \eqref{intil}, \eqref{tinil} we get:
\begin{align*}
\| &S_{A,V}(t)\varphi\|_{L^\infty((0,T); H^{1+\delta}_{\frac \delta 2(1-3\delta)})} \\
&\leq \sum_{{N-{\rm dyadic}}} \|S_{A,V}(t) \Delta_N \varphi\|_{L^\infty((0,T); H^{1+\delta}_{\frac \delta 2(1-3\delta)})}\\\nonumber&
\leq \sum_{{N-{\rm dyadic}}}
\|S_{A,V}(t)\Delta_N \varphi\|_{L^\infty((0,T); L^2_{2\delta})}^\frac{1-\delta}2 \|S_{A,V}(t)\Delta_N \varphi\|_{L^\infty((0, T);H^2_{-\delta})}^\frac{1+\delta}2
\\\nonumber& \leq 
{\mathcal P} \Big (|(A,V)|_{\delta,r}\Big ) \sum_{{N-{\rm dyadic}}} \|\Delta_N \varphi\|_{H^{\sqrt {2\delta}}_{4\sqrt {\delta}}}^\frac{1-\delta}2 \|\Delta_N \varphi\|_{H^2_{8\delta}}^\frac{1+\delta}2
\\\nonumber&\leq {\mathcal P} \Big (|(A,V)|_{\delta,r}\Big ) \sum_{{N-{\rm dyadic}}} N^{\frac{\sqrt{\delta}(1-\delta)}{\sqrt 2}+1+\delta} \|\Delta_N \varphi\|_{L^2_{4\sqrt {\delta}}}.
\end{align*}
The conclusion follows by \eqref{sumupLP}.
\end{proof}

\begin{lemma} Let $T>0$ be fixed and $A,V$ satisfy \eqref{exponentialAV} and \eqref{functionalDMAV}.
There exists $\bar \delta>0$ such that for every $\delta\in (0, \bar \delta)$ and $r\in [4, \infty)$, we have
\begin{equation}\label{3/2}
\|S_{A,V}(t) \varphi\|_{L^\infty((0,T);H^{\frac 32})}\leq  {\mathcal P}\Big (|(A,V)|_{\delta, r}\Big ) \|\varphi\|_{H^{\frac 32+\frac{\sqrt \delta^+}4}_{\sqrt \delta^+}}.
\end{equation}
\end{lemma}
\begin{proof} By combining a special case of \eqref{lem:BesovInterpolation}, \eqref{H2} and \eqref{L2Hs} (here we assume $\delta>0$ small enough in order to guarantee $\delta+9\delta^+<4\sqrt \delta$, and hence we can apply \eqref{L2Hs}
with $s=\sqrt \delta$), we get
\begin{align*}
\| &S_{A,V}(t) \varphi\|_{H^{\frac 32}}\\
&\leq  
\sum_{{N-{\rm dyadic}}} \|S_{A,V}(t) \Delta_N \varphi\|_{L^\infty((0,T);H^{\frac 32})}\\\nonumber&\leq C \sum_{{N-{\rm dyadic}}} \|S_{A,V}(t) \Delta_N 
\varphi\|_{L^\infty((0,T);H^{2}_{-\frac \delta 3})}^\frac 34
\|S_{A,V}(t) \Delta_{N}\varphi\|_{L^\infty((0,T);L^{2}_\delta)}^\frac 14
\\\nonumber& \leq {\mathcal P}\Big(|(A,V)|_{\delta, r}\Big )\sum_{{N-{\rm dyadic}}}  \|\Delta_N \varphi\|_{H^{2}_{8\delta}}^\frac{3}4
\|\Delta_N \varphi\|_{H^{\sqrt \delta^+}_{\sqrt \delta^+}}^\frac {1}4\\\nonumber&
\leq {\mathcal P}\Big (|(A,V)|_{\delta, r}\Big )\sum_{{N-{\rm dyadic}}} N^{\frac 32+\frac {\sqrt \delta^+}4} \|\Delta_N \varphi\|_{L^{2}_{\sqrt \delta^+}}.
\end{align*}
The conclusion follows by \eqref{sumupLP}.
\end{proof}

We are now ready to prove the Strichartz estimates associated with $S_{A,V}(t)$.

\begin{prop}\label{strichob}
Let $T>0$ be fixed and $A,V$ satisfy \eqref{exponentialAV} and \eqref{functionalDMAV}. For every $\delta,s>0$, there exist $\tilde \delta, \delta_1, s_1>0$ such that 
$\frac{\delta_1}{s_1}>1$ and for every $r\in [4,\infty)$,
\begin{equation}\label{strichnai}
\|S_{A,V}(t)\varphi\|_{L^l((0,T);L^q)}\\\leq {\mathcal P} \Big (|(A,V)|_{\tilde \delta,r}\Big ) \mathcal P \Big (\|(A,V)\|_{\delta_1, \frac 2{s_1}} \Big)  \|\varphi\|_{H^{\frac 1l+s}_{\delta}},
\end{equation}
where $\frac 1l+\frac 1q=\frac 12$ and $l>2$.
\end{prop}

\begin{proof}
Notice that by \eqref{key}, for every $\delta_2, s_2>0$, there exist $\tilde \delta, \delta_1, s_1>0$ such that
for every $r\in [4,\infty)$,
\begin{equation}\label{adria}\|S_{A,V}(t) \varphi\|_{L^\infty((0,T);H^{1+s_1}_{\delta_1})}\leq {\mathcal P}\Big(|(A,V)|_{\tilde \delta,r}\Big) \|\varphi\|_{H^{1+\frac{s_2}2}_{\delta_2}}.\end{equation}
Moreover, by \eqref{inclusionHsdelta} we can also assume
$\frac{\delta_1}{s_1}$ to be arbitrarily large and in particular larger than $1$.
We also have the following bound
for any given $s_2, \delta_2>0$ as a consequence of \eqref{L2Hs}:
\begin{equation}\label{adria1}
\|S_{A,V}(t) \varphi\|_{L^\infty((0,T);L^2)}\leq C\|\varphi\|_{H^{\frac{s_2}2}_{\delta_2}}.
\end{equation}

Next, following \cite{BGT, KTz} we split the interval $[0,T]$ in an essentially disjoint union of intervals of size $ N^{-1}$ as
\begin{equation}\label{spli}
[0,T]=\bigcup_j I_j
\end{equation}
and we aim to estimate 
$
\|S_{A,V}(t)\Delta_{N}\varphi\|_{L^l(I_j;L^q )}\, .
$
Suppose that $I_j=[a,b]$. Then, for $t\in [a,b]$ we can write:
\begin{align}
\label{decomposition}
\begin{split}
S_{A,V}(t)\Delta_{N}\varphi
&=
e^{i(t-a)\Delta} S_{A,V}(a)\Delta_{N}\varphi
\\
&\quad +
i\int_a^t e^{i(t-\tau)\Delta} (H_{A,V} -\Delta) S_{A,V}(\tau) \Delta_{N}\varphi d\tau.
\end{split}
\end{align}
We now estimate each term in the r.h.s. of \eqref{decomposition}. Using the Strichartz estimates on $\R^2$ (see \cite{Caz03, GV, KT}),
and \eqref{adria1}, we get
\begin{align*}
\| &e^{i(t-a)\Delta} S_{A,V}(a)\Delta_{N}\varphi\|_{L^l(I_j;L^q)}\\
&\leq C 
\|S_{A,V}(a)\Delta_{N}\varphi\|_{L^2} \\
&\leq C\|\Delta_N \varphi\|_{H^{\frac{s_2}2}_{\delta_2}}\\
&\leq C N^{-\frac 1l-\frac{s_2}2}
(\|\Delta_{\frac N2}\varphi\|_{H^{\frac 1l+s_2}_{\delta_2}}+\|\Delta_{N}\varphi\|_{H^{\frac 1l+s_2}_{\delta_2}}+\|\Delta_{2N}\varphi\|_{H^{\frac 1l+s_2}_{\delta_2}})
\\\nonumber&\leq C N^{-\frac 1l-\frac{s_2}2}
\|\varphi\|_{H^{\frac 1l+s_2}_{\delta_2}}
\end{align*}
where we used \eqref{LPtinsnewnew} at the last step.
Now we estimate the second term in the r.h.s. of \eqref{decomposition}.
Using Minkowski's inequality, the Strichartz estimates on $\R^2$, \eqref{bound1}, and \eqref{adria} (with $s_2$ small), we get
\begin{align*}
\Big \| &\int_a^t e^{i(t-\tau)\Delta} (H_{A,V} -\Delta) S_{A,V}(\tau) \Delta_{N}\varphi d\tau \Big \|_{L^l(I_j;L^q)}
\\\nonumber&
\leq C
\int_{I_j}
\|
(H_{A,V} -\Delta) S_{A,V}(\tau) \Delta_{N}\varphi\|_{L^2} d\tau\\\nonumber&\leq
 C \|(A,V)\|_{\delta_1, \frac 2{s_1}}  {\mathcal P} \Big(|(A,V)|_{\tilde \delta,r}\Big)
N^{-1} \|\Delta_{N}\varphi\|_{H^{1+\frac{s_2}2}_{\delta_2}}\\\nonumber&
\leq C \|(A,V)\|_{\delta_1, \frac 2{s_1}}  {\mathcal P} \Big(|(A,V)|_{\tilde \delta,r}\Big)
N^{-1}
N^{1-\frac{1}{l}-\frac{s_2}2}
 \|\varphi\|_{H^{\frac 1l+s_2}_{\delta_2}},
\end{align*}
where we used \eqref{LPtinsnew}.
Summarizing, we get
\begin{align*}
\| &S_{A,V}(t)\Delta_{N}\varphi\|_{L^l(I_j;L^q)}\\
& \leq 
C    {\mathcal P} \Big(|(A,V)|_{\tilde \delta,r}\Big) \Big (1+\|(A,V)\|_{\delta_1, \frac 2{s_1}} \Big )
N^{-\frac 1l-\frac{s_2}2}
 \|\varphi\|_{H^{\frac{1}{l}+s_2}_{\delta_2}}.
\end{align*}
Using that the number of $I_j$ in \eqref{spli} is $O(TN)$, taking the $l$'th power of the previous bound, and summing over $j$, we get the estimate 
\begin{align*}
\| &S_{A,V}(t)\Delta_{N}\varphi\|_{L^l((0,T);L^q)}\\
& \leq C T^\frac 1l  {\mathcal P} \Big(|(A,V)|_{\tilde \delta,r}\Big) \Big (1+ \|(A,V)\|_{\delta_1, \frac 2{s_1}}
\Big ) N^{-\frac{s_2}2}\|\varphi\|_{H^{\frac{1}{l}+s_2}_{\delta_2}}
\end{align*}
and hence we conclude by summation over $N$.
\end{proof}

In the sequel, we shall need the following Strichartz estimate.
\begin{prop}\label{strichcor} Assume \eqref{exponentialAV} and \eqref{functionalDMAV} and
let $T>0$ be fixed. For every $s, \delta>0$ small enough, there exist $s_1, \delta_1, \tilde \delta>0$ such that 
$\frac{\delta_1}{s_1}>1$ and for every $r\in [4,\infty)$, we have
\begin{align}
\label{noDuhamel}
\begin{split}
\|&S_{A,V}(t)\varphi\|_{L^4((0,T);W^{\frac 34-s,4})}\\
&\leq {\mathcal P} \Big(|(A,V)|_{\frac{\delta^2}4,r}\Big) 
 {\mathcal P} \Big(|(A,V)|_{\tilde \delta,r}\Big){\mathcal P}\Big(\|(A,V)\|_{\delta_1, \frac 2{s_1}}\Big)  \|\varphi\|_{H^1_{\delta}}
\end{split}
\end{align}
and
\begin{align}
\label{Duhamel}
\begin{split}
\Big \| &\int_0^t S_{A,V}(t-\tau) f(\tau) d\tau\Big \|_{L^4((0,T);W^{\frac 34-s,4})}\\
&\leq {\mathcal P} \Big(|(A,V)|_{\frac{\delta^2}4,r}\Big) 
 {\mathcal P} \Big(|(A,V)|_{\tilde \delta,r}\Big){\mathcal P}\Big(\|(A,V)\|_{\delta_1, \frac 2{s_1}}\Big)  \|f\|_{L^1((0,T);H^1_{\delta})}\, .
\end{split}
\end{align}
\end{prop}
\begin{proof}
Notice that it is not restrictive to assume $\delta, s$ small enough (we will exploit this fact later). Notice also that \eqref{Duhamel} follows by combining \eqref{noDuhamel} with Minkowski's inequality.
Thus, we focus on the proof of \eqref{noDuhamel}.

For every $s\in (0,\frac 12)$, there exists $q\in (2, \infty)$ such that 
the following Gagliardo-Nirenberg inequality holdss:
$$\|u\|_{W^{\frac 34-s,4}}\leq C \|u\|_{L^q}^{\frac 12} \|u\|_{H^\frac 32}^{\frac 12},$$
and hence by integration in time and the H\"older inequality in time we get
\begin{align*}
\|S_{A,V}(t)\varphi\|_{L^4((0,T);W^{\frac 34-s,4})}^4 &\leq C \|S_{A,V}(t)\varphi\|_{L^{2}((0,T);L^q)}^{2} 
\|S_{A,V}(t)\varphi\|_{L^\infty((0,T);H^\frac 32)}^{2}
\\\nonumber&
\leq C {\mathcal P}\Big(|(A,V)|_{\frac{\delta^2}4, r}\Big) 
 \|S_{A,V}(t)\varphi\|_{L^{l}((0,T);L^q)}^{2} \|\varphi\|_{H^{\frac 32+(\frac{\delta}8)^+}_{(\frac \delta 2)^+}}^2,
\end{align*}
where $\frac 1l+\frac 1q=1, l>2$ are Strichartz admissible and we used \eqref{3/2} (where $r\in [4,\infty)$ is arbitrary and we have replaced 
$\delta$ by $\frac{\delta^2}4$).
By using the Strichartz estimates \eqref{strichnai},
we can continue the estimate above as follows
$$\dots \leq {\mathcal P}\Big(|(A,V)|_{\frac{\delta^2}4, r}\Big)  {\mathcal P} \Big(|(A,V)|_{\tilde \delta,r}\Big) \mathcal P
\Big (|(A,V)\|_{\delta_1, \frac 2{s_1}} \Big)  \|\varphi\|_{H^{\frac 1l+s}_{\delta}}^2
\|\varphi\|_{H^{\frac 32+(\frac{\delta}8)^+}_{(\frac \delta 2)^+}}^2,$$
where $\tilde \delta, \delta_1, s_1>0$ depend on $s, \delta$.
Notice that for initial datum $\varphi=\Delta_N \varphi$ which is localized at dyadic frequency $N$, we get from the previous bound that
\begin{align*}
\| &S_{A,V}(t)\Delta_N \varphi\|_{L^4((0,T);W^{\frac 34-s,4})} \\\nonumber&
\leq  {\mathcal P} \Big(|(A,V)|_{\frac{\delta^2}4,r}\Big)  {\mathcal P} \Big(|(A,V)|_{\tilde \delta,r}\Big)\mathcal P\Big (\|(A,V)\|_{\delta_1, \frac 2{s_1}}\Big) 
\|\Delta_N \varphi\|_{H^{\frac 1{2l}+\frac s2 +\frac 34+(\frac{\delta}{16})^+}_{\delta}}.
\end{align*}
We conclude \eqref{noDuhamel} by summing over $N$ and using \eqref{sumupLP}, once we notice that $\frac 1{2l}+\frac 34<1$ for $l>2$ and
$s, \delta>0$ are small enough.
\end{proof}

\section{Nonlinear estimates}
\label{sec:nonlin}

Along this section, we focus on solutions to the following nonlinear problem
\begin{equation}\label{NLSAV}
\ii \partial_t v= H_{A,V} v -\lambda e^{-pA} v|v|^p, \quad \lambda>0
\end{equation}
where $H_{A,V}$ is defined in \eqref{eq:HAV}.
\subsection{Nonlinear Energy Estimates}
We have the following conserved quantities for any solution to \eqref{NLSAV}:
\begin{equation}
\label{eq:MollifiedQuantitiesNAV}
\frac{\dd}{\dd t} \int_{\R^2}| v(t)|^2 e^{-2A} dx=0;
\end{equation}
\begin{equation}\label{eq:MollifiedQuantitiesHAV}
\frac{\dd}{\dd t} \int_{\R^2} \Big(\frac{1}{2}|\D v|^2e^{-2A}-\frac{1}{2}|v(t)|^2  V e^{-2A}
+\frac{\lambda}{p+2} |v(t) |^{p+2} e^{-(p+2)A}\Big) dx=0.
\end{equation}
\begin{prop} \label{prop:L2H1}
Assume \eqref{exponentialAV} and \eqref{functionalDMAV} and let $T>0$ be fixed. Then,
for every $\delta\in (0, \frac 19)$, there exists $C>0$ such that for every solution $v$ to \eqref{NLSAV}, we have
\begin{equation}\label{nonloin}\|v\|_{L^\infty((0,T);L^2_\delta)}\leq  C \Big (1+\|v(0)\|_{H^1_{4\delta}}^{\frac{p+2}2}\Big );\end{equation}
\begin{equation}\label{nonlin}\|v\|_{L^\infty((0,T);H^1_{-\delta})}\leq C\Big (1+\|v(0)\|_{H^1_{8\delta}}^{\frac{p+2}2}\Big).\end{equation}
Moreover, for every $q\in [2, \infty)$ and $\delta\in \Big(0, \frac 1{36(2q-1)}\Big)$, we have
\begin{equation}\label{soblweider}
\|v\|_{L^\infty((0,T);W^{1,q}_{\delta})}\leq C  \|v\|_{L^\infty((0,T);H^2_{-\delta})}^{1-\frac 1q} {\mathcal P}\Big( \|v(0)\|_{H^1_{4(2q-1)\delta}}\Big).
\end{equation}
\end{prop}



\begin{proof}
By using conservations \eqref{eq:MollifiedQuantitiesNAV} and \eqref{eq:MollifiedQuantitiesHAV} and recalling that $\lambda>0$, we get
\begin{align*}
\int_{\R^2} &\frac{1}{2}(|\D v(t)|^2+|v(t)|^2) e^{-2A} dx-\int_{\R^2} \frac{1}{2}|v(t)|^2  V e^{-2A}dx
\\\nonumber&\leq \int_{\R^2} \frac{1}{2}(|\D v(0)|^2+|v(0)|^2)e^{-2A}dx-\int_{\R^2} \frac{1}{2}|v(0)|^2  V e^{-2A}dx
\\\nonumber& \quad +\int_{\R^2} \frac{\lambda}{p+2} |v(0) |^{p+2} e^{-(p+2)A}dx
\end{align*}
and by \eqref{functionalDMAV} we get
\begin{align}
\label{h1bdd-1}
\begin{split}
\int_{\R^2} &\frac{1}{2}(|\D v(t)|^2+|v(t)|^2) e^{-2A} dx\\&
\leq C \|v (t)\|_{H^\frac 12_{\chi}}^2
+ \int_{\R^2} \frac{1}{2}(|\D v(0)|^2+|v(0)|^2)e^{-2A} dx
\\&
\quad +\int_{\R^2} \frac{\lambda}{p+2} |v(0) |^{p+2} e^{-(p+2)A} dx.
\end{split}
\end{align}
In turn by \eqref{exponentialAV} and the Sobolev embedding $H^1_{\delta}\subset L^{p+2}_\delta$, we get
\begin{align}
\label{h1bdd-2}
\|v(t)\|_{H^1_{-\delta}}^2\leq C\Big (1+\|v(0)\|_{H^1_{\delta}}^{p+2}\Big )+C \|v(t)\|_{H^1_{-\delta}}\|v(t)\|_{L^2_{2\delta}},
\end{align}
where we used a special case of \eqref{lem:BesovInterpolation} and we chose $\chi = \frac{\delta}{2}$.
Hence, by elementary considerations, we get
\begin{equation}\label{mattmaj}\|v(t)\|_{H^1_{-\delta}}^2\leq C\Big (1+\|v(0)\|_{H^1_{\delta}}^{p+2}\Big )+ C \|v(t)\|_{L^2_{2\delta}}^2.
\end{equation}
Next, by following the computation in \cite[Lemma 3.1]{DM} (see also \eqref{latersee}), we get
\begin{equation*}
\frac{\dd}{\dd t} \int_{\R^2}  |\w^{\delta} \,v(t)|^2 e^{-2A} dx
\leq C \,\int_{\R^2} \w^{2\delta-1} |\D v(t)| \,|v(t)| \,e^{-2A} dx,
\end{equation*}
which, by integration in time, \eqref{exponentialAV}, and the Cauchy-Schwartz inequality, implies that
\begin{equation*}
\|v(t) e^{-A}\|_{L^2_\delta}^2 \leq \|e^{-A}v(0)\|_{L^2_\delta}^2
+C \int_0^t \|\nabla v(\tau)\|_{L^2_{2\delta-1}} \|v(\tau)\|_{L^2_{\frac \delta 2}}d\tau,
\end{equation*}
so that by using again \eqref{exponentialAV},
\begin{equation*}
\|v(t)\|_{L^2_\frac \delta 2}^2 \leq  C \|v(0)\|_{L^2_{2\delta}}^2
+C \int_0^t \|\nabla v(\tau)\|_{L^2_{-\frac \delta 4}}^2 d\tau+ C \int_0^t \|v(\tau)\|_{L^2_{\frac \delta 2}}^2d\tau
\end{equation*}
where we assumed $\delta\in (0, \frac 19)$  in order to guarantee $\langle x\rangle^{2\delta- 1}\leq \langle x\rangle^{-\frac{\delta}4}$.
By combining this estimate with \eqref{mattmaj}, where we replace $\delta$ by $\frac \delta 4$, we get
\begin{equation*}
\|v(t)\|_{L^2_\frac \delta 2}^2 \leq  C \|v(0)\|_{L^2_{2\delta}}^2
+C t  \Big (1+\|v(0)\|_{H^1_{\frac \delta 4}}^{p+2}\Big ) + C \int_0^t \|v(\tau)\|_{L^2_{\frac \delta 2}}^2d\tau.
\end{equation*}
We deduce \eqref{nonloin} by the Gronwall inequality. The estimate \eqref{nonlin} follows by combining \eqref{nonloin} and \eqref{mattmaj}.
Concerning \eqref{soblweider} we can combine \eqref{*****} with \eqref{nonloin}.
\end{proof}

\begin{remark}
\label{rmk:foc}
We consider \eqref{NLSAV} with $\lambda < 0$. If $0 < p < 2$, \eqref{nonlin} and \eqref{nonloin} follow from Proposition 3.2 and Corollary 3.3 in \cite{DM}, respectively. If $p \geq 2$, we further assume that $\| v(0) \|_{H^1_{\delta_0}} \ll 1$ for some $\delta_0 > 0$. In this case, \eqref{h1bdd-1} needs to be replaced by
\begin{align*}
\int_{\R^2} &\frac{1}{2}(|\D v(t)|^2+|v(t)|^2) e^{-2A} dx\\\nonumber&
\leq C \|v (t)\|_{H^\frac 12_{\frac{\delta}{2}}}^2
+ \int_{\R^2} \frac{1}{2}(|\D v(0)|^2+|v(0)|^2)e^{-2A} dx
\\\nonumber&
\quad +\int_{\R^2} \frac{\lambda}{p+2} |v(0) |^{p+2} e^{-(p+2)A} dx - \int_{\R^2} \frac{\lambda}{p+2} |v(t) |^{p+2} e^{-(p+2)A} dx.
\end{align*}
Thus, instead of \eqref{h1bdd-2}, we obtain
\begin{align*}
\|v(t)\|_{H^1_{-\delta}}^2 &\leq C \|v(0)\|_{H^1_{\delta}}^{2} +C \|v(t)\|_{H^1_{-\delta}}\|v(t)\|_{L^2_{2p\delta}} + \| v(t) \|_{L^{p + 2}_{\frac{3p \delta}{p+2}}}^{p + 2} \\
&\leq C \|v(0)\|_{H^1_{\delta}}^{2} + \frac 12 \|v(t)\|_{H^1_{-\delta}}^2 + C \|v(t)\|_{L^2_{2p\delta}}^2 + \| v(t) \|_{H^{\frac{p}{p + 2}}_{\frac{3p \delta}{p+2}}}^{p + 2} \\
&\leq C \|v(0)\|_{H^1_{\delta}}^{2} + \frac 12 \|v(t)\|_{H^1_{-\delta}}^2 + C \|v(t)\|_{L^2_{2p\delta}}^2 + \| v(t) \|_{H^1_{-\delta}}^p \| v(t) \|_{L^2_{2p \delta}}^2
\end{align*}
where we take $\delta > 0$ to be sufficiently small and we used \eqref{exponentialAV}, the Sobolev embedding $H^{\frac{p}{p+2}}_{
\chi} \subset L^{p+2}_{\chi}$ for any $\chi \in \R$, and \eqref{lem:BesovInterpolation}. Assume that $\|v(t)\|_{H^1_{-\delta}} \leq 1$, so that we have the bound
\begin{align}
\label{h1bdd-3}
\|v(t)\|_{H^1_{-\delta}}^2 \leq C \|v(0)\|_{H^1_{\delta}}^{2} + C \|v(t)\|_{L^2_{2p\delta}}^2.
\end{align}
Then, by arguing as in the proof of Proposition \ref{prop:L2H1} and using \eqref{h1bdd-3} instead of \eqref{mattmaj}, we get
\begin{align*}
\| v(t) \|_{L^2_{2p \delta}}^2 \leq C(1 + t) \| v(0) \|_{H^1_{4p \delta}}^2 + C \int_0^t \| v(\tau) \|_{L^2_{2p \delta}}^2 d\tau.
\end{align*}
By applying the Gronwall inequality and taking $\| v(0) \|_{H^1_{4p \delta}}^2$ to be sufficiently small, we obtain 
$\| v(t) \|_{L^2_{2p \delta}}^2 \leq \frac{1}{100C}$, so that \eqref{h1bdd-3} gives $\| v(t) \|_{H^1_{-\delta}} \leq \frac 12$. Thus, by a standard continuity argument, we get the following two uniform bounds for $\delta > 0$ sufficiently small:
\begin{align*}
\| v \|_{L^\infty ((0, T); L^2_\delta)} &\leq 1, \\
\| v \|_{L^\infty ((0, T); H^1_{-\delta})} &\leq 1.
\end{align*}
\end{remark}

\medskip
The next proposition will also be useful.

\begin{prop} Assume \eqref{exponentialAV} and \eqref{functionalDMAV}. Let $0<\delta<\min \{\frac 19, \frac{\bar \eta}{9(6-4\bar \eta)} \}$ and $\bar \eta\in (0, 1)$.
For every $t>0$  and for every solution $v$ to \eqref{NLSAV}, we have the bound
\begin{equation}\label{fundam}\|v (t)\|_{H^{\frac 2{2-\bar \eta}}_\delta} \leq C 
{\mathcal P} \Big (\|v(0)\|_{H^1_{4\delta(\frac {6-4\bar \eta}{\bar \eta})}}\Big)\|v (t)\|_{H^{2}_{-\delta (\frac{1-2\bar \eta}{\bar \eta})}}^{\bar \eta}.
\end{equation}
In particular, for every $\eta_0\in (0, \frac 13), \delta_0>0$, there exists $\bar \delta>0$ such that
\begin{equation}\label{fundamnew}\|v (t)\|_{H^{\frac 2{2-\eta_0}}_\delta} \leq C 
{\mathcal P} \Big (\|v(0)\|_{H^1_{\delta_0}}\Big )\|v (t)\|_{H^{2}_{-\delta}}^{\eta_0}, \quad \forall \delta\in (0, \bar \delta).
\end{equation}
\end{prop}

\begin{proof}
By special case of \eqref{lem:BesovInterpolation}, we get
\begin{equation*}\|f\|_{H^{1+s}_\delta} \leq C \|f\|_{H^{1-s}_{2\delta}}^{\frac{1-s}{1+s}} \|f\|_{H^{2}_{-\delta(\frac{1-3s}{2s})}}^{\frac{2s}{1+s}},
\quad \delta>0, \quad  s\in (0,1)\end{equation*}
and also
\begin{equation*}\|f\|_{H^{1-s}_{2\delta}} \leq C \|f\|_{H^{1}_{-\delta}}^{1-s}  \|f\|_{L^2_{\delta(\frac{3-s}s)}}^{s},
\quad  \delta>0, \quad  s\in (0,1).\end{equation*}
Summarizing, we obtain
\begin{equation*}\|f\|_{H^{1+s}_\delta} \leq C\|f\|_{H^{1}_{-\delta}}^{\frac{(1-s)^2}{1+s}}  \|f\|_{L^2_{\delta(\frac{3-s}s)}}^{\frac{s(1-s)}{1+s}}   \|f\|_{H^{2}_{-\delta (\frac{1-3s}{2s})}}^{\frac{2s}{1+s}}, \quad \delta>0, \quad s\in (0,1)
\end{equation*}
Next, we select $s\in (0,1)$ such that $\bar \eta=\frac{2s}{s+1}$, namely $s=\frac{\bar \eta}{2-\bar \eta}$.
We recall that \eqref{nonloin} and \eqref{nonlin} are available for solutions $v$ to \eqref{NLSAV} and so we get
\begin{equation*}\|v (t)\|_{H^{\frac 2{2-\bar \eta}}_\delta} \leq C {\mathcal P} \Big(\|v(0)\|_{H^1_{8\delta}}\Big) {\mathcal P} \Big(\|v(0)\|_{H^1_{4\delta(\frac {6-4\bar \eta}{\bar \eta})}}\Big)\|v (t)\|_{H^{2}_{-\delta (\frac{1-2\bar \eta}{\bar \eta})}}^{\bar \eta}.
\end{equation*}
The conclusion follows since $8\delta\leq 4\delta(\frac {6-4\bar \eta}{\bar \eta})$ for every $\bar \eta\in (0,1)$.
The bound \eqref{fundamnew} is an easy consequence of \eqref{fundam}.
\end{proof}

\subsection{Nonlinear Strichartz Estimates}

Along this subsection, $v$ denotes any solution to \eqref{NLSAV} and $\delta_0>0$ is a fixed number
such that
$v(0)\in H^2_{\delta_0}$.
We aim at proving local in time Strichartz space-time bounds for the solution $v$.
We shall use the quantities introduced respectively in 
\eqref{AVdeltar} and \eqref{AVdeltakappa} .

\begin{prop} Assume \eqref{exponentialAV} and \eqref{functionalDMAV}. 
Let $T>0$ and $\bar \eta\in (0,1)$ be fixed. Then, for every $s>0$ and $0<\delta<\min\{\frac 1{18}, \frac{\bar \eta}{18(6-4\bar \eta)}\}$, there exist $s_1, \delta_1,\tilde \delta>0$ such that 
$\frac {\delta_1}{s_1}>1$ and for every $r\in [4, \infty)$, we have the following bound:
\begin{align}
\label{BF}
\begin{split}
\| &v\|_{L^4((0,T);W^{{\frac 34-s},4})} \\
& \leq {\mathcal P} \Big(|(A,V)|_{\delta,r}\Big) {\mathcal P} \Big(|(A,V)|_{\tilde \delta,r}\Big) {\mathcal P} \Big(|(A,V)|_{\frac{\delta^2}4,r}\Big)  \mathcal P \Big (\|(A,V)\|_{\delta_1, \frac 2{s_1}}^2 \Big) \\
&\quad \times \Big [\|v(0)\|_{H^1_{\delta}}
 +  {\mathcal P} \Big (\|v(0)\|_{H^1_{8\delta(\frac {6-4\bar \eta}{\bar \eta})}}\Big)\|v\|_{L^\infty((0,T);H^{2}_{-2\delta (\frac{1-2\bar \eta}{\bar \eta})})}^{(p+1)\bar \eta}\Big].
\end{split}
\end{align}
In particular, for every given $\eta_0\in (0,\frac 25)$ and $s>0$, there exists $\bar \delta>0$ such that for every $\delta\in (0, \bar \delta)$,
there are $s_1, \delta_1, \tilde \delta>0$ with $\frac{\delta_1}{s_1}>1$ and for every $r\in [4,\infty)$,
\begin{align}
\label{...}
\begin{split}
\|&v\|_{L^4((0,T);W^{{\frac 34-s},4})} \\
&\leq {\mathcal P} \Big(|(A,V)|_{\delta,r}\Big) {\mathcal P} \Big(|(A,V)|_{\tilde \delta,r}\Big) {\mathcal P} \Big(|(A,V)|_{\frac{\delta^2}4,r}\Big)  {\mathcal P}\Big(\|(A,V)\|_{\delta_1, \frac 2{s_1}}\Big) \\
&\quad \times {\mathcal P} \Big(\|v(0)\|_{H^1_{\delta_0}}\Big) \|v\|_{L^\infty((0,T);H^{2}_{-\delta})}^{(p+1)\eta_0}.
\end{split}
\end{align}

\end{prop}
\begin{proof}
We get by \eqref{noDuhamel} and \eqref{Duhamel} that
\begin{align*}
\|v\|_{L^4((0,T);W^{{\frac 34-s},4})}&\leq {\mathcal P} \Big(|(A,V)|_{\tilde \delta,r}\Big)
 {\mathcal P} \Big(|(A,V)|_{\frac{\delta^2}4,r}\Big) \mathcal P \Big (\|(A,V)\|_{\delta_1, \frac 2{s_1}}^2 \Big) \\\nonumber&
\quad \times \Big (\|v(0)\|_{H^1_{\delta}} 
+ |\lambda|\int_0^T  \|e^{-pA} v(\tau) |v(\tau)|^p\|_{H^1_{\delta}} d\tau\Big).
\end{align*}
On the other hand, we have 
\begin{align*}
\| &e^{-pA} v |v|^p\|_{H^1_{\delta}} \\
&\leq C \|e^{-pA} v |v|^p\|_{L^2_{\delta}}+ C \|\nabla A e^{-pA} v |v|^p\|_{L^2_{\delta}}+
C\|e^{-pA} \nabla v |v|^p\|_{L^2_{\delta}}\\\nonumber&
\leq C \|\w^{-\delta} e^{-pA}\|_{L^\infty} \|v\|_{H^1_{2\delta}}  \|v\|_{L^\infty}^p + C \|\w^{-\delta} \nabla A e^{-pA}\|_{L^r} \|v\|_{L^\frac{2r}{r-2}_{2\delta}}
\|v\|_{L^\infty}^p\\\nonumber&
\leq  C |(A,V)|_{\delta,r}  \|v\|_{H^{\frac 2{2-\bar \eta}}_{2\delta}}^{p+1},
\end{align*}
where $v=v(t)$ for some fixed $t\in [0, T]$ and we used the Sobolev embedding $H^{\frac 2{2-\bar \eta}}_{2\delta}\subset L^q_{\delta}$ for every $q\in [2, \infty]$.
By using \eqref{fundam}, we get
$$\|e^{-pA} v |v|^p\|_{L^\infty((0,T);H^1_{\delta})}\leq C{\mathcal P} \Big (\|v(0)\|_{H^1_{8\delta(\frac {6-4\bar \eta}{\bar \eta})}}\Big)
\|v\|_{L^\infty((0,T);H^{2}_{-2\delta (\frac{1-2\bar \eta}{\bar \eta})})}^{(p+1)\bar \eta}.$$
The bound \eqref{...} easily follows from \eqref{BF}.
\end{proof}


As a consequence, we can show the following estimates.

\begin{prop}\label{maiM} Assume \eqref{exponentialAV} and \eqref{functionalDMAV}. 
Let $T>0$, $\eta_0\in (0,\frac 25)$, and $s\in (0,\frac 18)$ be given. Then, there exists $\bar \delta>0$ such that for every $\delta\in (0, \bar \delta)$,
there exist $s_1, \delta_1, \tilde \delta>0$ with $\frac{\delta_1}{s_1}>1$ and for every $r\in [4,\infty)$,
\begin{align}
\label{strimon}
\begin{split}
\|&v\|_{L^2((0,T);W^{1,4})}\\
&\leq
{\mathcal P} \Big(|(A,V)|_{\delta,r}\Big) {\mathcal P} \Big(|(A,V)|_{\tilde \delta,r}\Big)
{\mathcal P} \Big(|(A,V)|_{\frac{\delta^2}4,r}\Big)  {\mathcal P}\Big(\|(A,V)\|_{\delta_1, \frac 2{s_1}}\Big)
\\& \quad \times{\mathcal P} \Big(\|v(0)\|_{H^1_{\delta_0}}\Big)\|v\|_{L^\infty((0,T);H^{2}_{-\delta})}^{\frac{2(1-2s)(p+1)\eta_0}{3}}
 \|v\|_{L^\infty((0,T);H^2_{-\frac{s\delta_0}{4(4-s)}})}^{\frac{(1+4s)(2-s)}6}.
\end{split}
\end{align}
In particular, by choosing $s=\frac 1{16}$ and
$0<\delta<\min\Big \{\bar \delta, \frac{\delta_0}{64(4-\frac 1{16})}\Big\}$,
we get
\begin{align}
\label{strigi}
\begin{split}
\|&v\|_{L^2((0,T);W^{1,4})}\\
&\leq 
{\mathcal P} \Big(|(A,V)|_{\delta,r}\Big) {\mathcal P} \Big(|(A,V)|_{\tilde \delta,r}\Big)
{\mathcal P} \Big(|(A,V)|_{\frac{\delta^2}4,r}\Big)  {\mathcal P}\Big(\|(A,V)\|_{\delta_1, \frac 2{s_1}}\Big)\\
&\quad \times {\mathcal P} \Big(\|v(0)\|_{H^1_{\delta_0}}\Big)\|v\|_{L^\infty((0,T);H^{2}_{-\delta})}^{\frac 7{12}(p+1)\eta_0+\frac{155}{384}}
\end{split}
\end{align}
\end{prop}
\begin{proof}
We have the following interpolation bound at time fixed:
$$\|v(t)\|_{W^{1,4}}\leq C \|v(t)\|_{W^{\frac 34-s, 4}}^{\frac{2(1-2s)}{3}} \|v(t)\|_{H^{2-s}}^{\frac{1+4s}3}, \quad s\in \Big(0,\frac 12\Big)$$
and hence by integration in time and using H\"older in time, we get 
\begin{align}
\label{Palom}
\begin{split}
\|&v(t)\|_{L^2((0,T);W^{1,4})}\\
&\leq C \|v(t)\|_{L^4((0,T);W^{\frac 34-s, 4})}^{\frac{2(1-2s)}{3}} \|v(t)\|_{L^\infty((0,T);H^{2-s})}^{\frac{1+4s}3}, \quad s\in \Big(0,\frac 18\Big).
\end{split}
\end{align}
Next, notice that we have
\begin{equation}\label{palom}\|v(t)\|_{H^{2-s}_\frac{s\delta_0}{4(4-s)}} \leq C \|v(t)\|_{L^2_{\frac{\delta_0}4}}^\frac s2 \|v(t)\|_{H^2_{-\frac{s\delta_0}{4(4-s)}}}^\frac{2-s}2.
\end{equation}
We conclude \eqref{strimon} by combining \eqref{nonloin} with \eqref{...}, \eqref{Palom}, and \eqref{palom}.
The estimate \eqref{strigi} follows by \eqref{strimon} once we notice that the condition $0<\delta<\frac{\delta_0}{64(4-\frac 1{16})}$
implies that we have the embedding $H^2_{-\delta}\subset H^2_{-\frac{s\delta_0}{4(4-s)}}$ for $s=\frac 1{16}$. Then one can abosorb the term 
$ \|v\|_{L^\infty((0,T);H^2_{-\frac{s\delta_0}{4(4-s)}})}$ in the factor $\|v\|_{L^\infty((0,T);H^2_{-\delta})}$ on the r.h.s. in \eqref{strimon}.
\end{proof}

\section{Modified energies}
\label{sec:energy}

Along this section we denote by $v$ a solution to the following equation with time-independent $A$ and $V$:
\begin{align}
\label{eq:IntroductionMollifiedEquationAV}
\ii \partial_t v=\Delta v- 2\nabla A\cdot \nabla v+V v-\lambda e^{-pA}v |v|^{p}, \quad \lambda\geq 0
\end{align}
and $\delta_0>0$ will denote a fixed given number such that 
$v(0)\in H^2_{\delta_0}$.
The following result has already been used in the linear case ($\lambda=0$).

\begin{prop}\label{modifen} 
Let $v$ be solution to \eqref{eq:IntroductionMollifiedEquationAV}, then we have the following identity:
\begin{equation}\label{MODIFIED}
\frac {\dd}{\dd t} {\mathcal E}_{A,V}(v(t))= -\lambda \mathcal H_{A,V}(v(t)),
\end{equation}
where
\begin{equation*}
{\mathcal E}_{A,V}(v(t))= \int_{\R^2} |\Delta v(t)|^2 e^{-2A} dx+{\mathcal F}_{A,V}(v(t))
- \lambda{\mathcal G}_{A,V}(v(t))
\end{equation*}
and the energies ${\mathcal F}_{A,V}, {\mathcal G}_{A,V}, \mathcal H_{A,V}$
are defined as follows:
\rm
\begin{align*}
 {\mathcal F}_{A,V}(v(t)) :\!&=
-4 \Re \int_{\R^2} \Delta v(t) \nabla A \cdot \nabla \bar v(t) e^{-2A} dx
\\\nonumber &
\quad - 4 \int_{\R^2} (\nabla A\cdot \nabla v)^2 e^{-2A}dx
+2 \Re \int_{\R^2} v(t) V  \nabla \bar v(t) \cdot \nabla (e^{-2A})dx
\\\nonumber&
\quad +2 \Re \int_{\R^2} \Delta v(t) \bar v(t) V e^{-2A}dx
+ \int_{\R^2}  |v(t)|^2 V^2 e^{-2A}dx, \\
{\mathcal G}_{A,V}(v(t)) :\!&= 
- \int_{\R^2} |\nabla v(t)|^2|v(t)|^p e^{-(p+2)A} dx \\
&\quad -2 \Re \int_{\R^2} v(t) \nabla (|v(t)|^p) \cdot \nabla \bar v(t)  e^{-(p+2)A}dx
\\\nonumber&
\quad +\frac p4 \int_{\R^2} |\nabla(|v(t)|^2)|^2|v(t)|^{p-2}  e^{-(p+2)A}dx \\
&\quad +\frac {2 }{p+2} \int_{\R^2} |v(t)|^{p+2} V e^{-(p+2)A}dx\\\nonumber&
\quad +2p\Re \int_{\R^2} v(t) |v(t)|^p \nabla A  \cdot \nabla \bar v(t) e^{-(p+2)A}dx, \\
\mathcal H_{A,V}(v(t)):\!&= 
-\int_{\R^2} |\nabla v(t)|^2 \partial_t (|v(t)|^p)  e^{-(p+2)A} dx\\\nonumber&
\quad -2 \Re \int_{\R^2} \partial_t v(t) \nabla (|v(t)|^p) \cdot \nabla \bar v(t)
 e^{-(p+2) A} dx\\\nonumber&
\quad -\frac p4 \int_{\R^2} |\nabla(|v(t)|^2)|^2\partial_t (|v(t)|^{p-2})
e^{-(p+2)A} dx \\
&\quad +2p \Re \int_{\R^2} \partial_t (v(t) |v(t)|^p) 
\nabla A \cdot \nabla \bar v(t) e^{-(p+2)A}dx. 
\end{align*}
\end{prop}

\begin{proof} In the sequel, we denote by $(\cdot,\cdot)$ the usual scalar product in $L^2$. We start with the following computation:
\begin{align}\label{co}
\begin{split}
\frac {\dd}{\dd t} (\Delta v, \Delta v e^{-2A}) &= 2\Re (\partial_t \Delta v,  \Delta v e^{-2A}) \\
&= 2\Re (\partial_t \Delta v, \ii\partial_t v e^{-2A} ) + 4\Re (\partial_t \Delta v, \nabla A \cdot \nabla ve^{-2A}) \\
&\quad -2\Re (\partial_t \Delta v, v Ve^{-2A})
+2\lambda \Re (\partial_t \Delta v,  v|v|^p e^{-(p+2)A})
\\&=I+II+III+IV.
\end{split}
\end{align}
Notice that
\begin{align}\label{I}
\begin{split}
I&=-2\Re (\partial_t \nabla v, \ii\partial_t \nabla v e^{-2A} )-2\Re (\partial_t \nabla v, \ii\partial_t v \nabla (e^{-2A}) )\\
&=-2 \Im  (\partial_t \nabla v, \partial_t v \nabla (e^{-2A}) ).
\end{split}
\end{align}
Moreover we have 
\begin{align*}
II &= 4\Re (\partial_t \Delta v, \nabla A \cdot \nabla ve^{-2A})\\
&=2\Re (\Delta v, \partial_t \nabla v \cdot \nabla (e^{-2A})) + 4 \frac \dd{\dd t}\Re (\Delta v,  \nabla A \cdot \nabla ve^{-2A})
\end{align*}
and using again the equation solved by $v$, we obtain
\begin{align*}
II&=4 \frac \dd{\dd t}\Re (\Delta v, \nabla A \cdot \nabla ve^{-2A})-2\Im(\partial_t v, \partial_t \nabla v \cdot \nabla (e^{-2A})) 
\\\nonumber&
\quad +4\Re (\nabla A \cdot \nabla v, \partial_t \nabla v \cdot \nabla (e^{-2A})) -
2\Re (V v , \partial_t \nabla v \cdot \nabla (e^{-2A})) \\
&\quad + 2\lambda \Re (e^{-pA}v|v|^p , \partial_t \nabla v \cdot \nabla (e^{-2A}))
\\
&=4\frac \dd{\dd t}\Re (\Delta v, \nabla A \cdot \nabla ve^{-2A})+2 \Im (\partial_t \nabla v, \partial_t v \nabla (e^{-2A})) \\\nonumber&
\quad +4\Re (\nabla A \cdot \nabla v, \partial_t \nabla v \cdot \nabla (e^{-2A}))-
2\Re (V v , \partial_t \nabla v \cdot \nabla (e^{-2A})) \\\nonumber&
\quad + 2\lambda \Re (e^{-pA}  v|v|^p , \partial_t \nabla v \cdot \nabla (e^{-2A})).
\end{align*}
Hence by \eqref{I} we get
\begin{align*}
II  &=  4 \frac \dd{\dd t}\Re (\Delta v, \nabla A \cdot \nabla ve^{-2A})
 -I 
+4\Re (\nabla A \cdot \nabla v, \partial_t \nabla v \cdot \nabla (e^{-2A}))
\\\nonumber&
\quad -
2\Re (V v, \partial_t \nabla v \cdot \nabla (e^{-2A})) 
+2\lambda \Re (e^{-pA} v|v|^p , \partial_t \nabla v \cdot \nabla (e^{-2A}))
\\\nonumber&
= 4\frac \dd{\dd t}\Re (\Delta v, \nabla A \cdot \nabla ve^{-2A}) 
-I 
+4\Re (\nabla A \cdot \nabla v, \partial_t \nabla v \cdot \nabla (e^{-2A}))
\\\nonumber&
\quad -2 \frac \dd{\dd t} \Re (Vv  , \nabla v \cdot \nabla (e^{-2A})) 
+2\Re (V \partial_tv, \nabla v \cdot \nabla (e^{-2A})) 
\\\nonumber &
\quad +2\lambda \Re (e^{-pA}v|v|^p , \partial_t \nabla v \cdot \nabla (e^{-2A})).
\end{align*}
Namely, we have
\begin{align*}
I+II &= 2\Re (V \partial_tv, \nabla v \cdot \nabla (e^{-2A})) + 4 \frac \dd{\dd t}\Re (\Delta v, \nabla A \cdot \nabla ve^{-2A}) \\
&\quad +4\Re (\nabla A \cdot \nabla v, \partial_t \nabla v \cdot \nabla (e^{-2A})) -
2\frac \dd{\dd t} \Re (V v , \nabla v \cdot \nabla (e^{-2A})) \\
&\quad +2\lambda \Re (e^{-pA} v|v|^p , \partial_t \nabla v \cdot \nabla (e^{-2A})).
\end{align*}
On the other hand, we can compute the last term in the above identity
\begin{align*}
 &2 \lambda \Re (e^{-pA} v|v|^p , \partial_t \nabla v \cdot \nabla (e^{-2A}))\\\nonumber&
 = - 2\lambda \Re (\nabla ( e^{-pA} v|v|^p) , \partial_t \nabla v e^{-2A})
 -2\lambda \Re 
(e^{-pA} v|v|^p , \partial_t \Delta v e^{-2A})
\\\nonumber&
= -2\lambda \Re (\nabla ( e^{-pA}v|v|^p) , \partial_t \nabla v e^{-2A})
-IV
\\\nonumber&
= -2\lambda \Re (e^{-pA} \nabla v|v|^p, \partial_t \nabla v e^{-2A})
-2\lambda \Re (e^{-pA}v\nabla (|v|^p), \partial_t \nabla v e^{-2A})\\\nonumber&
\quad -2\lambda \Re (\nabla(e^{-pA})
v |v|^p, \partial_t \nabla v e^{-2A})
-IV\\\nonumber&
=-\lambda (\partial_t (|\nabla v|^2)|v|^p,  e^{-(p+2)A})
-2\lambda \frac \dd{\dd t} \Re (v\nabla (|v|^p), \nabla v  e^{-(p+2)A}) \\
&\quad +2\lambda \Re (\partial_t v\nabla (|v|^p), \nabla v  e^{-(p+2)A}) + 2\lambda \Re (v\nabla \partial_t (|v|^p), \nabla v  e^{-(p+2)A}) \\
&\quad -2\lambda \Re (\nabla(e^{-pA})v |v|^p, \partial_t \nabla v e^{-2A})
-IV\\\nonumber&
=-\lambda \frac \dd{\dd t}  (|\nabla v|^2|v|^p,  e^{-(p+2)A})
+\lambda (|\nabla v|^2 \partial_t (|v|^p),  e^{-(p+2)A}) \\
&\quad -2\lambda \frac \dd{\dd t} \Re (v\nabla (|v|^p), \nabla v  e^{-(p+2)A})
+2\lambda \Re (\partial_t v\nabla (|v|^p), \nabla v  e^{-(p+2)A}) \\
&\quad +\frac{\lambda p}{2} (\partial_t (\nabla(|v|^2)|v|^{p-2}), \nabla (|v|^2)  e^{-(p+2)A}) \\
&\quad -2\lambda \Re (\nabla(e^{-pA})v |v|^p, \partial_t \nabla v e^{-2A})
-IV
\\\nonumber&
=-\lambda \frac \dd{\dd t}  (|\nabla v|^2|v|^p,  e^{-(p+2)A})
+\lambda (|\nabla v|^2 \partial_t (|v|^p),  e^{-(p+2)A}) \\
&\quad -2\lambda \frac \dd{\dd t} \Re (v\nabla (|v|^p), \nabla v  e^{-(p+2)A})
+2\lambda \Re (\partial_t v\nabla (|v|^p), \nabla v  e^{-(p+2)A}) \\
&\quad +\frac {\lambda p}4 (\partial_t (|\nabla(|v|^2)|^2)|v|^{p-2},  e^{-(p+2)A})
+ \frac{\lambda p}2 (\nabla(|v|^2) \partial_t (|v|^{p-2}), \nabla (|v|^2)  e^{-(p+2)A}) \\
&\quad -2\lambda \Re (\nabla(e^{-pA})v |v|^p, \partial_t \nabla v e^{-2A})
-IV
\\\nonumber&
=-\lambda \frac \dd{\dd t}  (|\nabla v|^2|v|^p,  e^{-(p+2)A})
+\lambda (|\nabla v|^2 \partial_t (|v|^p), e^{-(p+2)A}) \\
&\quad -2\lambda \frac \dd{\dd t} \Re (v\nabla (|v|^p), \nabla v  e^{-(p+2)A})
+2\lambda \Re (\partial_t v\nabla (|v|^p), \nabla v  e^{-(p+2)A}) \\
&\quad +\frac {\lambda p} 4 \frac \dd{\dd t}(|\nabla(|v|^2)|^2|v|^{p-2},  e^{-(p+2)A})
+\frac {\lambda p}4 (|\nabla(|v|^2)|^2\partial_t (|v|^{p-2}),  e^{-(p+2)A}) \\
&\quad -2\lambda \Re (\nabla(e^{-pA})v |v|^p, \partial_t \nabla v e^{-2A})-IV\,. 
\end{align*}
By combining the identities above we get
\begin{align*}I&+II+IV \\
&=2\Re (\partial_tvV, \nabla v \cdot \nabla (e^{-2A}))
+ 2 \frac \dd{\dd t}\Re (\Delta v, 2 \nabla A \cdot \nabla ve^{-2A}) \\
&\quad +4\Re (\nabla A \cdot \nabla v, \partial_t \nabla v \cdot \nabla (e^{-2A}))
-2\frac \dd{\dd t} \Re (Vv, \nabla v \cdot \nabla (e^{-2A}))\\\nonumber&
\quad -\lambda \frac \dd{\dd t}  \Re (|\nabla v|^2|v|^p, e^{-(p+2)A})
+\lambda \Re (|\nabla v|^2 \partial_t (|v|^p), e^{-(p+2)A}) \\
&\quad -2\lambda \frac \dd{\dd t} \Re (v\nabla (|v|^p), \nabla v  e^{-(p+2)A})
+2\lambda \Re (\partial_t v\nabla (|v|^p), \nabla v  e^{-(p+2)A}) \\
&\quad +\frac {\lambda p} 4 \frac \dd{\dd t}(|\nabla(|v|^2)|^2|v|^{p-2},  e^{-(p+2)A}) +\frac {\lambda p}4 (|\nabla(|v|^2)|^2\partial_t (|v|^{p-2}),  e^{-(p+2)A}) \\
&\quad -2\lambda \Re (\nabla(e^{-pA})v |v|^p, \partial_t \nabla v e^{-2A}).
\end{align*}
Since
$III= -2\frac \dd {\dd t} \Re (\Delta v, V v e^{-2A})+ 
2 \Re (\Delta v, \partial_t v  V e^{-2A})
$,
we obtain the following identity:
\begin{align}\label{plual}
\begin{split}
I&+II+III+IV \\
&=2\Re (V \partial_tv, \nabla v \cdot \nabla (e^{-2A})) + 4\frac \dd{\dd t}\Re (\Delta v, \nabla A \cdot \nabla ve^{-2A}) \\
&\quad +
4\Re (\nabla A \cdot \nabla v, \partial_t \nabla v \cdot \nabla (e^{-2A})) - 2\frac \dd{\dd t} \Re (V v , \nabla v \cdot \nabla (e^{-2A})) \\
&\quad -\lambda \frac \dd{\dd t}  \Re (|\nabla v|^2|v|^p,  e^{-(p+2)A}) +\lambda \Re (|\nabla v|^2 \partial_t (|v|^p),  e^{-(p+2)A}) \\
&\quad -2\lambda \frac \dd{\dd t} \Re (v\nabla (|v|^p), \nabla v  e^{-(p+2)A}) +2\lambda \Re (\partial_t v\nabla (|v|^p), \nabla v  e^{-(p+2)A}) \\
&\quad +\frac {\lambda p} 4 \frac \dd{\dd t}(|\nabla(|v|^2)|^2|v|^{p-2},  e^{-(p+2)A}) +\frac {\lambda p}4 (|\nabla(|v|^2)|^2\partial_t (|v|^{p-2}),  e^{-(p+2)A}) 
\\
&
\quad -2\lambda \frac \dd{\dd t} \Re (\nabla(e^{-pA})v |v|^p, \nabla v e^{-2A})
+2\lambda \Re (\nabla(e^{-pA})\partial_t (v |v|^p), \nabla v e^{-2A})\\
&
\quad -2\frac \dd{\dd t}\Re (\Delta v, v V e^{-2A})+ 
2 \Re (\Delta v, \partial_t v V e^{-2A}).
\end{split}
\end{align}
Next, by using the equation solved by $v$, we compute the first, the third, and the last term on the r.h.s. above as
\begin{align*}
2&\Re (\partial_tv V, \nabla v \cdot \nabla (e^{-2A}))
+ 4\Re (\nabla A \cdot \nabla v, \partial_t \nabla v \cdot \nabla (e^{-2A})) \\
&\quad +2 \Re (\Delta v, \partial_t v V e^{-2A})
\\\nonumber&= 2\Re (\partial_tv V, \nabla v \cdot \nabla (e^{-2A}))-8\Re (\nabla A \cdot \nabla v, \partial_t \nabla v \cdot \nabla A e^{-2A}) \\
&\quad +4 \Re (\nabla v \cdot \nabla A, 
\partial_t v Ve^{-2A}) - 2 \Re (V v, \partial_t v V e^{-2A}) \\
&\quad+ 2 \lambda \Re ( e^{-pA}v|v|^p, \partial_t v V  e^{-2A})
 \\\nonumber
&= -4\Re (\partial_t v V, \nabla v \cdot \nabla A e^{-2A}) - 4 \frac \dd{\dd t} ((\nabla A\cdot \nabla v)^2, e^{-2A}) \\
&\quad + 4 \Re (\nabla v\cdot  \nabla A, 
\partial_t v V e^{-2A}) - 2 \Re (Vv, \partial_t v V e^{-2A}) \\
&\quad + 2 \lambda \Re (v|v|^p, \partial_t v V   e^{-(p+2)A})
\\\nonumber&=- 4 \frac \dd{\dd t} ((\nabla A\cdot \nabla v)^2, e^{-2A})-\frac \dd{\dd t} (V |v|^2, V e^{-2A}) \\
&\quad + \frac {2 \lambda}{p+2} \frac \dd{\dd t}\Re (|v|^{p+2}, V e^{-(p+2)A}).
\end{align*}
The proof of Proposition~\ref{modifen} easily follows by combining this identity with \eqref{co} and 
\eqref{plual}.
\end{proof}

Next, we provide some useful estimates on the energies
${\mathcal F}_{A,V}, {\mathcal G}_{A,V}, {\mathcal H}_{A,V}$.
We recall that the quantity $|(A,V)|_{\delta,r}$ has been introduced in \eqref{AVdeltar}.
\begin{prop}\label{FAV}
For every $\delta>0$ and $r\in (2, \infty)$, we have 
\begin{equation}\label{intermF}|{\mathcal F}_{A,V} (w)|\leq \mathcal P\Big(|(A,V)|_{\delta,r}\Big) \Big (\|e^{-A} \Delta w\|_{L^2} \|w\|_{W_\delta^{1,\frac{2r}{r-2}}}
+\|w\|_{W_\delta^{1,\frac{2r}{r-2}}}^2\Big)\end{equation}
for any generic time-independent function $w$.
Moreover, there exists $\bar \delta>0$ such that for any given $T>0$ and for every $\delta\in (0, \bar \delta)$ and
$r\in [4, \infty)$, we have
\begin{align}\label{FAV2}
\begin{split}
\sup_{t\in (0,T)} &|{\mathcal F}_{A,V} (v(t))|\leq \mathcal P\Big(|(A,V)|_{\delta,r}\Big) {\mathcal P}\Big( \|v(0)\|_{H^1_{\delta_0}}\Big)\\
&\times \Big (\|e^{-A} \Delta v\|_{L^\infty((0,T);L^2)}\|v\|_{L^\infty((0,T);H^2_{-\delta})}^{\frac{r+2}{2r}}
+ \|v\|_{L^\infty((0,T);H^2_{-\delta})}^{\frac{r+2}{r}}\Big).
\end{split}
\end{align}
\end{prop}

\begin{proof}
By the H\"older inequality we get:
\begin{align*}
|{\mathcal F}_{A,V} (w)| &\leq C \|e^{-A} \Delta w\|_{L^2} \| \w^\delta  \nabla w(t)\|_{L^{\frac{2r}{r-2}}}
\|\w^{-\delta} \nabla A e^{-A}\|_{L^r}
\\\nonumber&\quad + C  \|\w^{-\delta} e^{-A} \nabla A\|_{L^{r}}^2 \|\w^{\delta} \nabla w\|_{L^{\frac{2r}{r-2}}}^2
\\\nonumber&\quad + C   \|\w^{-2\delta } \nabla A e^{-2A} V\|_{L^\frac r2} \|\w^\delta \nabla w\|_{L^\frac{2r}{r-2}} \|\w^\delta w\|_{L^\frac{2r}{r-2}}
\\\nonumber&\quad +C \|e^{-A}\Delta w\|_{L^2} \|\w^\delta w\|_{L^\frac{2r}{r-2}}\|\w^{-\delta}V e^{-A}\|_{L^r}
\\\nonumber&\quad +C \|\w^{-\delta}V e^{-A}\|_{L^r}^2  \|\w^{\delta} w\|_{L^\frac{2r}{r-2}}^2,
\end{align*}
which in turn implies \eqref{intermF}.
On the other hand, by \eqref{soblweider} for
$\delta\in (0, \frac {r-2}{36(3r+2)})$, we have
\begin{align}\label{usefullateralso}
\begin{split}
\|v\|_{L^\infty((0,T);W^{1,\frac{2r}{r-2}}_{\delta})} &\leq C  \|v\|_{L^\infty((0,T);H^2_{-\delta})}^{\frac{r+2}{2r}} {\mathcal P}\Big ( \|v(0)\|_{H^1_{4(\frac{3r+2}{r-2})\delta}}\Big )
\\&\leq C  \|v\|_{L^\infty((0,T);H^2_{-\delta})}^{\frac{r+2}{2r}} {\mathcal P}\Big( \|v(0)\|_{H^1_{\delta_0}}\Big),
\end{split}
\end{align}
where at the last step we used that
for $r>4$ we have $4(\frac{3r+2}{r-2})\delta\leq 28 \delta$ and hence
$\|v(0)\|_{H^1_{4(\frac{3r+2}{r-2})\delta}} \leq \|v(0)\|_{H^1_{\delta_0}}$
provided that $\delta>0$ is small enough. The proof is complete.
\end{proof}

\begin{prop}\label{GAV}
For any given $T>0$ and $\eta_0\in (0, \frac 13)$, there exists $\bar \delta>0$ such that
for every $r\in [4,\infty)$ and $\delta\in (0, \bar \delta)$, we have the bound
\begin{equation}\label{GAV2}
\sup_{t\in (0,T)}|{\mathcal G}_{A,V} (v(t))|\leq |(A,V)|_{\delta,r} {\mathcal P} \Big(\|v(0)\|_{H^1_{\delta_0}}\Big)
\Big (1+\|v\|_{L^\infty ((0,T);H^2_{-\delta})}^{\frac{1}{r}+\frac 12+\eta_0 (p+2)}
\Big).
\end{equation} 

\end{prop}

\begin{proof}
We have by the H\"older inequality
\begin{align}\label{hold1}
\begin{split}
|{\mathcal G}_{A,V} (v(t))| &\leq C\|\nabla v(t)\|_{L^{\frac{2r}{r-2}}}^2 \|\w^\delta |v(t)|^p\|_{L^\frac r2} \|\w^{-\delta}e^{-(p+2)A}\|_{L^\infty}
\\& \quad +C\|\w^\delta |v(t)|^{p+2}\|_{L^\frac r{r-1}} \|\w^{-\delta} V e^{-(p+2)A}\|_{L^r}\\
&\quad +C\|\nabla v(t)\|_{L^\frac{2r}{r-2}} \|\w^\delta |v(t)|^{p+1}\|_{L^2} \|\w^{-\delta}\nabla A e^{-(p+2)A}\|_{L^r}.
\end{split}
\end{align}
Next, notice that by combining the Sobolev embedding
$L^q_\rho \subset H^\frac 2{2-\eta_0}_\rho$ for $ q\in [2, \infty]$ with \eqref{fundamnew}, we get 
\begin{align*}&\|\w^\delta |v(t)|^p\|_{L^\frac r2} = \|v(t)\|^p_{L^\frac{rp}2_\frac \delta p}\leq 
{\mathcal P} \Big(\|v(0)\|_{H^1_{\delta_0}}\Big)\|v(t)\|_{H^{2}_{-\delta}}^{\eta_0 p}
\\\nonumber& \|\w^\delta |v(t)|^{p+2}\|_{L^\frac r{r-1}} = \|v(t)\|_{L^\frac{r(p+2)}{r-1}_{\frac \delta{p+2}}}^{p+2}\leq 
{\mathcal P} \Big(\|v(0)\|_{H^1_{\delta_0}}\Big)\|v(t)\|_{H^{2}_{-\delta}}^{\eta_0 (p+2)}
\\\nonumber& \|\w^\delta |v(t)|^{p+1}\|_{L^2}= \|v(t)\|^{p+1}_{L^{2(p+1)}_{\frac \delta{p+1}}} \leq 
{\mathcal P} \Big (\|v(0)\|_{H^1_{\delta_0}}\Big)\|v(t)\|_{H^{2}_{-\delta}}^{\eta_0 (p+1)}.
\end{align*}
By combining the estimates above with \eqref{hold1} and \eqref{soblweider},
where we choose $q=\frac{2r}{r-2}$, we conclude the proof.
\end{proof}

\begin{prop}\label{HAV}
For any given $T>0$ and $\eta_0\in (0, \frac 13)$,  there exists $\bar \delta>0$ such that for every $r\in [4,\infty)$ and $\delta\in (0, \bar \delta)$, we have the bound
\begin{align}\label{HAV2}
\begin{split}
\int_0^T |{\mathcal H}_{A,V} (v(\tau))|d\tau &\leq |(A,V)|_{r,\delta} {\mathcal P} \Big(\|v(0)\|_{H^1_{\delta_0}}\Big) \|\partial_t v e^{-A}\|_{L^\infty ((0,T);L^2)} \\
&\quad \times \Big (1+ \|\nabla v\|_{L^2((0,T);L^4)}^2\Big) \Big
(1+ \|v\|_{L^\infty((0,T);H^{2}_{-\delta})}\Big)^{\eta_0 p}.
\end{split}
\end{align}

\end{prop}

\begin{proof}
We only focus on the case when $p > 1$. The case $p = 1$ will follow in a similar (and easier) manner. Notice that by the H\"older inequality, one can estimate 
\begin{align*}
|{\mathcal H}_{A,V} &(v(t))|
\leq C \|\nabla v(t)\|_{L^4}^2 \|\partial_t v e^{-A}\|_{L^2} \|\w^\delta |v(t)|^{p-1}\|_{L^\infty} \|\w^{-\delta} e^{-(p+1)A}\|_{L^\infty}
\\\nonumber&\quad +C \|\nabla v(t)\|_{L^4} \|\partial_t v e^{-A}\|_{L^2} \|\w^\delta |v(t)|^{p}\|_{L^\frac{4r}{r-4}} \|\w^{-\delta} |\nabla A|e^{-(p+1)A}\|_{L^r}.
\end{align*}
By combining \eqref{fundamnew} with the
Sobolev embedding
$H^\frac 2{2-\eta_0}_\rho\subset L^q_\rho$ for $ q\in [2, \infty]$, we get
$$\|\w^\delta |v(t)|^{p}\|_{L^\frac{4r}{r-4}}= \|v(t)\|_{L^\frac{4rp}{r-4}_\frac \delta p}^p\leq 
{\mathcal P} \Big(\|v(0)\|_{H^1_{\delta_0}}\Big)\|v(t)\|_{H^{2}_{-\delta}}^{\eta_0 p}$$
and
$$\|\w^\delta |v(t)|^{p-1}\|_{L^\infty}= \|v(t)\|_{L^\infty_{\frac \delta{p-1}}}^{p-1}\leq {\mathcal P} \Big(\|v(0)\|_{H^1_{\delta_0}}\Big)\|v(t)\|_{H^{2}_{-\delta}}^{\eta_0 (p-1)}.$$
After integration in time and the H\"older inequality w.r.t. time variable, we obtain \eqref{HAV2}.
\end{proof}

\section{$H^2$ a-priori bound}
\label{sec:H2}
We introduce the following family of regularized and localized potentials, for $\varepsilon>0$ and $n\in \N$:
$$A_{\varepsilon, n}=\theta_n Y_\varepsilon, \quad V_{\varepsilon, n}= \theta_n \tDYe$$
where $\theta_n (x) = \theta (\frac{x}{n})$, $\theta\in C^\infty_0(\R^2)$, $\theta\geq 0$, and $\theta(0)=1$.
We first notice that due to \eqref{exponential}, if we choose $A=A_{\varepsilon, n}$ and 
$V=V_{\varepsilon, n}$, then the condition \eqref{exponentialAV} holds uniformly w.r.t. $n$ and $\varepsilon$ with the constant $C$ replaced by a random 
constant $C(\omega)$. The same for \eqref{functionalDMAV} which is satisfied uniformly w.r.t. $n, \varepsilon$ 
with a random constant $C(\omega)$. In order to prove this fact combine \eqref{lem:ConvergenceYnew},
\eqref{lem:BoundNoise}, \eqref{convexpY}, \eqref{eq:BoundNoiseP}, and \cite{DM} (more precisely, see at page 1160 three lines below (33)).

We introduce, following the notation \eqref{eq:HAV}, the operators
$$H_{\varepsilon, n}= H_{A_{\varepsilon, n},V_{\varepsilon, n}}$$
as well as the associated nonlinear Cauchy problem
\begin{align}
\label{eq:IntroductionMollifiedEquationAVepsilonn}
\ii \partial_t v_{\varepsilon, n}=H_{\varepsilon, n} v_{\varepsilon, n}+\lambda e^{-pA_{\varepsilon, n}}v_{\varepsilon,n} |v_{\varepsilon,n}|^{p}, \quad v_{\varepsilon, n}(0)=v_0, \quad \lambda\geq 0
\end{align}
where $v_0=u_0e^{Y}\in H^2_{\delta_0}$ for some fixed $\delta_0 > 0$.
The main point in this section is to establish the following a-propri bound for any sufficiently small $\delta > 0$:
\begin{equation}\label{fundBound}\sup_{n\in \N} \|v_{\varepsilon, n}\|_{L^\infty((0,T);H^2_{-\delta})}\leq C(\omega) |\ln \varepsilon|^{C} \mathcal{P} \Big( \| v_0 \|_{H_{\delta_0}^2} \Big), \quad
\hbox{ a.s. } \omega\end{equation}
for any given $T>0$.

We now focus on proving the bound \eqref{fundBound}. Recall the quantity $|(A, V)|_{\delta, r}$ defined in \eqref{AVdeltar}. Assume that $\delta>0$ has been fixed in such a way that  \eqref{FAV2}, \eqref{GAV2} and \eqref{HAV2} are satisfied with the choice $A =A_{\varepsilon, n}, V=V_{\varepsilon, n}$.
On the other hand, one can show that once $\delta$ is fixed, one can choose 
$r \geq 4$ large enough in such a way 
that 
\begin{equation}\label{AVepsilonn}
\sup_{n\in \N} |(A_{\varepsilon, n},V_{\varepsilon, n})|_{\delta,r}\leq C(\omega) O(|\ln \varepsilon|^{C}), \quad \hbox{ a.s. } \omega.
\end{equation}
This bound follows from \eqref{AVdeltar}, \eqref{firstnoiseLp}, and \eqref{exponential} (the cut-off $\theta\big( \frac xn \big) $ can be easily handled by an elementary argument and the bounds are uniform in $n \in \N$).
Moreover, by \eqref{MODIFIED} we get
\begin{align}\label{naifra}
\begin{split}
\int_{\R^2} &|\Delta v_{\varepsilon, n}(t)|^2 e^{-2A_{\varepsilon, n}} dx \\
&=\int_{\R^2} |\Delta v_{\varepsilon, n}(0)|^2 e^{-2A_{\varepsilon, n}}dx
- {\mathcal F}_{_{\varepsilon, n}}(v_{\varepsilon, n}(t)) + {\mathcal F}_{_{\varepsilon, n}}(v_{\varepsilon, n}(0)) \\
&\quad - \lambda {\mathcal G}_{\eps,n}(v_{\varepsilon, n}(t))+ \lambda {\mathcal G}_{\eps,n}(v_{\varepsilon, n}(0))+\int_0^t {\mathcal H}_{\eps,n}(v_{\varepsilon, n}(\tau))d\tau
\end{split}
\end{align}
where $${\mathcal F}_{\varepsilon,n}={\mathcal F}_{A_{\varepsilon, n},V_{\varepsilon, n}}, \quad {\mathcal G}_{\varepsilon,n}={\mathcal G}_{A_{\varepsilon, n},V_{\varepsilon, n}}, 
\quad {\mathcal H}_{\varepsilon,n}={\mathcal H}_{A_{\varepsilon, n},V_{\varepsilon, n}}$$
are the energies defined along Proposition \ref{modifen} with $A =A_{\varepsilon, n}$ and $V=V_{\varepsilon, n}$.
Next, we apply \eqref{FAV2}, \eqref{GAV2} and \eqref{HAV2} 
with $A =A_{\varepsilon, n}$ and $V=V_{\varepsilon, n}$.
By using \eqref{FAV2} in conjunction with \eqref{AVepsilonn} and by choosing $r$ large enough, we deduce
\begin{align}\label{FAV1}
\begin{split}
&\sup_{t\in [0,T]} |{\mathcal F}_{\varepsilon, n}(v_{\varepsilon, n}(t))| \leq 
C(\omega)|\ln \varepsilon|^C  {\mathcal P}\Big( \|v_{\varepsilon, n}(0)\|_{H^1_{\delta_0}}\Big) \\ 
&\times \Big (\|e^{-A_{\varepsilon, n}} \Delta v_{\varepsilon, n}\|_{L^\infty((0,T);L^2)}\|v_{\varepsilon, n}\|_{L^\infty((0,T);H^2_{-\delta})}^{1^-}
+ \|v_{\varepsilon, n}\|_{L^\infty((0,T);H^2_{-\delta})}^{2^-}\Big)
\end{split}
\end{align}
for suitable $1^-\in (0,1)$ and $2^-\in (1,2)$.
Moreover, we can choose $\eta_0$ in such a way that
\begin{equation}\label{eta02}\frac{1}{r}+\frac 12+\eta_0 (p+2)\in (0,1);\end{equation}
\begin{equation}\label{eat01}\frac 7{6}(p+1)\eta_0+\frac{155}{192}+\eta_0 p\in (0,1);
\end{equation}
\begin{equation}\label{paolt}
\eta_0(p+1)\in (0,1).
\end{equation}
Hence by \eqref{GAV2}, \eqref{AVepsilonn}, and \eqref{eta02}, we obtain
\begin{align}\label{GAV1}
\begin{split}
\sup_{t\in [0,T]} &|{\mathcal G}_{\varepsilon, n}(v_{\varepsilon, n}(t))|\\
& \leq C(\omega) |\ln \varepsilon|^C
{\mathcal P} \Big(\|v_{\varepsilon, n}(0)\|_{H^1_{\delta_0}}\Big)
\Big ( 1+\|v_{\varepsilon, n}(t)\|_{L^\infty((0,T);H^{2}_{-\delta})}^{1^-}
\Big),
\end{split}
\end{align}
where $1^-\in (0,1)$.
Notice also that we have (by choosing $r$ even larger than above)
$$\sup_n |(A_{\varepsilon, n},V_{\varepsilon, n})|_{\frac{\delta^2}4,r}\leq C(\omega)|\ln \varepsilon|^C, \quad 
\sup_n |(A_{\varepsilon, n},V_{\varepsilon, n})|_{\tilde \delta,r}\leq C(\omega)|\ln \varepsilon|^C \hbox{ a.s. } \omega,$$
where $\tilde \delta>0$ is the one that appears in \eqref{strigi}.
Moreover, by \eqref{firstnoiseLp} in conjunction with the fact that in \eqref{strigi} we can assume $\frac{\delta_1}{s_1}>1$, we also have (see \eqref{AVdeltakappa})
\begin{equation}\label{nonsig}\sup_n \|(A_{\varepsilon, n},V_{\varepsilon, n})\|_{\delta_1, \frac 2{s_1}}\leq C(\omega)O(|\ln \varepsilon|^C) \hbox{ a.s. } \omega
\end{equation}
and hence by combining \eqref{HAV2} with \eqref{strigi} and \eqref{eat01}, we get
\begin{align}\label{HAVpao}
\begin{split}
\int_0^T |{\mathcal H}_{\varepsilon, n} &(v_{\varepsilon, n}(\tau))|d\tau \leq C(\omega)|\ln \varepsilon|^C {\mathcal P}\Big( \|v_{\varepsilon, n}(0)\|_{H^1_{\delta_0}}\Big)  \\
& \times \|\partial_t v_{\varepsilon, n} e^{-A_{\varepsilon, n}}\|_{L^\infty((0,T);L^2)}
\Big (1+ \|v_{\varepsilon, n}\|_{L^\infty((0,T);H^{2}_{-\delta})}\Big)^{1^-},
\end{split}
\end{align}
where $1^-\in (0,1)$.
Next we notice that by using the equation solved by $v_{\varepsilon, n}$, we get by the H\"older inequality and the Sobolev embedding 
$H^\frac 2{2-\eta_0}_\delta\subset L^q_\delta$, $q\in [2, \infty]$, in conjunction with \eqref{soblweider}, 
\eqref{fundamnew}, \eqref{nonsig}, and \eqref{paolt},
\begin{align}\label{HAVeq}
\begin{split}
\| &\partial_t v_{\varepsilon, n}(t) e^{-A_{\varepsilon, n}}\|_{L^2}\\
&\leq \|\Delta v_{\varepsilon, n}(t) e^{-A_{\varepsilon, n}}\|_{L^2} + C(\omega)|\ln \varepsilon|^C {\mathcal P} \Big (\|v_{\varepsilon, n}(0)\|_{H^1_{\delta_0}}\Big )\|v_{\varepsilon, n}\|_{L^\infty((0,T);H^2_{-\delta})}^{1^-}\\
&\quad +C(\omega) |\ln \varepsilon|^C  \|v_{\varepsilon, n}(t)|v_{\varepsilon, n}(t)|^p\|_{L^2_\delta} \\
&\leq \|\Delta v_{\varepsilon, n}(t) e^{-A_{\varepsilon, n}}\|_{L^2} + C(\omega) |\ln \varepsilon|^C {\mathcal P} \Big (\|v_{\varepsilon, n}(0)\|_{H^1_{\delta_0}}\Big ) \|v_{\varepsilon, n}\|_{L^\infty((0,T);H^2_{-\delta})}^{1^-}
\\ &\quad + C(\omega) |\ln \varepsilon|^C  {\mathcal P} \Big (\|v_{\varepsilon, n}(0)\|_{H^1_{\delta_0}}\Big )\|v_{\varepsilon, n}\|_{L^\infty((0,T);H^2_{-\delta})}^{\eta_0(p+1)}
\\& \leq \|\Delta v_{\varepsilon, n}(t) e^{-A_{\varepsilon, n}}\|_{L^2}+ C(\omega) |\ln \varepsilon|^C 
{\mathcal P} \Big (\|v_{\varepsilon, n}(0)\|_{H^1_{\delta_0}}\Big)\|v_{\varepsilon, n}\|_{L^\infty((0,T);H^2_{-\delta})}^{1^-}
\end{split}
\end{align}
for a suitable $1^{-}\in (0,1)$.
Hence, we can gather together \eqref{naifra}, \eqref{FAV1}, \eqref{GAV1}, \eqref{HAVpao}, and \eqref{HAVeq} and we get by elementary manipulations that
\begin{align*}
\int_{\R^2} &|\Delta v_{\varepsilon, n}(t)|^2 e^{-2A_{\varepsilon, n}} dx\leq C\int_{\R^2} |\Delta v_{\varepsilon, n}(0)|^2 e^{-2A_{\varepsilon, n}}dx
\\\nonumber&+ C \|v_{\varepsilon, n}\|_{L^\infty ((0,T);H^2_{-\delta})}^{2^-}+C(\omega)|\ln \varepsilon|^C+ {\mathcal P} \Big (\|v_{\varepsilon, n}(0)\|_{H^1_{\delta_0}}\Big),
\end{align*}
which in turn by \eqref{equivH2A} implies
\begin{align*} \|v_{\varepsilon, n} &\|_{L^\infty ((0,T);H^2_{-\delta})}^2\leq {\mathcal P}\Big( \|v_{\varepsilon, n}(0)\|_{H^2_{\delta_0}}\Big) +C \|v_{\varepsilon, n}\|_{L^\infty((0,T); H^2_{-\delta})}^{2^-}+ C(\omega)|\ln \varepsilon|^C.
\end{align*}
By an elementary continuity argument we get \eqref{fundBound}.

\section{Proof of the main result}
\label{sec:thm}
In this section, we first prove Theorem \ref{thmreg}, global well-posedness of the mollified equation \eqref{eq:NLSve}. After that, we prove Theorem \ref{mainthm}, the convergence of the solutions of the mollified problem to a unique solution of \eqref{eq:NLSv}.

\subsection{Proof of Theorem \ref{thmreg}}
We first follow the steps from \cite[Proposition 2.11]{DM} to show the existence of a solution $v_\eps$ to \eqref{eq:NLSve}, where $\eps \in (0, \frac 12)$ is fixed. Fix $\delta_0 > 0$, $T > 0$, and let $v_0 \in H_{\delta_0}^2$. By \eqref{fundBound}, \eqref{fundamnew}, and \eqref{lem:BesovInterpolation}, we have the following bound for any $\gamma \in (1, 2)$ and for some $\delta > 0$:
\begin{align}
\label{Hgbdd}
\sup_{n \in \N} \| v_{\eps, n} \|_{\mathcal{C}([0, T]; H^\gamma_\delta)} \leq C(\omega) |\ln \eps|^C \mathcal{P} \Big( \| v_0 \|_{H^2_{\delta_0}} \Big), \quad \hbox{ a.s. } \omega
\end{align}
where the bound is uniform in $n \in \N$. Also, by \eqref{convexpY}, \eqref{lem:ConvergenceYnew}, and \eqref{exponential}, for any $\alpha \in (0, 1)$, $0 < \delta^- < \delta$, and $\eps \in (0, \frac 12)$ we have the following bounds: 
\begin{align}
\label{bounds}
\sup_{n \in \N} \Big( \| \theta_n Y_\eps \|_{\mathcal{C}^\alpha_{-\delta^-}} + \| \theta_n \widetilde{\Wick{\D Y^2_{\eps}}} \|_{\mathcal{C}^{\alpha - 1}_{- \delta^-}} + \| e^{-p \theta_n Y_\eps} \|_{L^\infty_{-\delta^-}} \Big) \leq C(\omega), \quad \hbox{ a.s. } \omega.
\end{align}

\noindent
Using the equation \eqref{NLSfurthersmoothing}, \eqref{Hgbdd}, \eqref{bounds}, and \eqref{Sobolevweighted}, we can easily deduce that $\{ \partial_t v_{\eps, n} \}_{n \in \N}$ is bounded (uniformly in $n \in \N$) in $\mathcal{C} ([0, T]; H^{\gamma -2}_{\delta'})$ for any $0 < \delta' < \delta$. By the Arzel\`a-Ascoli theorem along with the compact embedding \eqref{inclusionHsdelta}, we obtain a convergent subsequence $\{v_{\eps, n_k}\}_{k \in \N}$ in $\mathcal{C} ([0, T]; H^{\gamma_1 - 2}_{\delta_1})$ for any $\gamma_1 < \gamma$ and $\delta_1 < \delta$, and we denote the limit as $v_\eps$. By \eqref{fundBound} and \eqref{lem:BesovInterpolation}, the convergence also holds in $\mathcal{C} ([0, T]; H^s_{\bar \delta})$ for any $s \in (1, 2)$ and some $\bar \delta > 0$. Also, by \eqref{fundBound}, the Banach-Alaoglu theorem, and taking a further subsequence if necessary, we obtain the following bound:
\begin{align}
\label{fundBound2}
\| v_\eps \|_{L^\infty ((0, T); H^2_{- \tilde \delta})} \leq C(\omega) |\ln \eps|^C \mathcal{P} \Big( \| v_0 \|_{H^2_{\delta_0}} \Big),\quad \hbox{ a.s. } \omega
\end{align} 
for some $\tilde \delta > 0$. Furthermore, $v_\eps$ satisfies the equation \eqref{eq:NLSve}.

Next, we show the uniqueness of $v_\eps$ in $\mathcal{C} ([0, T]; H^s_{\bar \delta})$. Assume $v_{\eps}$ and $w_{\eps}$ are two solutions to \eqref{eq:NLSve}. Define
\begin{align*}
r_\eps (t) = v_\eps (t) - w_\eps (t), \quad t \in [0,T].
\end{align*} 
Then, $r_\eps$ satisfies the equation:
\begin{align*}
\ii \partial_t r_\eps = \Delta r_\eps + r_\eps \widetilde{\Wick{\D Y^2_{\eps}}} - 2 \D r_\eps \D Y_\eps - \lambda e^{-p Y_\eps} (|v_\eps|^p v_\eps - |w_\eps|^p w_\eps), \quad r_\eps (0) = 0.
\end{align*}

\noindent
Using the equation for $r$, we can deduce that
\begin{align*}
\frac 12 \frac{\dd}{\dd t} \int_{\R^2} |r_\eps (t)| e^{-2 Y_\eps} dx = \lambda \Im \int_{\R^2} \bar r_\eps (t) \big( |v_\eps (t)|^p v_\eps(t) - |w_\eps (t)|^p w_\eps(t) \big) e^{-(p + 2) Y_\eps} dx.
\end{align*}
Thus, using the embedding $H^s_{\bar \delta} \subset L_{\bar \delta}^\infty$ and \eqref{exponential}, there exists $\delta > 0$ small enough such that
\begin{align*}
\frac 12 \frac{\dd}{\dd t} \int_{\R^2} C |r_\eps (t)| e^{-2 Y_\eps} dx &\leq \| r_\eps e^{- Y_\eps} \|_{L^2}^2 \Big( \| v_\eps (t) \|_{L_{p \delta}^\infty}^p + \| w_\eps (t) \|_{L_{p \delta}^\infty}^p \Big) \| e^{- Y_\eps} \|_{L_{- \delta}^\infty}^p \\
&\leq C(\omega) \| r_\eps e^{- Y_\eps} \|_{L^2}^2 \Big( \| v_\eps (t) \|_{L_{p \delta}^\infty}^p + \| w_\eps (t) \|_{L_{p \delta}^\infty}^p \Big).
\end{align*}
By the Gronwall inequality, we obtain $r_\eps (t) = 0$, so the uniqueness result follows. This implies that the whole sequence $\{v_{\eps, n}\}_{n \in \N}$ converges to $v_\eps$ in $\mathcal{C} ([0, T]; H^s_{\bar \delta})$.

\subsection{Proof of Theorem \ref{mainthm}}
The scheme of the proof is similar to the previous
works on NLS equation with white noise potential.

We first note that by taking $$A = Y_\eps,\quad V = \tDYe,$$ 
the conditions \eqref{exponentialAV} and \eqref{functionalDMAV} hold almost surely uniformly in $\eps \in (0, \frac 12)$ by using the same reasoning as in the beginning of Section \ref{sec:H2}. In particular, we can use \eqref{nonloin}, \eqref{nonlin}, and \eqref{fundamnew} with $v$ replaced by $v_\eps$.

Fix $\delta_0 > 0$ and assume that $v_0 \in H^2_{\delta_0}$. Let $0 < \eps_2 < \eps_1 < \frac 12$ and define
$$
r(t)=v_{\eps_1}(t)-v_{\eps_2}(t), \quad t\in [0,T].
$$
Then, $r$ satisfies the equation:
\begin{align*}
\ii \partial_t r&=\Delta r+r{\widetilde{\Wick{\D Y^2_{\eps_1}}}}+v_{\eps_2}({\widetilde{\Wick{\D Y^2_{\eps_1}}}}-{\widetilde{\Wick{\D Y^2_{\eps_2}}}}) - 2\D r \D Y_{\eps_1} \\
&\quad -2\D v_{\eps_2}
(\D Y_{\eps_1}-\D Y_{\eps_2}) -\lambda |v_{\eps_1} e^{-Y_{\eps_1}}|^{p} v_{\eps_1}+\lambda |v_{\eps_2} e^{-Y_{\eps_2}}|^{p} v_{\eps_2} ,\\
r(0)&=0.
\end{align*}

\noindent
Using the equation for $r$, we can deduce that
\begin{align}
\label{main}
\begin{split}
\frac12 \frac{\dd}{\dd t} &\int_{\R^2} |r(t)|^2 e^{-2Y_{\eps_1}} d x \\
&=   \Im\int_{\R^2} \bar r(t) v_{\eps_2}(t)({\widetilde{\Wick{\D Y^2_{\eps_1}}}}-{\widetilde{\Wick{\D Y^2_{\eps_2}}}})
e^{-2Y_{\eps_1}} d x \\
&\quad -2 \Im\int_{\R^2} \bar r(t) \D v_{\eps_2}(t)(\D Y_{\eps_1}-\D Y_{\eps_2})e^{-2Y_{\eps_1}} d x \\
&\quad -\lambda \Im\int_{\R^2} \bar r(t) \left( |v_{\eps_1}(t) e^{-Y_{\eps_1}}|^{p } v_{\eps_1}(t) - |v_{\eps_2}(t) e^{-Y_{\eps_2}}|^{p } v_{\eps_2}(t)\right) e^{-2Y_{\eps_1}} d x\\
&=: I + II + III.
\end{split}
\end{align}

Using \eqref{dual}, \eqref{prod1}, \eqref{tY2_conv}, \eqref{exponential}, and interpolation with \eqref{nonloin} and \eqref{nonlin}, we get that for $\alpha \in (0, \frac{1}{100})$, there exists $\delta > 0$ small enough and $\kappa > 0$ such that
\begin{align}
\label{main1}
\begin{split}
|I| &\le C \|\bar r(t) v_{\eps_2}(t) e^{-2Y_{\eps_1}}\|_{\mathcal{B}^{\alpha}_{1,1, \delta}}
 \| {\widetilde{\Wick{\D Y^2_{\eps_1}}}}-{\widetilde{\Wick{\D Y^2_{\eps_2}}}}\|_{\mathcal{C}^{-\alpha}_{- \delta}}\\
 &\le C(\omega) \eps_1^\kappa \| r(t) \|_{H^{2\alpha}_{\delta}} \| v_{\eps_2}(t)\|_{H^{3\alpha}_{\delta}} \|e^{-2Y_{\eps_1}}\|_{\mathcal{C}^{3\alpha}_{- \delta}}\\
 &\le C(\omega)\eps_1^\kappa \mathcal{P} \Big( \|v_0\|_{H^1_{\delta_0}} \Big).
\end{split}
\end{align}

\noindent
Similarly, using \eqref{dual}, \eqref{prod1}, \eqref{lem:ConvergenceYnew}, \eqref{exponential}, \eqref{fundamnew}, and \eqref{fundBound2}, there exists $\delta > 0$ small enough and $\kappa > 0$ such that
\begin{align}
\label{main2}
\begin{split}
|II| &\le C \|\bar r(t) \D v_{\eps_2}(t) e^{-2Y_{\eps_1}}\|_{\mathcal{B}^{\alpha}_{1,1, \delta}}
 \|\D Y_{\eps_1}-\D Y_{\eps_2} \|_{\mathcal{C}^{-\alpha}_{-\delta}}\\
 &\le C(\omega) \eps_1^\kappa \| r(t) \|_{H^{2\alpha}_\delta} \| \D v_{\eps_2}(t) \|_{H^{3\alpha}_\delta} \|e^{-2Y_{\eps_1}}\|_{\mathcal{C}^{3\alpha}_{-\delta}}\\
 &\le C(\omega) \eps_1^\kappa \mathcal{P} \Big( \|v_0\|_{H^1_{\delta_0}} \Big) \| v_{\eps_2}(t) \|_{H^2_{-\delta}}^{C} \\
 &\le C(\omega) \eps_1^\kappa |\ln \eps_2|^{C} \mathcal{P} \Big( \|v_0\|_{H^2_{\delta_0}} \Big).
\end{split}
\end{align}

\noindent
Concerning $III$, using the embedding $H^{\frac{2}{2 - \eta}}_{\tilde \delta} \subset L^\infty_{\tilde \delta}$ ($\eta, \tilde \delta > 0$ small enough), \eqref{exponential}, \eqref{nonloin}, \eqref{exp_conv}, \eqref{fundamnew}, and \eqref{fundBound2}, there exists $\delta > 0$ small enough and $\kappa > 0$ such that
\begin{align}
\label{main3}
\begin{split}
|III| &\le  C \|r e^{-Y_{\eps_1}} \|_{L^2}^2 \left( \|v_{\eps_1}(t)\|_{L^\infty_{p \delta}}^{p}+\|v_{\eps_2}(t)\|_{L^\infty_{p \delta}}^{p}\right)\|e^{-Y_{\eps_1}}\|_{L^\infty_{-\delta}}^{p}\\
&\quad + \|r\|_{L^2_\delta} \|v_{\eps_2}(t)\|_{L^{2p + 2}_\delta}^{p+1} \|e^{-p Y_{\eps_1}}-
e^{-p Y_{\eps_2}}\|_{L^\infty_{- p\delta}} \| e^{- Y_{\eps_1}} \|_{L^\infty_{-\delta}} \\
&\le C(\omega) \|r e^{-Y_{\eps_1}} \|_{L^2}^2 \Big( \|v_{\eps_1}(t)\|_{H^{\frac{2}{2 - \eta}}_{p \delta}}^{p}+\|v_{\eps_2}(t)\|_{H^{\frac{2}{2 - \eta}}_{p \delta}}^{p} \Big) \\
&\quad + C(\omega) \eps_1^\kappa \| v_{\eps_2}(t) \|_{H_\delta^{\frac{2}{2 - \eta}}} \mathcal{P} \Big( \|v_0\|_{H^1_{\delta_0}} \Big) \\
&\le C(\omega) |\ln \eps_2|^{\gamma} \mathcal{P} \Big( \|v_0\|_{H^2_{\delta_0}} \Big) \|r e^{-Y_{\eps_1}} \|_{L^2}^2  +C(\omega)\eps_1^\kappa |\ln \eps_2|^{C} \mathcal{P} \Big( \|v_0\|_{H^2_{\delta_0}} \Big),
\end{split}
\end{align}

\noindent
for some $\gamma \in (0, 1)$.

Now, letting $\eps_1 = 2^{-k}$ and $\eps_2 = 2^{- (k+1)}$ for $k \in \N$, we combine \eqref{main}, \eqref{main1}, \eqref{main2}, and \eqref{main3} and apply the Gronwall inequality to obtain
\begin{align*}
\sup_{t \in [0, T)} &\int_{\R^2} |r (t)|^2 e^{- 2 Y_{2^{-k}}} dx \\
&\leq C(\omega) 2^{-k \kappa} (k+1)^C \mathcal{P} \Big( \|v_0\|_{H^2_{\delta_0}} \Big) e^{ C(\omega) T |\ln 2^{- (k + 1)}|^\gamma \mathcal{P} ( \|v_0\|_{H^2_{\delta_0}} )} \\
&\leq C(\omega) 2^{-\frac{k \kappa}{2}} 2^{\tilde{\gamma} (k+1)C(\omega) T \mathcal{P} ( \|v_0\|_{H^2_{\delta_0}} )} \mathcal{P} \Big( \|v_0\|_{H^2_{\delta_0}} \Big),
\end{align*}

\noindent
where we used $e^{|\ln (1 + x)|^\gamma} \leq C x^{\tilde{\gamma}}$ for $\gamma \in (0, 1)$, $\tilde{\gamma} > 0$ arbitrarily small, and $x > 0$ large. Thus, by \eqref{exponential}, for any $\delta > 0$ we obtain
\begin{align*}
\| v_{2^{-k}} - v_{2^{- (k+1)}} \|_{\mathcal{C}([0, T); L^2_{- \delta})}^2 \leq C(\omega) 2^{-\frac{k \kappa}{4}}.
\end{align*}

\noindent
Using interpolation along with the bounds \eqref{fundamnew} and \eqref{fundBound2}, it is not hard to deduce that for any $s \in (1, 2)$, there exists $\bar \delta > 0$ such that $\{ v_{2^{-k}} \}_{k \in \N}$ is a Cauchy sequence in $\mathcal{C}([0, T); H^s_{\bar \delta})$ and converges to some function $v \in \mathcal{C}([0, T); H^s_{\bar \delta})$. By using similar steps as above, we can also deduce that
\begin{align*}
\sup_{\eps \in (2^{-(k+1)}, 2^{-k}]} \| v_\eps - v_{2^{-k}} \|_{\mathcal{C} ([0, T); H^s_{\bar \delta})} \leq C(\omega) 2^{-k \tilde \kappa}  \mathcal{P} \Big( \|v_0\|_{H^2_{\delta_0}} \Big),
\end{align*}
for some $\tilde \kappa > 0$, so that the whole sequence $\{ v_\eps \}_{\eps \in (0, \frac 12)}$ converges to $v$ in $\mathcal{C} ([0, T); H^s_{\bar \delta})$ as $\eps \to 0$. This finishes the convergence part of the theorem. The uniqueness of the solution $v$ to the equation \eqref{eq:NLSv} follows from a similar (and easier) manner.

\begin{ackno} \rm
A.D. was supported by the ANR grant ``ADA'' and the French government ``Investissernents d'Avenir'' program ANR-11-LABX-0020-01. R.L. thanks Tadahiro Oh, Guangqu Zheng, and Younes Zing for helpful conversations. R.L. was supported by the European Research Council (grant no. 864138 ``SingStochDispDyn''). N.T. was partially supported by the ANR projet Smooth ``ANR-22-CE40-0017''. N.V. was supported by PRIN 2020XB3EFL from MIUR and PRA\_2022\_11 from University of Pisa.
\end{ackno}

\end{document}